\documentclass[12pt]{article}
\PassOptionsToPackage{hyphens}{url}
\usepackage[pagebackref=true,colorlinks]{hyperref}
\usepackage{amsmath,amssymb,amsthm,mathtools,amsrefs,mathrsfs}
\usepackage{pifont}
\usepackage{bbm}
\usepackage{tikz}
\usepackage{adjustbox}
\usepackage{stmaryrd}
\usepackage{fancyhdr}
\usepackage{soul}
\usepackage{dsfont}
\usepackage[normalem]{ulem}
\usepackage[bottom]{footmisc}
\usepackage{enumerate}
\usepackage[all]{xy} 
\usetikzlibrary{circuits.ee.IEC}
\hypersetup{
    colorlinks,
    linkcolor={red!50!black},
    citecolor={blue!50!black},
    urlcolor={blue!80!black}
}

\usepackage{iftex}

\ifpdf
  \usepackage[T1]{fontenc}        
\else
  \usepackage{breakurl}           
\fi

\usepackage{tocbasic}
\DeclareTOCStyleEntry[
  beforeskip=0.3em plus 1pt,
  pagenumberformat=\textbf
]{tocline}{section}

\makeatletter
\newcommand*{\shifttext}[2]{%
  \settowidth{\@tempdima}{#2}%
  \makebox[\@tempdima]{\hspace*{#1}#2}%
}
\makeatother
\makeatletter
\renewcommand*\env@matrix[1][\arraystretch]{%
  \edef\arraystretch{#1}%
  \hskip -\arraycolsep
  \let\@ifnextchar\new@ifnextchar
  \array{*\c@MaxMatrixCols c}}
\makeatother

\theoremstyle{plain}
\newtheorem{theorem}[equation]{Theorem}
\newtheorem{lemma}[equation]{Lemma}
\newtheorem{proposition}[equation]{Proposition}
\newtheorem{corollary}[equation]{Corollary}
\theoremstyle{definition}
\newtheorem{definition}[equation]{Definition}
\newtheorem{construction}[equation]{Construction}
\newtheorem{question}[equation]{Question}

\newtheorem{conjecture}[equation]{Conjecture}
\newtheorem{problem}[equation]{Problem}
\newtheorem{example}[equation]{Example}
\newtheorem{exercise}[equation]{Exercise}
\newtheorem*{answer}{Answer}
\newtheorem*{solution}{Solution}
\newtheorem{remark}[equation]{Remark}

\newtheorem{notation}[equation]{Notation}
\newtheorem{noterm}[equation]{Notation and Terminology}

\newcommand\define[1]{\emph{\textbf{#1}}}

\numberwithin{equation}{section}

\let\C=\Chi

\newcommand{\be}{\begin{equation}}
\newcommand{\ee}{\end{equation}}
\def\ba{\begin{align}} 
\def\ea{\end{align}}
\newcommand{\bea}{\begin{eqnarray}}
\newcommand{\eea}{\end{eqnarray}}
\newcommand{\bx}{\begin{example}}
\newcommand{\ex}{\end{example}}
\newcommand{\bex}{\begin{exercise}}
\newcommand{\eex}{\end{exercise}}
\newcommand{\ban}{\begin{answer}}
\newcommand{\ean}{\end{answer}}
\newcommand{\bt}{\begin{theorem}}
\newcommand{\et}{\end{theorem}}
\newcommand{\bc}{\begin{corollary}}
\newcommand{\ec}{\end{corollary}}
\newcommand{\blem}{\begin{lemma}}
\newcommand{\elem}{\end{lemma}}
\newcommand{\bp}{\begin{problem}}
\newcommand{\ep}{\end{problem}}
\newcommand{\bn}{\begin{proposition}}
\newcommand{\en}{\end{proposition}}
\newcommand{\bd}{\begin{definition}}
\newcommand{\ed}{\end{definition}}
\newcommand{\bcon}{\begin{construction}}
\newcommand{\econ}{\end{construction}}
\newcommand{\bq}{\begin{question}}
\newcommand{\eq}{\end{question}}
\newcommand{\bprf}{\begin{proof}}
\newcommand{\eprf}{\end{proof}}
\newcommand{\br}{\begin{remark}}
\newcommand{\er}{\end{remark}}
\newcommand{\bs}{\begin{solution}}
\newcommand{\es}{\end{solution}}
\newcommand{\beqs}{\begin{eqnarray}}
\newcommand{\eeqs}{\end{eqnarray}}
\newcommand{\bnt}{\begin{noterm}}
\newcommand{\ent}{\end{noterm}}
\newcommand{\bnot}{\begin{notation}}
\newcommand{\enot}{\end{notation}}

\newcommand{\<}{\langle}
\renewcommand{\>}{\rangle}

\newcommand{\id}{\mathrm{id}}

\newcommand{\tr}{\operatorname{tr}}
\newcommand{\Tr}{\operatorname{Tr}}

\def\R{{{\mathbb R}}}
\def\C{{{\mathbb C}}}

\def\N{{{\mathbb N}}}

\def\Re{{{\mathfrak{Re}}}}
\def\Im{{{\mathfrak{Im}}}}

\newcommand\aeequals[1]{\underset{\raisebox{0.3ex}[0pt][0pt]{\scriptsize${#1}$}}{=}}

\newcommand{\Ad}{\mathrm{Ad}}

\newcommand{\stoch}{\;\xy0;/r.25pc/:(-3,0)*{}="1";(3,0)*{}="2";{\ar@{~>}"1";"2"|(1.06){\hole}};\endxy\!}

\DeclareFontFamily{OT1}{pzc}{}
\DeclareFontShape{OT1}{pzc}{m}{it}{ <-> s*[1.2] pzcmi7t }{}
\DeclareMathAlphabet{\mathpzc}{OT1}{pzc}{m}{it}
\newcommand{\Alg}[1]{\mathpzc{#1}}

\newcommand{\matr}{\mathbb{M}}

\newcommand{\op}{\mathrm{op}}

\newcommand{\esssup}{\mathrm{ess}\,\mathrm{sup}}

\newcommand{\ben}{\renewcommand{\theenumi}{\alph{enumi}} 
\renewcommand{\labelenumi}{(\theenumi)}\begin{enumerate}}
\newcommand{\een}{\end{enumerate}}

\newcommand\blfootnote[1]{%
  \begingroup
  \renewcommand\thefootnote{}\footnote{#1}%
  \addtocounter{footnote}{-1}%
  \endgroup
}

\newcommand{\bsm}{\begin{smallmatrix}}
\newcommand{\esm}{\end{smallmatrix}}
\newcommand{\E}{\mathcal{E}}

\title{Bayesian inference and retrodiction for faithful states on von~Neumann algebras}
\author{Pradyut Karmakar and Arthur J.~Parzygnat}
\date{August 20, 2026}

\newcommand{\Addresses}{{
 \bigskip
 \footnotesize

   P.~Karmakar, \textsc{%
  Sam Houston State University, 332 A LDB, 1900 Avenue I, Huntsville, Texas, USA}\par\nopagebreak
 \textit{E-mail address}, 
P.~Karmakar: \texttt {karmakar.pradyut@gmail.com}\\

   A.~J.~Parzygnat, \textsc{%
 The Experimental Study Group, 
Massachusetts Institute of Technology, 
Cambridge, Massachusetts 02139, USA}\par\nopagebreak
 \textit{E-mail address}, A.~J.~Parzygnat: \texttt{arthurjp@mit.edu}
}}

\begin{document}
\emergencystretch 2em

\maketitle

\begin{abstract}  
Retrodiction is the act of inferring a cause from its effects, the most common example of which is Bayesian inference. Retrodiction can be defined by its structural process-theoretic properties, which are mathematically captured by category theory. This categorical definition of retrodiction has recently been shown to potentially isolate the Petz recovery map as a unique universal candidate for quantum Bayesian inference. This paper extends these results to the infinite-dimensional setting on von~Neumann algebras. In the process, we provide a pedagogical review of the Petz recovery map in infinite dimensions and its relation to the more commonly used expression in the finite-dimensional setting. We formalize the open question as to whether these categorical axioms for retrodiction do in fact uniquely characterize the Petz recovery map. If such a characterization holds, this would show that Bayesian inversion and the Petz recovery map are structural necessities and not simply useful algorithms for classical and quantum inference. 


\end{abstract}

\vspace{-7mm}
\tableofcontents

\section{Introduction and main results}
\label{sec:intro}

Bayes' rule~\cite{Bayes1763} can be interpreted as the statement that associated with every nowhere-vanishing probability $\{\mathbb{P}(x)\}$ on a finite set $X$ and a set of conditional probabilities $\{\mathbb{P}(y|x)\}$ (read ``the probability of $y$ given $x$'') involving another finite set $Y$, there exists a set of conditional probabilities $\{\mathbb{P}(x|y)\}$, called the \emph{Bayesian inverse} of $\{\mathbb{P}(y|x)\}$, such that $\mathbb{P}(y|x)\mathbb{P}(x)=\mathbb{P}(x|y)\mathbb{P}(y)$ for all $x\in X,y\in Y$, with uniqueness guaranteed whenever $\mathbb{P}(y):=\sum_{x\in X}\mathbb{P}(y|x)\mathbb{P}(x)$ is nowhere vanishing. One of its key structural properties is \emph{functoriality}, which can be described as follows. Given another set of conditional probabilities $\{\mathbb{P}(z|y)\}$ involving a finite set $Z$, Bayes' rule provides a set of conditional probabilities $\{\mathbb{P}(y|z)\}$ satisfying $\mathbb{P}(z|y)\mathbb{P}(y)=\mathbb{P}(y|z)\mathbb{P}(z)$ for all $y\in Y,z\in Z$, where $\mathbb{P}(z):=\sum_{y\in Y}\mathbb{P}(z|y)\mathbb{P}(y)$ (again, assumed to be nowhere vanishing). On the other hand, one also has conditional probabilities $\mathbb{P}(z|x):=\sum_{y\in Y}\mathbb{P}(z|y)\mathbb{P}(y|x)$. Applying Bayes' rule to this set of conditional probabilities yields $\{\mathbb{P}(x|z)\}$ satisfying $\mathbb{P}(z|x)\mathbb{P}(x)=\mathbb{P}(x|z)\mathbb{P}(z)$ for all $x\in X, z\in Z$. Functoriality then says that these three sets of conditional probabilities obtained via Bayes' rule are related via $\mathbb{P}(x|z)=\sum_{y\in Y}\mathbb{P}(x|y)\mathbb{P}(y|z)$. In diagrammatic form, this can be summarized by turning one commutative diagram into another by the process of Bayesian inversion we just described%
\[
\xy0;/r.25pc/:
(-15,-7.5)*+{\{ \mathbb{P}(x) \}}="X";
(0,7.5)*+{\{ \mathbb{P}(y) \}}="Y";
(15,-7.5)*+{\{ \mathbb{P}(z) \}}="Z";
{\ar"X";"Y"^{\{\mathbb{P}(y|x)\}}};
{\ar"Y";"Z"^{\{\mathbb{P}(z|y)\}}};
{\ar"X";"Z"_{\{\mathbb{P}(z|x)\}}};
\endxy
\xmapsto{\text{Bayesian inversion}}
\xy0;/r.25pc/:
(-15,-7.5)*+{\{ \mathbb{P}(x) \}}="X";
(0,7.5)*+{\{ \mathbb{P}(y) \}}="Y";
(15,-7.5)*+{\{ \mathbb{P}(z) \}}="Z";
{\ar"Y";"X"_{\{\mathbb{P}(x|y)\}}};
{\ar"Z";"Y"_{\{\mathbb{P}(y|z)\}}};
{\ar"Z";"X"^{\{\mathbb{P}(x|z)\}}};
\endxy
.
\]
Besides functoriality, Bayesian inversion satisfies many other natural properties including (1) respecting product structure for combining independent systems, (2) specializing to ordinary inversion when the conditional probabilities describe a deterministic and reversible correlation, and (3) being involutive, which means that the Bayesian inverse operation applied twice yields the original set of conditional probabilities. The goals of this paper are to make these observations more precise, extend Bayesian inversion to the setting of infinite-dimensional classical and quantum systems using the theory of von~Neumann algebras and the Petz recovery map~\cite{Pe84}, and provide a conceptual synthesis of Bayesian inversion, the Petz map, and retrodiction using categorical concepts.  

In the context of von~Neumann algebras, in Ref.~\cite{Pe84}, Petz formulated the notion of an adjoint (called ``dual'' in Ref.~\cite{Pe84} and ``transpose channel'' in Ref.~\cite{OhPe93}) of a positive contraction between von~Neumann algebras equipped with faithful states~\cites{Dix81,Tak1}. This notion simultaneously extends the notions of an inverse, Hilbert--Schmidt adjoint, which takes a quantum channel acting on states (Schr\"odinger picture) to a quantum operation acting on observables (Heisenberg picture), and also a conditional expectation~\cites{Mo54,Um54}. In the special case of full matrix algebras, or more generally type I factors, the Hilbert--Schmidt adjoint of Petz' map reproduces what is now known as the Petz recovery map~\cites{Wilde15,WildeQIT16,JRSWW16,JRSWW18}. 

Although it has often been claimed that the Petz recovery map can be viewed as a form of retrodiction or as a quantum extension of Bayes' rule~\cites{LiWi18,Wilde15,WildeQIT16,AwBuSc21}, it was not until recently that this analogy has been made more structurally precise through a variety of perspectives including category theory~\cites{PaBu22,FuPa22a,Pa24} and information geometry~\cites{SASDS23,Ts22,Ts22b}. Namely, for full matrix algebras and faithful states, the Petz recovery map has been shown to satisfy a list of functoriality axioms~\cite{PaBu22}, singling it out among some alternatives in the quantum setting. These results were then extended to nonfaithful states on finite-dimensional $C^*$-algebras in Ref.~\cite{Pa24}. 

In this paper, we address the categorical formulation of the Petz recovery map in the infinite-dimensional setting for faithful states on von~Neumann algebras. Namely, we show that there is a retrodiction functor~\cite{Pa24} (an inverting symmetric monoidal dagger functor) on the category of von~Neumann algebras equipped with normal faithful states as objects and state-preserving normal completely positive and unital maps as morphisms. This dagger functor is precisely the assignment that sends a morphism to its associated Petz recovery map. Interestingly, the connection between the Petz recovery map and category theory was already begun by Petz himself in Ref.~\cite{Pe84}. However, at that time, it seems as though the connection to Bayes' rule, Bayesian inversion, and retrodiction were not made. In this paper, we bring this connection out and highlight how the Petz recovery map in the finite- and infinite-dimensional settings can be expressed purely in categorical terms. Figure~\ref{fig:diagramfunctors} summarizes our contributions. 

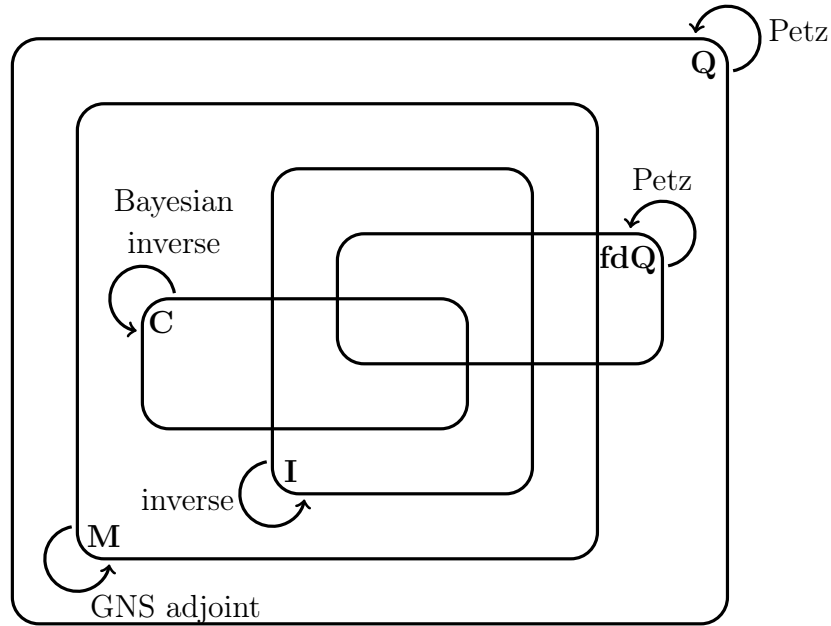
\begin{figure}[htb]
\centering
\begin{tikzpicture}[scale=0.86]
\draw[rounded corners=10pt,very thick] (-3,1) node[below,xshift=7pt]{$\mathbf{C}$} rectangle (2,-1);
\draw[rounded corners=10pt,very thick] (0,0) rectangle (5,2) node[below,xshift=-13pt]{$\mathbf{fdQ}$};
\draw[rounded corners=10pt,very thick] (-1,-2) node[above,xshift=7pt]{$\mathbf{I}$} rectangle (3,3);
\draw[rounded corners=10pt,very thick] (-4,-3) node[above,xshift=10pt]{$\mathbf{M}$} rectangle (4,4); 
\draw[rounded corners=10pt,very thick] (-5,-4) rectangle (6,5) node[below,xshift=-9pt]{$\mathbf{Q}$}; 
\draw[->,very thick] ({-1+0.5*cos(100)},{-2+0.5*sin(100)}) arc[start angle=100, end angle=350, radius=0.5cm] node[left,xshift=-22pt]{inverse};
\draw[->,very thick] ({-3+0.5*cos(10)},{1+0.5*sin(10)}) arc[start angle=10, end angle=260, radius=0.5cm] node[above,yshift=20pt,xshift=14pt]{\begin{tabular}{c}Bayesian\\ inverse\end{tabular}};
\draw[->,very thick] ({-4+0.5*cos(100)},{-3+0.5*sin(100)}) arc[start angle=100, end angle=350, radius=0.5cm] node[below,yshift=-7pt,xshift=25pt]{GNS adjoint};
\draw[->,very thick] ({5+0.5*cos(-80)},{2+0.5*sin(-80)}) arc[start angle=-80, end angle=170, radius=0.5cm] node[above,yshift=10pt,xshift=12pt]{Petz};
\draw[->,very thick] ({6+0.5*cos(-80)},{5+0.5*sin(-80)}) arc[start angle=-80, end angle=170, radius=0.5cm] node[right,yshift=1pt,xshift=23pt]{Petz};
\end{tikzpicture}
\caption{A diagram of the categories used in this paper and endofunctors on them. All of these categories are defined in the paper; briefly, $\mathbf{C}$ is a category of classical systems and stochastic (discrete-time) evolution, $\mathbf{Q}$ is a category of quantum systems and quantum channels, $\mathbf{I}$ is a category of invertible dynamics (morphisms are isomorphisms), $\mathbf{fdQ}$ is a category of finite-dimensional quantum systems and quantum channels, and $\mathbf{M}$ is a category of Markov maps (certain covariant evolutions). Each containment signifies a subcategory and the endofunctor specializes to the one provided on that subcategory. For example, the endofunctor labelled as ``Petz'' on $\mathbf{Q}$, which is the Petz retrodiction functor, restricts to the classical ``Bayesian inverse'' endofunctor on $\mathbf{C}$. The ``inverse'' endofunctor on the category $\mathbf{I}$ of isomorphisms sends a morphism to its inverse. When restricted to the subcategory $\mathbf{M}$ consisting of Markov maps, Petz retrodiction coincides with the GNS adjoint. This diagram extends the diagram from Ref.~\cite{PaBu22}, which dealt with the subcategory $\mathbf{fdQ}$ of finite-dimensional von~Neumann algebras.}
\label{fig:diagramfunctors}
\end{figure}

It remains an interesting open question whether these functoriality axioms uniquely characterize the Petz recovery map. If these axioms do characterize the Petz recovery map, one wonders how the formula comes about in the classical and quantum settings from the simple algebraic axioms. In particular, how does the square-root formula in the finite-dimensional case arise? Alternatively, if these axioms do not characterize the Petz recovery map, what other retrodiction functors are admissible? This has important practical and philosophical implications for our general understanding of inference, since the axioms proposed are axioms motivated by our intuition on what properties inference and retrodiction should satisfy. 

Let us briefly explain some physical motivation behind this work. Qubit (and more generally qudit) systems~\cites{Schumacher1995,NiCh11} and bounded operators on Hilbert space associated with the quantum mechanics of a finite number of particles~\cite{Sakurai20} are, to a large extent, described by the set of all bounded operators%
\footnote{Technically, unbounded operators already appear in the description of a single particle in space that has infinite extent~\cites{HallQuantum13,Schechter2002,ReedSimonI1980}.}
on Hilbert space or certain superselection sectors singled out by symmetry~\cite{BaRuSp2007}. 
However, when one is dealing with an infinite number of degrees of freedom, such as in infinite lattice systems and the local algebras associated with spacetime regions in quantum field theory (indeed, quantum field theory is often introduced in terms of an infinite number of harmonic oscillators~\cite[Section 1.2]{Weinberg1995QFT1}), von~Neumann algebras beyond this class become essential in a more accurate mathematical treatment~\cites{Yngvason2005,Witten18,Nair2020}. Although many concepts from quantum information theory have been incorporated into relativistic quantum field theory with great success and providing enormous insight to field theory and quantum gravity~\cites{PeresTerno2004,FewsterVerch2020quantum,OhPe93,Ha76,Page1993IBHR,RyuTakayanagi2006}, there are still several unexplored directions. In particular, Bayesian inference and retrodiction, though ubiquitously important in classical probability and statistics~\cite{Pearl88}, deserves a proper treatment in the setting of von~Neumann algebras. This is what we accomplish in this paper, focusing on the particular quantum Bayes' rule from Ref.~\cite{FuPa22a} associated with the Petz recovery map. We leave the extension of the other Bayes' rules and their associated categorical descriptions to future work.  

The outline of this paper is as follows. Section~\ref{sec:KMSinnerproduct} reviews an inner product associated with every von~Neumann algebra equipped with a faithful normal state. Section~\ref{sec:PetzMap} uses this inner product to construct a formal adjoint to a state-preserving normal completely positive unital (NCPU) map (a quantum channel in the Heisenberg picture), which is now commonly referred to as the Petz recovery map. Section~\ref{sec:retrodictionfunctor} proves the categorical properties satisfied by the Petz recovery map, which in turn proves that it defines a retrodiction functor, a concept which is reviewed. Section~\ref{sec:commutativevNAlg} specializes the Petz recovery map to commutative von~Neumann algebras and illustrates the connection to classical Bayesian inference. 
Section~\ref{sec:MarkovMaps} analyzes the case where the Petz recovery map is applied to Markov maps, which satisfy a certain modular covariance symmetry and are particular examples of covariant quantum operations. 
Section~\ref{sec:conclusion} contains a summary, discussion, and outlook, including the precise conjecture stating that Bayesian inference and the Petz recovery map are characterized by process-theoretic axioms. 
Appendix~\ref{sec:reviewModularTheory} contains a review of concepts from functional analysis, von~Neumann algebras, and modular theory. Most terminology used in this paper is defined in the appendix.

\vspace{3mm}
\textbf{Acknowledgements.}
The authors thank Rudrajit Banerjee, Marcel Bischoff, Simone Del Vecchio, Francesco Fidaleo, Tobias Fritz, Yasuyuki Kawahigashi, Brent Nelson, Jonathan Sorce, and D{\'a}niel Virosztek for discussions. The authors thank the organizers of the Great Lakes Mathematical Physics Meeting (GLaMP) 2025 and the University of Kentucky at Lexington, KY, for their hospitality, since the discussions leading to this paper began at that conference.

\section{The KMS inner product}
\label{sec:KMSinnerproduct}

Before defining the Petz recovery map, we must first describe the inner product with respect to which it is defined as a formal adjoint. The goal of this section is to define this inner product and describe some of its properties that will in turn be useful for proving the functoriality properties of the Petz recovery map. 
Let $\Alg{A}$ be a von~Neumann algebra and let $\omega:\Alg{A}\to\C$ be a faithful normal state on $\Alg{A}$. Then there exists a Hilbert space $\mathcal{H}$, an injective normal unital $*$-homomorphism $\pi:\Alg{A}\to\mathcal{B}(\mathcal{H})$, and a unit vector $\Omega\in\mathcal{H}$ such that $\omega(a)=\<\Omega,\pi(a)\,\Omega\>$ for all $a\in\Alg{A}$ and such that $\Omega$ is cyclic and separating for $\Alg{M}:=\pi(\Alg{A})\subseteq\mathcal{B}(\mathcal{H})$ (an example is the GNS representation reviewed in Appendix~\ref{sec:reviewModularTheory}). When $\pi$ is clear from context, $\pi(a)\,\Omega$ may occasionally be written as $x\,\Omega$, where $x=\pi(a)\in\Alg{M}$, to avoid heavy notation. There is no ambiguity with such shorthand notation because $\pi(a_1 a_2^*)\,\Omega=\pi(a_1)\pi(a_2)^*\,\Omega=x_1 x_2^*\,\Omega$, where $x_j=\pi(a_j)$, and $\pi(1_{\Alg{A}})\,\Omega=1_{\Alg{M}}\,\Omega=\id_{\mathcal{H}}\,\Omega=\Omega$ by the $*$-homomorphism property of $\pi$, and $\pi(a)\,\Omega=x\,\Omega=0$ implies $x=0$ by the separating assumption on $\Omega$, which implies $a=0$ by the injectivity of $\pi$. That said, we do distinguish between $\Alg{A}$ and $\Alg{M}$ since they may have different interpretations in the physical context. 
Set 
\be
\label{eqn:ConnesSelfPolarForm}
\Alg{A}\times\Alg{A}\ni(a_1,a_2)\mapsto\<\!\<a_1,a_2\>\!\>_{\omega}:=\big\<\pi(a_1)\,\Omega,J\, \pi(a_2^*)\,\Omega\big\>\equiv\<x_1\,\Omega,J x_2^*\,\Omega\>,
\ee
where $x_j=\pi(a_j)$, $J$ is the modular conjugation operator on $\mathcal{H}$ associated with $(\Alg{M},\Omega)$ and $\<\;\cdot\;,\;\cdot\;\>$ is the inner product on $\mathcal{H}$. Note that~\eqref{eqn:ConnesSelfPolarForm} is a \define{sesquilinear form} on $\Alg{A}$, i.e., $\<\!\<\;\cdot\;,\;\cdot\;\>\!\>_{\omega}:\Alg{A}\times\Alg{A}\to\C$ defines a function that is conjugate-linear in the first variable and linear in the second variable. The form~\eqref{eqn:ConnesSelfPolarForm} was introduced by Connes in Ref.~\cite{Connes74}, where it was called a \emph{forme autopolaire}, and further discussed in Ref.~\cite{Woronowicz1974}, where it was called a \emph{selfpolar form} (see also Ref.~\cite{Pe84} and Equation 8.19 in Ref.~\cite{OhPe93}). Today, it is more commonly called the \emph{KMS inner product}~\cite{CaMa17} after Kubo, Martin, and Schwinger~\cites{Kubo1957,MartinSchwingerI1959} (the connection between the selfpolar form and the KMS condition will be made briefly in Remark~\ref{rmk:KMS}). 

\blem
\label{lem:ConnesInnerProduct}
Using the notation in the previous paragraph, the selfpolar form~\eqref{eqn:ConnesSelfPolarForm} defines an inner product on $\Alg{A}$. Moreover, $\<\!\<\;\cdot\;,\;\cdot\;\>\!\>_{\omega}$ is independent of the representation $\pi$ on $\mathcal{H}$ and the cyclic and separating vector $\Omega\in\mathcal{H}$ in the sense that if $\tau:\Alg{A}\to\mathcal{B}(\mathcal{K})$ is another injective unital $*$-homomorphism and $\Xi\in\mathcal{K}$ is a vector such that $\omega(a)=\<\Xi,\tau(a)\,\Xi\>$ for all $a\in\Alg{A}$ and such that $\Xi$ is cyclic and separating for $\tau(\Alg{A})$, then $\big\<\pi(a_1)\,\Omega,J_{\Omega}\,\pi(a_2^*)\,\Omega\big\>=\big\<\tau(a_1)\,\Xi,J_{\Xi}\,\tau(a_2^*)\,\Xi\big\>$ for all $a_1,a_2\in\Alg{A}$, where $J_{\Omega}$ and $J_{\Xi}$ are the modular operators associated with $(\pi(\Alg{A}),\Omega)$ and $(\tau(\Alg{A}),\Xi)$, respectively. 
\elem
 
Note that the claim using our shorthand notation would read $\<x_1\,\Omega,J_{\Omega}x_2^*\,\Omega\>=\<x_1\,\Xi,J_{\Xi}x_2^*\,\Xi\>$, but where $x_j=\pi(a_j)$ in the left expression and $x_j=\tau(a_j)$ in the right expression. 
This abuse of notation will not last long, but will occur briefly in the proof.

\bprf[Proof of Lemma~\ref{lem:ConnesInnerProduct}]
The fact that $\<\!\<\;\cdot\;,\;\cdot\;\>\!\>_{\omega}$ is conjugate-linear in the first variable and linear in the second variable follows from the fact that the inner product $\<\;\cdot\;,\;\cdot\;\>$ on $\mathcal{H}$ satisfies these conditions and because $J$ is conjugate-linear. Therefore, to prove that $\<\!\<\;\cdot\;,\;\cdot\;\>\!\>_{\omega}$ defines an inner product, we must check positivity and nondegeneracy. 
For positivity, we have 
\begin{align*}
    \<\!\<a, a\>\!\>_{\omega}
    &=\< x\, \Omega, J x^*\, \Omega \> & &\mbox{by definition~\eqref{eqn:ConnesSelfPolarForm} and $x:=\pi(a)$}\\
    &=\< x\, \Omega, J S x\, \Omega \> & & \mbox{by definition of $S$} \\
    &=\< x\,\Omega, \Delta^{1/2} x\, \Omega \> && \mbox{since $S=J\Delta^{1/2}$ and $J^2=\id_{\mathcal{H}}$} \\
    &=\< \Delta^{1/4} x\, \Omega, \Delta^{1/4} x\, \Omega \> &&\mbox{since $\Delta^{1/2}$ is a self-adjoint positive operator on $\mathcal{H}$}\\
    &=\big\|\Delta^{1/4} x\, \Omega\big\|^2\ge0
\end{align*}
for all $a\in\Alg{A}$. Note that in the above calculations, the operators $S$, $J$, and $\Delta$ all acted on elements in their domains. This proves positivity. To prove nondegeneracy of $\<\!\<\,\cdot\,, \,\cdot\,\>\!\>_{\omega}$, suppose that $a\in\Alg{A}$ satisfies $\<\!\<a, a\>\!\>_{\omega}=0$. Then, by the above calculation, $\big\|\Delta^{1/4} x\, \Omega\big\|=0$ as well. This is equivalent to $\Delta^{1/4} x\, \Omega=0$. However, since $\Delta$ is positive (hence invertible), this is equivalent to $x\, \Omega=0$. Since $\Omega$ is separating, this implies $x=0$. Since $\pi$ is injective and $x=\pi(a)$, this implies $a=0$. This proves nondegeneracy. Therefore, $\<\!\<\;\cdot\;,\;\cdot\;\>\!\>_{\omega}$ defines an inner product on $\Alg{A}$.  

We next show that this inner product does not depend on the representation and cyclic and separating vector. This proof is essentially the same as the proof of the equivalence between two standard forms of von~Neumann algebras as in~\cite[Theorem~2.3]{Ha1975}, and it is also similar to the proofs of equivalence between GNS representations~\cites{GN43,PaGNS,Se47}. Suppose $\mathcal{K},\tau$, and $\Xi$ are as in the statement. Define an operator $U$ from $\mathcal{H}$ to $\mathcal{K}$ with domain $\pi(\Alg{A})\,\Omega\subseteq\mathcal{H}$ given by $U\big(\pi(a)\,\Omega\big):=\tau(a)\,\Xi$ for all $a\in\Alg{A}$, which in shorthand notation reads $Ux\,\Omega=x\,\Xi$ (the vector to the right of $x$ indicates the representation). The operator $U$ is well-defined; 
indeed, if $x_1\,\Omega=x_2\,\Omega$, then $x_1=x_2$ since $\Omega$ is separating. Moreover, the operator $U$ is bounded of operator norm $\lVert U\rVert=1$ on its domain because 
\begin{align*}
    \lVert U x\,\Omega\rVert^2
    &=\lVert x\,\Xi\rVert^2 &&\mbox{by definition of $U$}\\
    &=\<x\,\Xi,x\,\Xi\> &&\mbox{by the relationship between norms and inner products} \\
    &=\<\Xi,x^*x\,\Xi\> &&\mbox{since $x$ acts as a bounded operator on $\mathcal{K}$}\\
    &=\omega(x^*x) &&\mbox{since $\Xi$ is a vector state representation for $\omega$ on $\mathcal{K}$}\\
    &=\<\Omega,x^*x\,\Omega\> && \mbox{since $\Omega$ is a vector state representation for $\omega$ on $\mathcal{H}$}\\
    &=\<x\,\Omega,x\,\Omega\> &&\mbox{since $x$ acts as a bounded operator on $\mathcal{H}$}\\
    &=\lVert x\,\Omega\rVert^2 &&\mbox{by the relationship between norms and inner products}.
\end{align*}
In fact, this chain of arguments proves that $U$ defines an isometry. 
Since the domain of $U$ is dense by the cyclic assumption on $\Omega$, the operator $U$ can be uniquely extended to a bounded operator $\overline{U}:\mathcal{H}\to\mathcal{K}$ that's also an isometry. 
Analogous statements hold for $U^*$, the adjoint of $U$. Since $U^*=U^{-1}$, the same holds for their extensions. Therefore, $\overline{U}$ is a unitary operator. To avoid cumbersome notation, we henceforth denote $\overline{U}$ by $U$. 

Next, in order to prove that the inner products give the same value independent of the vector representation, we use the notation $S_{\Omega}=J_{\Omega} \Delta_{\Omega}^{1/2}$ for the Tomita operator on $\mathcal{H}$ and $S_{\Xi}=J_{\Xi}\Delta_{\Xi}^{1/2}$ for the Tomita operator on $\mathcal{K}$ and their polar decompositions. Then, 
\[
U S_{\Omega} \pi(a)\,\Omega
=U \pi(a^*)\,\Omega
=\tau(a^*)\,\Xi
=S_{\Xi} \tau(a)\,\Xi
=S_{\Xi} U \pi(a)\,\Omega
\]
for all $a\in\Alg{A}$ proves that $U S_{\Omega}=S_{\Xi}U$ since both operators have domain $\mathcal{D}_{S_{\Omega}}$ and agree on that domain. 
Because of this
\[
J_{\Xi}\Delta_{\Xi}^{1/2}
=S_{\Xi}
=U S_{\Omega} U^*
=U J_{\Omega} \Delta_{\Omega}^{1/2} U^*
=(U J_{\Omega} U^*)( U \Delta_{\Omega}^{1/2} U^*)
\]
provides two polar decompositions of $S_{\Xi}$. 
By the uniqueness of polar decomposition discussed in Appendix~\ref{sec:reviewModularTheory}, we conclude that 
\[
J_{\Xi}=U J_{\Omega}U^*
\quad\text{ and }\quad
\Delta_{\Xi}^{1/2}=U\Delta_{\Omega}^{1/2}U^*.
\]
Therefore, 
we also have $U\Delta_{\Omega}^{1/4} U^*=\Delta_{\Xi}^{1/4}$. Therefore, 
\begin{align*}
  \<\!\<a_1, a_2\>\!\>_{\omega}&=  \left\langle \Delta_{\Omega}^{1/4} x_1\, \Omega,  \Delta_{\Omega}^{1/4} x_2\, \Omega \right\rangle &&\mbox{by definition and our earlier calculation} \\
    &=\left\langle U\Delta_{\Omega}^{1/4} x_1\, \Omega,  U\Delta_{\Omega}^{1/4} x_2\, \Omega \right\rangle &&\mbox{since $U$ is unitary}\\
    & =\left\langle \Delta_{\Xi}^{1/4} U x_1\, \Omega, \Delta_{\Xi}^{1/4} U x_2\, \Omega \right\rangle &&\mbox{since $U x_1\, \Omega= x_1\, \Xi$ and $U \Delta_{\Omega}^{1/4}=\Delta_{\Xi}^{1/4} U$} \\
    & =\left\langle \Delta_{\Xi}^{1/4} x_1\, \Xi, \Delta_{\Xi}^{1/4} x_2\, \Xi\right\rangle &&\mbox{by definition of $U$}.
\end{align*}
Thus, the inner product $\<\!\<\;\cdot\;,\;\cdot\;\>\!\>_{\omega}$ is independent of the chosen cyclic and separating vector state representation.
\eprf

\bd
Given a von~Neumann algebra $\Alg{A}$ and a faithful normal state $\omega$ on $\Alg{A}$, the inner product $\<\!\<\;\cdot\;,\;\cdot\;\>\!\>_{\omega}$ from~\eqref{eqn:ConnesSelfPolarForm} and Lemma~\ref{lem:ConnesInnerProduct} is called the 
\define{KMS inner product} associated with $(\Alg{A},\omega)$. 
\ed

\br
\label{rmk:KMS}
A consequence of the proof of the first part of Lemma~\ref{lem:ConnesInnerProduct} shows that 
\[
\<\!\<a_1,a_2\>\!\>_{\omega}=\big\<\Delta^{1/4}x_1\,\Omega,\Delta^{1/4}x_2\,\Omega\big\>
\]
for all $a_1,a_2\in\Alg{A}$, where $x_j=\pi(a_j)$. 
Yet another rewriting of this is 
\[
\<\!\<a_1,a_2\>\!\>_{\omega}
=\big\<x_1\,\Omega,\Delta^{1/2}x_2\,\Omega\big\>
=\big\< \Omega, \Delta^{-1/2}x_1^*\Delta^{1/2}x_2\,\Omega\big\>
=\omega\big(\sigma_{i/2}^{\omega}(a_1)^* a_{2}\big),
\]
where in the second equality we used the fact that $\Delta\, \Omega=\Omega$, and the final equality assumes that $a_1$ is an entire analytic element in $\Alg{A}$~\cite[Section 5.3.1]{BrRo2}. 
This last relation makes a closer connection to the KMS condition $\omega\big( a_{1} \sigma^{\omega}_{i}(a_2)\big)=\omega(a_2 a_1)$~\cite[Definition 5.3.1]{BrRo2} and shows how the selfpolar form of Connes reproduces the KMS two-point function at the midpoint of the standard horizontal strip in the complex plane (see Refs.~\cites{StratilaModular2020,Tak2} for more details). 
Also, writing it in this form makes a closer connection to another large body of work on quantum information geometry~\cites{PetzGhinea2011,SASDS23,LeRu99,Petz96}. Namely, following Ref.~\cite{Wirth2026} for example, we have 
\[
\<\!\<a_1,a_2\>\!\>_{\omega}^{f}
=\big\< x_1 \,\Omega, f(\Delta) y \, \Omega \big\>
\]
where 
\[
[0,\infty)\ni y\mapsto f(y):=\sqrt{y}
\]
(and $f(\Delta)$ is calculated using the functional calculus for positive unbounded operators). 
One can then replace $f$ with other suitable measurable functions $f:[0,\infty)\to[0,\infty)$ such that $f(\Delta)$ remains strictly positive and  
$\mathcal{D}_{f(\Delta)}$ is a suitably large and dense domain. 
Another useful identity is
\be
\label{eqn:CPIPredtostate}
\<\!\<1_{\Alg{A}},a\>\!\>_{\omega}=\omega(a)
\ee
for all $a\in\Alg{A}$.
Note that $\Alg{A}$ equipped with the inner product $\<\!\<\,\cdot\,,\,\cdot\,\>\!\>_{\omega}$ is not necessarily complete, i.e., $\Alg{A}$ equipped with this inner product is not necessarily a Hilbert space (Example~\ref{ex:CPinnerpFDcase} illustrates this point). This point will be important when we attempt to define and compute adjoints later. 
\er

\bx
\label{ex:CPinnerpFDcase}
Here, we re-express the 
KMS inner product on full matrix algebras, following the notation introduced in Example~\ref{ex:fdcaseintro}. We set $\Alg{A}=\matr_{m}$ so that $\Alg{M}=\Alg{A}\otimes\mathds{1}_{m}$. We write $x_1=a_1\otimes\mathds{1}_{m}$ and $x_2=a_2\otimes\mathds{1}_{m}$. We also let $\omega:\Alg{A}\to\C$ be a faithful state given by $\omega(a)=\tr[\rho a]$ for all $a\in\Alg{A}$, where $\rho$ is the full-rank density matrix associated with $\omega$. We diagonalize $\rho$ according to $\rho=\sum_{i}p_{i}\mathbf{u}_{i}\mathbf{u}_{i}^{\dag}$ for some orthonormal basis $\{\mathbf{u}_{i}\}$ and probabilities $\{p_i\}$. In the explicit calculation of $\<\!\<a_1,a_2\>\!\>_{\omega}$, we define $a^{T}$ to be the transpose of the matrix $a$ with respect to the basis $\{\mathbf{u}_{i}\}$. More precisely, writing $a=\sum_{k,l}a_{kl}\mathbf{u}_{k}\mathbf{u}_{l}^{\dag}$, then $a^{T}:=\sum_{k,l}a_{lk}\mathbf{u}_{k}\mathbf{u}_{l}^{\dag}$, i.e., $\mathbf{u}_{i}^{\dag}a^{T}\mathbf{u}_{j}=\mathbf{u}_{j}^{\dag}a\mathbf{u}_{i}$. Thus, the inner product between $a_1$ and $a_2$ then becomes
\[
\begin{split}
\<\!\<a_1,a_2\>\!\>_{\omega}
&=\<x_1\,\Omega,Jx_2^*\,\Omega\>
=\big((a_1\otimes\mathds{1}_{m})\Omega\big)^{\dag}J (a_2^{\dag}\otimes\mathds{1}_{m}) J J \Omega\\
&=\sum_{i,j}^{m}\sqrt{p_{i}p_{j}}\big(\mathbf{u}_{i}^{\dag}\otimes\mathbf{u}_{i}^{\dag}\big)\big(a_1^{\dag}\otimes\mathds{1}_{m}\big)\big(\mathds{1}_{m}\otimes a_2^{T}\big)\big(\mathbf{u}_{j}\otimes\mathbf{u}_{j}\big)\\
&=\sum_{i,j}^{m}\sqrt{p_{i}p_{j}}\big(\mathbf{u}_{i}^{\dag}a_1^{\dag}\mathbf{u}_{j}\big)\big(\mathbf{u}_{i}^{\dag}a_2^{T}\mathbf{u}_{j}\big)\\
&=\sum_{i,j}\sqrt{p_{i}p_{j}}\mathbf{u}_{i}^{\dag}a_1^{\dag}\mathbf{u}_{j}\mathbf{u}_{j}^{\dag}a_2\mathbf{u}_{i}\\
&=\sum_{i,j}\sqrt{p_{i}p_{j}}\Tr\big[\mathbf{u}_{i}\mathbf{u}_{i}^{\dag}a_1^{\dag}\mathbf{u}_{j}\mathbf{u}_{j}^{\dag}a_2\big]\\
&=\Tr\big[\rho^{1/2}a_1^{\dag}\rho^{1/2}a_2\big],
\end{split}
\]
where we have used the cyclicity of trace in the penultimate equality. 
Referring to Remark~\ref{rmk:KMS}, notice that this can also be rewritten as 
\[
\<\!\<a_1,a_2\>\!\>_{\omega}
=\Tr\big[\rho a_1^{\dag}\rho^{1/2}a_2\rho^{-1/2}\big]
=\omega\big( a_{1}^{\dag} \mathcal{J}^{f}_{\rho} (a_2) \big),
\]
where $\mathcal{J}_{\rho}^{f}=f(\mathcal{L}_{\rho}\mathcal{R}_{\rho}^{-1})$, with $f(y):=\sqrt{y}$ for all $y\in[0,\infty)$, and where $\mathcal{L}_{\rho}$ and $\mathcal{R}_{\rho}$ are left and right multiplication by $\rho$, i.e.,   $\mathcal{L}_{\rho}(a):=\rho a$ and $\mathcal{R}_{\rho}(a):= a\rho$. The expression $\mathcal{L}_{\rho}\mathcal{R}_{\rho}^{-1}$ is the way the modular operator $\Delta$ is often written in the literature for finite-dimensional systems~\cites{TKRWV10,GaoWilde2021,CaMa17,Petz96,CaVe20}, but it is important to note that it is not possible to use this same expression for infinite-dimensional systems in general%
\footnote{We know that $\Delta$ is in general unbounded for infinite-dimensional systems (see Example~\ref{ex:fdcaseintro}). Therefore, if we started with an element $a\in\Alg{A}$ represented as $x:=\pi(a)$ in $\Alg{M}\subseteq\mathcal{B}(\mathcal{H})$, then, although $\Delta x \,\Omega$ is a perfectly well-defined vector in $\mathcal{H}$, the operator $\Delta x$ need not be a bounded operator on $\mathcal{H}$, and therefore is \emph{not} generally an element of $\Alg{M}$, which by definition only includes bounded operators. Therefore, writing $\omega(a_1^{\dag} \Delta a_{2})$ in the general infinite-dimensional setting (cf.\ Remark~\ref{rmk:KMS}) requires some care.}, 
and $\Delta$, viewed as an unbounded operator on $\mathcal{H}$, is its appropriate extension to arbitrary von~Neumann algebras. 

From this finite-dimensional case, we can see why $\Alg{A}$ is not in general complete with respect to $\<\!\<\,\cdot\,,\,\cdot\,\>\!\>_{\omega}$. Let $p_{k}=\frac{6}{\pi^2 k^2}$ for all $k\in\N$ and let $\{\mathbf{u}_{m}\}$ be an orthonormal basis for an infinite-dimensional separable Hilbert space. Then $\rho=\sum_{m=1}^{\infty}p_{m}\mathbf{u}_{m}\mathbf{u}_{m}^{\dag}$ is an infinite-dimensional density matrix. To see that $\Alg{A}$, the algebra of all bounded operators on this Hilbert space, is not complete with respect to $\<\!\<\,\cdot\,,\,\cdot\,\>\!\>_{\omega}$, it suffices to construct a Cauchy sequence $a:\N\to\Alg{A}$ that does not converge to an element in $\Alg{A}$ in the associated norm topology. Set $a_{m}=\sum_{k=1}^{m}k^{1/4}\,\mathbf{u}_{k}\mathbf{u}_{k}^{\dag}$. Then each $a_{m}$ is a bounded operator, hence an element of $\Alg{A}$, and (assuming without loss of generality that $n\ge m$)
\[
\lVert a_{n}-a_{m}\rVert^{2}_{\omega}
=\Tr\big[\rho^{1/2}(a_{n}-a_{m})^{\dag}\rho^{1/2}(a_{n}-a_{m})\big]
=\frac{6}{\pi^2}\sum_{j=m+1}^{n}\frac{1}{j^{3/2}}
\]
shows that $a$ is indeed a Cauchy sequence. However, $a_{\infty}:=\lim_{m\to\infty}a_{m}=\sum_{k=1}^{\infty}k^{1/4}\mathbf{u}_{k}\mathbf{u}_{k}^{\dag}$ is an unbounded operator, because $\lVert a_{\infty} \mathbf{u}_{m}\rVert=m^{1/4}$ (and so its operator norm is infinite). Therefore $a_{\infty}$ is not in $\Alg{A}$. Thus, even in this simple example, $\Alg{A}$ is not complete with respect to $\<\!\<\,\cdot\,,\,\cdot\,\>\!\>_{\omega}$.
\ex

We will now state some useful facts regarding the 
KMS inner product that we will use later. 

\begin{lemma}
\label{limilemm} 
Let $\Alg{A}$ be a von~Neumann algebra, let $\omega:\Alg{A}\to\C$ be a faithful normal state, and let $a:A\to\Alg{A}$ be a net that converges to $a_{\infty}:=\lim_{\alpha\in A}a_{\alpha}\in\Alg{A}$ in the weak operator topology%
\footnote{Since $\Alg{A}$ is a subalgebra of bounded operators on a Hilbert space, the weak operator topology on $\Alg{A}$ is defined in terms of the weak operator topology on bounded operators (see Appendix~\ref{sec:reviewModularTheory}).}. 
 Then 
    \[
        \<\!\<\tilde{a}, a_{\infty}\>\!\>_{\omega}=\lim_{\alpha\in A} \<\!\<\tilde{a}, a_{\alpha}\>\!\>_{\omega} \quad\text{ for all }\quad \tilde{a} \in \Alg{A}\,.
    \]
\end{lemma}
    \begin{proof}
    Let $\pi:\Alg{A}\to\mathcal{B}(\mathcal{H})$ be an injective normal representation of $\Alg{A}$ on a Hilbert space $\mathcal{H}$ and let $\Omega\in\mathcal{H}$ be a cyclic and separating vector for $\omega$. Set $x=\pi\circ a$ so that $x_\alpha=\pi(a_{\alpha})$ and $x_{\infty}=\pi(a_{\infty})$. Since $\lim_{\alpha\in A} a_\alpha=a_{\infty}$ in the weak operator topology, we have $\lim_{\alpha\in A} a_\alpha^{\ast}=a_{\infty}^{\ast}$ in the weak operator topology as well. Therefore, 
        \begin{align*}
           \<\!\<\tilde{a}, a_{\infty}\>\!\>_{\omega}
           &=\langle \tilde{x} \,\Omega, J x_{\infty}^{\ast} \,\Omega \rangle & & \mbox{ by~\eqref{eqn:ConnesSelfPolarForm}} \\
            &=\langle x_{\infty}^{\ast} \, \Omega, J \tilde{x} \,\Omega \rangle  & & \mbox{ since $J$ is a self-adjoint conjugate-linear unitary} \\
            &=\lim_{\alpha\in A}  \langle x_{\alpha}^{\ast} \, \Omega, J \tilde{x} \,\Omega \rangle & & \mbox{ since $\pi$ is normal and hence weakly continuous} \\
            &= \lim_{\alpha\in A} \langle \tilde{x} \, \Omega, J x_{\alpha}^{\ast} \,\Omega \rangle & & \mbox{ since $J$ is a self-adjoint conjugate-linear unitary}\\
            &=\lim_{\alpha\in A} \<\!\<\tilde{a}, a_{\alpha}\>\!\>_{\omega}  & & \mbox{ by~\eqref{eqn:ConnesSelfPolarForm}}
        \end{align*}
        for all $\tilde{a}\in\Alg{A}$. 
    \end{proof}

\blem
\label{lem:PetzIPtensor}
Let $(\Alg{A}_{i},\omega_{i})$, with $i\in\{1,2\}$, denote two von~Neumann algebras equipped with faithful states. Then 
\[
\<\!\<a_1\otimes a_2,\tilde{a}_1\otimes \tilde{a}_{2}\>\!\>_{\omega_1\otimes\omega_2}=\<\!\<a_1,\tilde{a}_{1}\>\!\>_{\omega_1}\<\!\<a_2,\tilde{a}_{2}\>\!\>_{\omega_2}
\]
for all $a_1,\tilde{a}_{1}\in\Alg{A}_{1}$ and $a_2,\tilde{a}_{2}\in\Alg{A}_{2}$. 
\elem

\bprf
Let $(\mathcal{H}_{i},\pi_{i},\Omega_{i})$ denote GNS representations for $(\Alg{A}_{i},\omega_{i})$ with $i\in\{1,2\}$. 
If we set $x_{i}:=\pi_{i}(a_{i})$ and $\tilde{x}_{i}=\pi_{i}(\tilde{a}_i)$ for $i=1,2$, then 
\begin{align}
\<\!\<a_1\otimes a_2,\tilde{a}_1\otimes \tilde{a}_{2}\>\!\>_{\omega_1\otimes\omega_2}
&=\big\<(x_1\otimes x_2)(\Omega_1\otimes\Omega_2),(J_1\otimes J_2)(\tilde{x}_1\otimes \tilde{x}_2)^*(\Omega_1\otimes\Omega_2)\big\> \nonumber \\
&=\big\< x_1\Omega_1\otimes x_2\Omega_2,J_1 \tilde{x}_1^*\Omega_1\otimes J_2\tilde{x}_{2}^*\Omega_2\big\> \nonumber\\
&=\<x_1\Omega_1,J_1\tilde{x}_{1}^*\Omega_1\> \<x_2\Omega_2,J_2\tilde{x}_{2}^*\Omega_2\> \nonumber \\
&=\<\!\<a_1,\tilde{a}_{1}\>\!\>_{\omega_1}\<\!\<a_2,\tilde{a}_{2}\>\!\>_{\omega_2} \nonumber
\end{align}
for all $a_1,\tilde{a}_{1}\in\Alg{A}_{1}$ and $a_2,\tilde{a}_{2}\in\Alg{A}_{2}$. In these calculations, we have used properties of tensor products reviewed at the end of Appendix~\ref{sec:reviewModularTheory}. 
\eprf

\section{Adjoints with respect to the KMS inner product}
\label{sec:PetzMap}

In this section, we use the 
KMS inner product to construct adjoints of NCPU maps. These adjoints define the Petz recovery maps of quantum operations. Because we will analyze the categorical properties of the Petz recovery map, we will begin by introducing the relevant category to view the Petz recovery map as a special kind of functor on this category. 

Let $\mathbf{Q}$ be the category defined as follows. An object of $\mathbf{Q}$ is a pair $(\Alg{A},\omega)$ consisting of a von~Neumann algebra $\Alg{A}$ together with a faithful normal state $\omega:\Alg{A}\to\C$. A morphism from $(\Alg{B},\xi)$ to $(\Alg{A},\omega)$ is a NCPU map $\mathcal{E}:\Alg{B}\to\Alg{A}$ such that $\xi=\omega\circ\mathcal{E}$. The identity morphism on $(\Alg{A},\omega)$ is the identity function. Composition of morphisms is ordinary composition of functions. As mentioned in Appendix~\ref{sec:reviewModularTheory},  
the composite of two NCPU maps remains NCPU so that $\mathbf{Q}$ indeed defines a category. 
A morphism of the form $(\Alg{B},\xi)\xrightarrow{\mathcal{E}}(\Alg{A},\omega)$ is written in this way, namely from $\Alg{B}$ to $\Alg{A}$, to make the connection with quantum information theory more direct. The algebra $\Alg{A}$ represents the algebra of observables accessible to Alice, while the algebra $\Alg{B}$ for Bob. Alice and Bob are two fictitious characters used to provide an anthropomorphic description for the algebras of observables of a quantum system evolving in one discrete time step. The evolution is written from Bob to Alice since this is describing the Heisenberg picture for how observables evolve in time (as opposed to the Schr\"odinger picture, which describes how states evolve in time). More on this will be described later in Example~\ref{ex:Petzmapfinitedimensional}.  

\bt
\label{thm:PetzRecovery}
Let $(\Alg{B},\xi)\xrightarrow{\mathcal{E}}(\Alg{A},\omega)$ be a morphism in $\mathbf{Q}$. Then there exists a unique linear map $\mathcal{E}^{\star}:\Alg{A}\to\Alg{B}$ such that 
\[
\<\!\<\mathcal{E}(b),a\>\!\>_{\omega}=\<\!\<b,\mathcal{E}^{\star}(a)\>\!\>_{\xi}
\]
for all $a\in\Alg{A}$ and $b\in\Alg{B}$. Moreover, $\mathcal{E}^{\star}$ defines a morphism $(\Alg{B},\xi)\xleftarrow{\mathcal{E}^{\star}}(\Alg{A},\omega)$ in $\mathbf{Q}$. 
\et

The proof of this theorem will follow a combination of the proof of~\cite[Proposition~3.1]{AcCe82} and the content in~\cite[Section 3]{Pe84}. However, we will prove each step in thorough detail being sure to keep the proof as self-contained as possible. 

\bprf[Proof of Theorem~\ref{thm:PetzRecovery}]
This proof will be split into the following sequence of steps. First, we will construct $\mathcal{E}^{\star}$. 
Second, we will prove that it satisfies the adjointness condition in the statement of the theorem. 
Third, we will prove that $\mathcal{E}^{\star}\circ\xi=\omega$. 
Fourth, and finally, we will prove that $\mathcal{E}^{\star}$ is NCPU. 
To set the notation, let $\pi:\Alg{A}\to\mathcal{B}(\mathcal{H})$ be a representation of $\Alg{A}$ and let $\Omega$ be a cyclic and separating vector for $\Alg{M}:=\pi(\Alg{A})$ that also represents $\omega$ as a vector state. Similarly, let $\tau:\Alg{B}\to\mathcal{B}(\mathcal{K})$ be a representation of $\Alg{B}$ and let $\Xi$ be a cyclic and separating vector for $\mathcal{N}:=\tau(\Alg{B})$ that also represents $\xi$ as a vector state. As in the previous section, we may occasionally not explicitly write the representation maps $\pi$ and $\tau$ for brevity in some calculations by using notation such as $x=\pi(a)\in\Alg{M}$ and $y=\tau(b)\in\Alg{N}$. 

Following the method of proof in \cite[Proposition~3.1]{AcCe82}), for each $x'\in\Alg{M}'\subseteq\mathcal{B}(\mathcal{H})$ such that $x'\ge0$, define 
$\chi_{x'}:\Alg{N}\to\C$ by 
\[
\Alg{N}\ni y \mapsto
\chi_{x'}(y):=\<x'\,\Omega,\mathcal{E}(y)\,\Omega\>, 
\]
where $\mathcal{E}(y):=\pi(\mathcal{E}(b))\in\Alg{M}$, with $b\in\Alg{B}$ the unique element such that $y=\tau(b)$. 
Then for each $y\in\Alg{N}$ with $y\ge0$, we have 
\[
\chi_{x'}(y)=\left\<\sqrt{x'}\sqrt{\mathcal{E}(y)}\Omega,\sqrt{x'}\sqrt{\mathcal{E}(y)}\Omega\right\>
\]
since $x'\ge0$, $y\ge0$, $\mathcal{E}(y)\ge0$, and $[x',\mathcal{E}(y)]=0$. This proves that $\chi_{x'}$ is a positive map. Hence, 
\[
\chi_{x'}(y)=\left\lVert\sqrt{x'}\sqrt{\mathcal{E}(y)}\Omega\right\rVert^2
\le \big\lVert\sqrt{x'}\big\rVert^2 \big\lVert\sqrt{\mathcal{E}(y)}\Omega\big\rVert^2
=\lVert x'\rVert \<\Omega,\mathcal{E}(y)\Omega\>
=\lVert x'\rVert \omega\big(\mathcal{E}(y)\big)
=\lVert x'\rVert \xi(y)
\]
by the properties of bounded operators and their norms as well as the state-preserving condition $\xi=\omega\circ\mathcal{E}$. 
This calculation proves that $\chi_{x'}\le_{u}\xi$ for all $x'\in\Alg{M}'\subseteq\mathcal{B}(\mathcal{H})$ with $x'\ge0$. Hence, there exists an element, namely the Radon--Nikodym derivative $\mathcal{E}'(x'):=\frac{d\chi_{x'}}{d\xi}\in\mathcal{N}'$, such that $\mathcal{E}'(x')\ge0$ and 
\begin{align}\label{AccardiCech}
    \big\langle x' \,\Omega , \E(y)\, \Omega \big\rangle
    =\chi_{x'}(y)
    =\big\langle \Xi,\mathcal{E}'(x') y\,\Xi\big\rangle
    = \big\langle \E'(x')\, \Xi , y\, \Xi\big\rangle.
\end{align}
The linearity property of the Radon--Nikodym derivative, reviewed in Appendix~\ref{sec:reviewModularTheory}, guarantees that $\mathcal{E}'$ satisfies 
\[
\mathcal{E}'(x_{1}'+c x_{2}')=\mathcal{E}'(x_{1}')+c\mathcal{E}'(x_2')
\]
for all $x_{1}',x_{2}'\in\Alg{M}'\subseteq\mathcal{B}(\mathcal{H})$ with $x_{1}',x_{2}'\ge0$ and $c\ge0$. Since every general element $x'\in\Alg{M}'\subseteq\mathcal{B}(\mathcal{H})$ can be written as a unique complex combination of positive elements in $\Alg{M}'$ as 
\[
x'=x_{\Re}^{\prime \; +}-x_{\Re}^{\prime \; -}+ix_{\Im}^{\prime \; +}-ix_{\Im}^{\prime \; -}
\]
via~\eqref{eqn:positivedecomp}, $\mathcal{E}'$ uniquely extends to a linear map by defining 
\[
\mathcal{E}'(x'):=\mathcal{E}'\big(x_{\Re}^{\prime \; +}\big) -\mathcal{E}'\big(x_{\Re}^{\prime \; -}\big)+i\mathcal{E}'\big(x_{\Im}^{\prime \; +}\big)-i\mathcal{E}'\big(x_{\Im}^{\prime \; -}\big).
\] 
Thus, $\mathcal{E}':\Alg{M}'\to\Alg{N}'$ is a positive linear map by construction. 
Following Ref.~\cite{AcCe82}, we prove normality and complete positivity of $\mathcal{E}'$. For normality, let $x':A\to\Alg{M}'$ be an increasing net of positive elements in $\Alg{M}'$, with each element in the net being written as $x'_{\alpha}$ with $\alpha\in A$, such that $x'_{\infty}:=\sup_{\alpha}x_{\alpha}\in\Alg{M}'$ exists. Then 
\[
\big\<\mathcal{E}'(x'_{\infty})\,\Xi,y\,\Xi\big\>
=\big\<x'_{\infty}\,\Omega,\mathcal{E}(y)\,\Omega\big\>
=\sup_{\alpha}\big\<x'_{\alpha}\,\Omega,\mathcal{E}(y)\,\Omega\big\>
=\sup_{\alpha}\big\<\mathcal{E}'(x'_{\alpha})\,\Xi,y\,\Xi\big\>
=\left\<\sup_{\alpha}\mathcal{E}'(x'_{\alpha})\,\Xi,y\,\Xi\right\>
\]
for all $y\in \Alg{N}$. By the cyclic and separating condition on $\Xi$, this proves that $\mathcal{E}'(x'_{\infty})=\sup_{\alpha}\mathcal{E}'(x'_{\alpha})$, therefore showing that $\mathcal{E}'$ is normal (it is interesting to note that this is true whether or not $\mathcal{E}$ is normal). Secondly, $\mathcal{E}'$ is also CP since for every $n\in \N$, 
\begin{align*}
\sum_{j,k=1}^{n} \left\< y_{j}\,\Xi, \mathcal{E}'(x_{j}^{\prime \, *} x'_{k})  y_{k} \, \Xi \right\>
&=
\sum_{j,k=1}^{n} \overline{\left\< \mathcal{E}'(x_{j}^{\prime \, *} x'_{k}) y_{k}\,\Xi,y_{j} \, \Xi \right\>} && \mbox{ by the conjugate-symmetry of $\<\,\cdot\,,\,\cdot\,\>$} \\
&=\sum_{j,k=1}^{n} \overline{\left\< \mathcal{E}'(x_{j}^{\prime \, *} x'_{k}) \,\Xi,y_{k}^{*} y_{j} \, \Xi \right\>} && \mbox{ because $y_{k}$ commutes with $\mathcal{E}'(x_{j}^{\prime \, *} x'_{k})$ } \\
&=\sum_{j,k=1}^{n} \overline{\left\< x_{j}^{\prime \, *} x'_{k} \,\Omega, \mathcal{E}( y_{k}^{*} y_{j}) \, \Omega \right\>} && \mbox{ by~\eqref{AccardiCech}, the definition of $\mathcal{E}'$ } \\
&=\sum_{j,k=1}^{n} \overline{\left\< x'_{k} \,\Omega, \mathcal{E}( y_{k}^{*} y_{j}) x_{j}' \, \Omega \right\>} && \mbox{ since $x_{j}'$ commutes with $\mathcal{E}( y_{k}^{*} y_{j})$ }\\
&\ge0 &&\mbox{ since $\mathcal{E}$ is CP}
\end{align*}
for all $x'_{1},\dots,x'_{n}\in\Alg{M}'$ and for all $y_{1},\dots,y_{n}\in\Alg{N}$. Since $\Xi$ is cyclic and separating, this shows positivity of $(\id_{\matr_{n}}\otimes\mathcal{E}')(x^{\prime\,*}x')$ on a dense subspace of $\C^{n}\otimes\mathcal{H}\cong\mathcal{H}^{\oplus n}$. 
Hence, there exists a unique NCPU map $\E' \colon \Alg{M}' \to \Alg{N}'$ satisfying~\eqref{AccardiCech}
for all $x'\in\Alg{M}'$ and $y\in\Alg{N}$. 

Next, define $\mathcal{E}^{\star}:\Alg{M}\to\Alg{N}$ by 
\be
\label{eqn:Estardefn}
\Alg{M}\ni x\mapsto 
\mathcal{E}^{\star}(x):=J_{\Xi} \E'(J_{\Omega} x J_{\Omega}) J_{\Xi}. 
\ee
We also abuse notation and define $\mathcal{E}^{\star}:\Alg{A}\to\Alg{B}$ sending $a\in\Alg{A}$ to the unique element $\mathcal{E}^{\star}(a)\in\Alg{B}$ satisfying the condition that $\tau(\mathcal{E}^{\star}(a))=\mathcal{E}^{\star}(\pi(a))$.
Note that $\mathcal{E}^{\star}$ is NCPU if and only if $\mathcal{E}$ is NCPU since the two are related by adjoint actions by conjugate-linear unitary operators. 
Due to all of the different maps appearing in this definition, we find it convenient to write them all down in the following diagram
\[
\xy0;/r.25pc/:
(-20,-12.5)*+{\Alg{B}\cong\Alg{N}}="N";
(-7.5,0)*+{\mathcal{B}(\mathcal{K})}="BK";
(-20,12.5)*+{\Alg{N}'}="Np";
(20,-12.5)*+{\Alg{M}\cong\Alg{A}}="M";
(7.5,0)*+{\mathcal{B}(\mathcal{H})}="BH";
(20,12.5)*+{\Alg{M}'}="Mp";
{\ar@{^{(}->}"N";"BK"};
{\ar@{^{(}->}"Np";"BK"};
{\ar@/_0.75pc/"N";"M"_{\mathcal{E}}};
{\ar@/_0.75pc/"M";"N"_{\mathcal{E}^{\star}}};
{\ar@{^{(}->}"M";"BH"};
{\ar@{^{(}->}"Mp";"BH"};
{\ar"Mp";"Np"_{\mathcal{E}'}};
{\ar@{<->}"M";"Mp"_{\Ad_{J_{\Omega}}}};
{\ar@{<->}"N";"Np"^{\Ad_{J_{\Xi}}}};
\endxy
\]
in order to keep track of all their domains. 

We next prove that $\<\!\<\mathcal{E}(b),a\>\!\>_{\omega}=\<\!\<b,\mathcal{E}^{\star}(a)\>\!\>_{\xi}$. On the one hand, we have 
\begin{align*}
\<\!\<\mathcal{E}(b),a\>\!\>_{\omega}&=\<\mathcal{E}(y)\,\Omega,J_{\Omega}x^*\,\Omega\> & & \text{by definition of $\<\!\<\;\cdot\;,\;\cdot\;\>\!\>_{\omega}$, $x=\pi(a)$, and $y=\tau(b)$} \\
&=\<\mathcal{E}(y)\,\Omega,J_{\Omega}x^*J_{\Omega}\,\Omega\> & & \text{since $J_{\Omega}\Omega=\Omega$}\\
&=\<\mathcal{E}(y)\,\Omega,x'\,\Omega\> & & \text{where $x':=J_{\Omega}x^*J_{\Omega}$, since $J_{\Omega}\Alg{M}J_{\Omega}=\Alg{M}'$}\\
&= \<y\,\Xi,\mathcal{E}'(x')\,\Xi\>  & & \text{by~\eqref{AccardiCech}.}
\end{align*}
On the other hand, 
\begin{align*}
\<\!\<b,\mathcal{E}^{\star}(a)\>\!\>_{\xi}&=\<y\,\Xi,J_{\Xi}\mathcal{E}^{\star}(x)^*\,\Xi\> & &  \text{by definition of $\<\!\<\;\cdot\;,\;\cdot\;\>\!\>_{\xi}$, $x$, and $y$}\\
&=\Big\<y\,\Xi,J_{\Xi}\big(J_{\Xi} \mathcal{E}'(J_{\Omega} xJ_{\Omega})J_{\Xi}\big) ^*\,\Xi\Big\> & &  \text{by definition of $\mathcal{E}^{\star}$ from~\eqref{eqn:Estardefn}}\\
&=\big\<y\,\Xi, J_{\Xi} J_{\Xi}^{*}\mathcal{E}'(J_{\Omega}^* x^*J_{\Omega}^*) J_{\Xi}^{*}\,\Xi\big\> & &  \text{since $\mathcal{E}'$ is $*$-preserving}\\
&=\big\<y\,\Xi, \mathcal{E}'(J_{\Omega} x^*J_{\Omega})\,\Xi\big\> & &  \text{since $J^*=J$, $J^2=\id$, and $J_{\Xi}\Xi=\Xi$}\\
&=\big\<y\,\Xi, \mathcal{E}'(x')\,\Xi\big\>  & & \text{since $x':=J_{\Omega}x^*J_{\Omega}$}.
\end{align*}
This proves all the claims, including the fact that $\xi\circ\mathcal{E}^{\star}=\omega$ due to~\eqref{eqn:CPIPredtostate}. 
\eprf

\bd
\label{defn:PetzRecoveryMap}
The map $\mathcal{E}^{\star}$ in Theorem~\ref{thm:PetzRecovery} is called the \define{Petz recovery map} of the morphism $(\Alg{B},\xi)\xrightarrow{\mathcal{E}}(\Alg{A},\omega)$. When viewed as a morphism in $\mathbf{Q}$, the morphism $(\Alg{B},\xi)\xleftarrow{\mathcal{E}^{\star}}(\Alg{A},\omega)$ in $\mathbf{Q}$ is called the \define{Petz retrodiction} of $(\Alg{B},\xi)\xrightarrow{\mathcal{E}}(\Alg{A},\omega)$. 
\ed

\bx
\label{ex:Petzmapfinitedimensional}
We now examine Theorem~\ref{thm:PetzRecovery} in the case that $\Alg{A}=\matr_{m}$, $\Alg{M}=\Alg{A}\otimes\mathds{1}_{m}$, $\Alg{B}=\matr_{n}$, and $\Alg{N}=\Alg{B}\otimes\mathds{1}_{n}$ by explicitly constructing $\mathcal{E}^{\star}$. We use the notation from Examples~\ref{ex:CPinnerpFDcase} and~\ref{ex:fdcaseintro}, but we also represent the density matrix for $\xi$ by $\sigma$ so that $\xi=\Tr[\sigma\;\cdot\;]$. We write an eigendecomposition of $\sigma$ as $\sigma=\sum_{i}q_{i}\mathbf{v}_{i}\mathbf{v}_{i}^{\dag}$ with $\{\mathbf{v}_{i}\}$ an orthonormal basis of $\mathcal{K}=\C^{n}$ and $\{q_{i}\}$ a probability distribution, which is nowhere vanishing due to the faithfulness condition on $\xi$. Due to the finite-dimensionality assumption, we may construct $\mathcal{E}^{\star}$ directly from the definition of the 
KMS inner product and our expression of it in terms of density matrices and traces from Example~\ref{ex:CPinnerpFDcase}. Namely, $\mathcal{E}^{\star}$ is the unique map satisfying 
\begin{equation}
\label{eqn:CPinnerprodforEandEstar}
\Tr\big[\rho^{1/2}\mathcal{E}(b)^{\dag}\rho^{1/2}a\big]=\Tr\big[\sigma^{1/2}b^{\dag}\sigma^{1/2}\mathcal{E}^{\star}(a)\big]
\end{equation}
for all $a\in\matr_{m}$ and $b\in\matr_{n}$. Rewriting the LHS of~\eqref{eqn:CPinnerprodforEandEstar} yields
\[
\Tr\big[\rho^{1/2}\mathcal{E}(b)^{\dag}\rho^{1/2}a\big]
=\Tr\big[\mathcal{E}(b)^{\dag}\rho^{1/2}a\rho^{1/2}\big]
=\Tr\big[b^{\dag}\mathcal{E}^*\big(\rho^{1/2}a\rho^{1/2}\big)\big],
\]
where $\mathcal{E}^*$ denotes the Hilbert--Schmidt adjoint of $\mathcal{E}$. Rewriting the RHS of~\eqref{eqn:CPinnerprodforEandEstar} yields 
\[
\Tr\big[\sigma^{1/2}b^{\dag}\sigma^{1/2}\mathcal{E}^{\star}(a)\big]
=\Tr\big[b^{\dag}\sigma^{1/2}\mathcal{E}^{\star}(a)\sigma^{1/2}\big].
\]
Using the equality in~\eqref{eqn:CPinnerprodforEandEstar} and the faithfulness of the trace then yields 
\[
\mathcal{E}^*\big(\rho^{1/2}a\rho^{1/2}\big)=\sigma^{1/2}\mathcal{E}^{\star}(a)\sigma^{1/2}.
\]
Isolating $\mathcal{E}^{*}$ then gives
\[
\mathcal{E}^{\star}(a)=\sigma^{-1/2}\mathcal{E}^*\big(\rho^{1/2}a\rho^{1/2}\big)\sigma^{-1/2}
\]
for all $a\in\matr_{m}$, i.e., 
\[
\mathcal{E}^{\star}=\Ad_{\sigma^{-1/2}}\circ\mathcal{E}^{*}\circ\Ad_{\rho^{1/2}}. 
\]
We note that this might not look exactly like the standard Petz recovery map in some of the quantum information literature due to the fact that we have been working in the Heisenberg picture in the sense that the channels $\mathcal{E}$ and $\mathcal{E}^{\star}$ are viewed as functions sending observables to observables (as opposed to the Schr\"odinger picture, which involves functions sending density matrices to density matrices). 
\ex

As another class of examples for which the Petz recovery map takes on a particularly simple form is for isomorphisms $(\Alg{B},\xi)\xrightarrow{\mathcal{E}}(\Alg{A},\omega)$ in $\mathbf{Q}$, which implies $\mathcal{E}$ is a 
$*$-isomorphism (we note that ``isomorphism'' here is intended in the sense of category theory~\cite{Ma98}). 
We first state a lemma that will be used for the proposition.

\blem
\label{lem:contractionforE}
Let $(\Alg{B},\xi)\xrightarrow{\mathcal{E}}(\Alg{A},\omega)$ in $\mathbf{Q}$ be a morphism in $\mathbf{Q}$. Let $(\mathcal{H},\pi,\Omega)$ and $(\mathcal{K},\tau,\Xi)$ be GNS representations associated with $(\Alg{A},\omega)$ and $(\Alg{B},\xi)$, respectively. Set $\Alg{M}:=\pi(\Alg{A})$ and $\Alg{N}:=\tau(\Alg{B})$. Define the operator $E:\Alg{N}\,\Xi\to\Alg{M}\,\Omega$ by 
\[
E(y\,\Xi):=\widetilde{\mathcal{E}}(y)\,\Omega
\] 
for all $y\in\Alg{N}$, where $\widetilde{\mathcal{E}}:\Alg{N}\to\Alg{M}$ is defined by $\widetilde{\mathcal{E}}(y):=\pi(\mathcal{E}(b))$, where $y=\tau(b)$ for a unique $b\in\Alg{B}$. 
Then $E$ is well-defined and satisfies the following properties: 
\begin{enumerate}[i.]
\item 
$E(\Xi)=\Omega$, 
\item 
$E$ defines a contraction, hence extends uniquely to a bounded operator from $\mathcal{K}$ to $\mathcal{H}$, and
\item
$E^*\Omega=\Xi$.
\end{enumerate}
\elem

Before proving this, we make a comment as to why we use $\widetilde{\mathcal{E}}$ in the formulas instead of $\mathcal{E}$. One should think of $\mathcal{E}$ and $\widetilde{\mathcal{E}}$ as being the same exact quantum channel, except that $\widetilde{\mathcal{E}}$ is defined between subalgebras of the bounded operators on the GNS representations. The usage of $\widetilde{\mathcal{E}}$ is to be precise about the domains and codomains. 

\bprf[Proof of Lemma~\ref{lem:contractionforE}]
Before proving the properties, note that $E$ is a densely defined operator from $\mathcal{K}$ to $\mathcal{H}$. It is well-defined because if $y\,\Xi=\tilde{y}\,\Xi$, then $y=\tilde{y}$ by the separating condition for $\Xi$. It is densely defined since $\Xi$ is cyclic. 
\begin{enumerate}[i.]
\item First, $E(\Xi)=\Omega$ since $\mathcal{E}(1_{\Alg{N}})=1_{\Alg{M}}$. 
\item Second, $E$ is a contraction because
\begin{align*}
\big\lVert E(y\,\Xi)\big\rVert^2
&=\big\<\widetilde{\mathcal{E}}(y)\,\Omega,\widetilde{\mathcal{E}}(y)\,\Omega\big\> && \text{by definition of $E$} \\
&=\big\<\Omega,\widetilde{\mathcal{E}}(y)^*\widetilde{\mathcal{E}}(y)\,\Omega\big\> && \text{since $\widetilde{\mathcal{E}}(y)$ is bounded} \\ 
&=\omega\big(\mathcal{E}(b)^*\mathcal{E}(b)\big) & & \text{where $b\in\Alg{B}$ is the unique element such that $y=\tau(b)$ }\\
&\le\omega\big(\mathcal{E}(b^*b)\big) & & \text{by the Kadison--Schwarz inequality for $\mathcal{E}$} \\
&=\xi(b^* b) && \text{by the state-preserving condition on $\mathcal{E}$} \\
&=\lVert y\,\Xi\rVert^2  && \text{by definition of $\Xi$} 
\end{align*}
for all $y\in\Alg{N}$, which proves that $\lVert E\rVert\le 1$, so that $E$ is a contraction and hence bounded on its domain. Since the domain of $E$ is dense by the cyclic property of $\Xi$, the operator $E$ extends uniquely to a bounded operator from $\mathcal{K}$ to $\mathcal{H}$. In what follows, we also denote this bounded operator by $E$.

\item
Setting $E^*$ to be the adjoint of $E$, we have 
\[
\<\Xi,y\,\Xi\>
=\xi(b)
=\omega\big(\mathcal{E}(b)\big)
=\big\<\Omega,\widetilde{\mathcal{E}}(y)\,\Omega\big\>
=\big\<\Omega,E y\,\Xi\big\>
=\big\<E^*\Omega,y\,\Xi\big\>
\]
holds for all $y=\tau(b)\in\Alg{N}$. By the cyclic condition on $\Xi$ (density of $\Alg{N}\,\Xi$), this implies $E^*\Omega=\Xi$. 
\end{enumerate}
This proves all three claims. 
\eprf

\bn
\label{prop:PetzInverting}
Let $(\Alg{B},\xi)\xrightarrow{\mathcal{E}}(\Alg{A},\omega)$ in $\mathbf{Q}$ be an isomorphism in $\mathbf{Q}$. Let $E$ be defined as in Lemma~\ref{lem:contractionforE}. Then, $\mathcal{E}$ is a $*$-isomorphism and $\mathcal{E}^{\star}=\mathcal{E}^{-1}$. 
\en

Although this is proved in~\cite[Corollary 10]{Pe84}, we provide an alternative proof that does not rely on the modular automorphism group, but instead only uses modular operators. 

\bprf
Suppose that $\mathcal{E}$ is an isomorphism, i.e., $\mathcal{E}^{-1}$ exists and is a morphism $(\Alg{B},\xi)\xleftarrow{\mathcal{E}^{-1}}(\Alg{A},\omega)$ in $\mathbf{Q}$. Since $\mathcal{E}$ and $\mathcal{E}^{-1}$ are CPU, they satisfy the Kadison--Schwarz inequality, namely
\be
\label{KS}
    \E(b)^{\ast}\E(b) \leq \E(b^{\ast}b) \qquad\text{ for all }\; b \in \Alg{B}\,
\ee
and 
\be
\label{KSinv}
    \E^{-1}(a)^{\ast}\E^{-1}(a) \leq \E^{-1}(a^{\ast} a) \qquad\text{ for all }\;  a \in \Alg{A}. 
\ee
Now, given $b\in\Alg{B}$, set $a=\E(b)$ and substitute $a$ in the second inequality~\eqref{KSinv} to get
\be
\label{intineq}
    b^{\ast}b \leq \E^{-1}(\E(b)^{\ast}\E(b)).
\ee
Apply $\E$ on both sides of the inequality~\eqref{intineq} and use the fact that $\E$ is positive to obtain
\be
\label{revKad}
    \E(b^{\ast}b) \leq \E(b)^{\ast}\E(b).
\ee
Combining the inequality~\eqref{KS} and the inequality~\eqref{revKad} yields
\[
\E(b^{\ast}b)=\E(b)^{\ast}\E(b) \qquad\text{ for all }\; b \in \Alg{B}\,.
\]
By Choi's multiplicative domain theorem~\cite[Theorem 3.1]{Choi1974}, it follows that $\mathcal{E}$ is a unital $*$-homomorphism. For similar reasons, $\mathcal{E}^{-1}$ is a unital $*$-homomorphism. Thus, $\mathcal{E}$ is a $*$-isomorphism. 

Using the notation from Lemma~\ref{lem:contractionforE}, we next prove that $E$ is a unitary operator. To see that $E$ is an isometry, note that
\begin{align*}
\< E b_{1}\,\Xi, E b_{2}\,\Xi\>
&=\big\<\mathcal{E}(b_1)\,\Omega,\mathcal{E}(b_2)\,\Omega\big\>\\
&=\big\<\Omega,\mathcal{E}(b_1)^*\mathcal{E}(b_2)\,\Omega\big\>\\
&=\big\<\Omega,\mathcal{E}(b_1^*b_2)\,\Omega\big\>\\
&=\<\Xi,b_1^*b_2\,\Xi\>\\
&=\<b_1 \,\Xi,b_2\,\Xi\>
\end{align*}
for all $a\in\Alg{A}$ and $b\in\Alg{B}$, where we have abused notation by the identifications $\Alg{M}\cong\Alg{A}$ and $\Alg{N}\cong\Alg{B}$. Similarly, let $F:\mathcal{H}\to\mathcal{K}$ be the isometry characterized by $F(a\,\Omega)=\mathcal{E}^{-1}(a)\,\Xi$ for all $a\in\Alg{A}$. Since 
\[
E F a\,\Omega
=E \mathcal{E}^{-1}(a) \,\Xi
=\mathcal{E}\big(\mathcal{E}^{-1}(a)\big)\,\Omega
=a\,\Omega
\]
for all $a\in\Alg{A}$, continuity implies that $F$ is a right inverse of $E$ and $E$ is surjective onto a dense domain. Because $EE^*$ is a projection onto the range of $E$, which is all of $\mathcal{H}=\overline{\Alg{M}\Omega}$ ($E^*$ has no kernel because $E$ is injective and because $E^*E=\id_{\mathcal{K}}$), this implies $EE^*=\id_{\mathcal{H}}$. Thus, $E$ is a unitary operator and moreover $F=E^*$. 

Next, we prove that $E$ satisfies $J_{\Omega}E=EJ_{\Xi}$. We do this by first by showing $S_{\Omega}E=ES_{\Xi}$. This latter claim follows from 
\[
S_{\Omega} E b \,\Xi
=S_{\Omega}\mathcal{E}(b)\,\Omega
=\mathcal{E}(b^*)\,\Omega
=E b^*\, \Xi
=E S_{\Xi} b \,\Xi,
\]
which holds for all $b\in\Alg{B}$, therefore showing that there is equality on their common domains. But the domains of $S_{\Omega}E$ and $ES_{\Xi}$ are also equal because $\mathcal{D}_{S_{\Omega}E}=\Alg{N}\Xi=\mathcal{D}_{S_{\Xi}}=\mathcal{D}_{E S_{\Xi}}$. Thus, $S_{\Omega}E=ES_{\Xi}$. Hence, 
\[
S_{\Omega}
=ES_{\Xi}E^*
=(EJ_{\Xi} E^*)(E \Delta_{\Xi}^{1/2} E^*)
\]
gives a polar decomposition of $S_{\Omega}$ where the conjugate-linear unitary $W:=EJ_{\Xi} E^*$ satisfies $W^*W=\id_{\mathcal{K}}$, showing that $W^* W$ projects onto $\ker(S_{\Omega})^{\perp}=\{0\}^{\perp}=\mathcal{H}$. Therefore, by the uniqueness of polar decompositions, $E J_{\Xi} E^* = J_{\Omega}$, equivalently 
$J_{\Xi} = E^* J_{\Omega}E$.

Finally, we prove that $\mathcal{E}^{\star}=\mathcal{E}^{-1}$. This follows from 
\begin{align*}
\big\<\!\!\big\<\mathcal{E}(b), a\big\>\!\!\big\>_{\omega}
&=\big\<\mathcal{E}(b)\,\Omega,J_{\Omega} a^* \,\Omega\big\> &&\mbox{ by definition of $\<\!\<\,\cdot\,,\,\cdot\,\>\!\>_{\omega}$ }\\
&=\big\<\mathcal{E}(b)\,\Omega,J_{\Omega} \mathcal{E}\big(\mathcal{E}^{-1}(a^*)\big) \,\Omega\big\> && \mbox{ since $\mathcal{E}$ is a $*$-isomorphism} \\
&=\big\<Eb\,\Xi,J_{\Omega} E \mathcal{E}^{-1}(a^*)\,\Xi\big\> && \mbox{ by definition of $E$} \\
&=\big\<b\,\Xi,E^{*} J_{\Omega} E \mathcal{E}^{-1}(a^*)\Xi\big\> && \mbox{ by definition of $E^*$} \\
&=\big\<b\,\Xi,J_{\Xi} \mathcal{E}^{-1}(a^*)\,\Xi\big\> && \mbox{ since $J_{\Xi} = E^* J_{\Omega}E$ } \\
&=\big\<\!\!\big\<b,\mathcal{E}^{-1}(a)\big\>\!\!\big\>_{\xi} &&\mbox{ by definition of $\<\!\<\,\cdot\,,\,\cdot\,\>\!\>_{\xi}$ },
\end{align*}
which holds for all $a\in\Alg{A}$ and $b\in\Alg{B}$. By the uniqueness of the Petz recovery map $\mathcal{E}^{\star}$ from Theorem~\ref{thm:PetzRecovery}, this proves $\mathcal{E}^{\star}=\mathcal{E}^{-1}$. 
\eprf

Proposition~\ref{prop:PetzInverting} allows us to interpret the Petz recovery map as a generalization of the concept of an inverse. Namely, if we let $\mathbf{I}$ denote the subcategory of $\mathbf{Q}$ whose morphisms are isomorphisms (in the categorical sense~\cite{Ma98}), then the Petz recovery map of any morphism in $\mathbf{I}$ is the inverse. This will be explained in more depth in the next section, where we show that the Petz recovery map is a noncommutative extension of Bayesian inversion.

\section{The Petz retrodiction functor}
\label{sec:retrodictionfunctor}

First, we specialize the definition of retrodiction functors from Ref.~\cite{Pa24} to the category $\mathbf{Q}$ of normal faithful states on von~Neumann algebras and state-preserving NCPU maps. We will not present the general categorical definition from Ref.~\cite{Pa24} here, as it is not needed. The remainder of this section proves that Petz retrodiction defines a retrodiction functor on $\mathbf{Q}$. Here, we use the notation $\mathbf{Q}^{\op}$ to denote the \define{opposite category} of $\mathbf{Q}$, which has the same objects of $\mathbf{Q}$, but a morphism from $(\Alg{A},\omega)$ to $(\Alg{B},\xi)$ in $\mathbf{Q}^{\op}$ is a morphism from $(\Alg{B},\xi)$ to $(\Alg{A},\omega)$ in $\mathbf{Q}$~\cite{Ma98}. As such, we will often describe morphisms in $\mathbf{Q}^{\op}$ as morphisms in $\mathbf{Q}$, but with arrows reversed. We also note that $\mathbf{Q}$ is a symmetric monoidal category. This follows from taking the von~Neumann algebra tensor product when defining the monoidal structure on objects, as reviewed in Appendix~\ref{sec:reviewModularTheory} (the tensor product of two normal faithful states is also a normal faithful state). As for morphisms, the tensor product of two state-preserving NCPU maps is also a state-preserving NCPU map. The usual associators and braiding natural isomorphisms turn $\mathbf{Q}$ into a symmetric monoidal category (see Refs.~\cites{Ma98,Yanofsky2024Monoidal} for the definition of a symmetric monoidal category). 

\bd
\label{defn:retrodiction}
A \define{retrodiction functor} on $\mathbf{Q}$ is an assignment $\mathscr{R}:\mathbf{Q}\to\mathbf{Q}^{\op}$ satisfying the following properties: 
\begin{enumerate}
\item (\define{recovery property}) $\mathscr{R}(\Alg{A},\omega)=(\Alg{A},\omega)$ for all objects $(\Alg{A},\omega)$ in $\mathbf{Q}$ and each morphism $(\Alg{B},\xi)\xrightarrow{\mathcal{E}}(\Alg{A},\omega)$ in $\mathbf{Q}$ gets sent to a morphism $(\Alg{B},\xi)\xleftarrow{\mathscr{R}(\mathcal{E})}(\Alg{A},\omega)$ in $\mathbf{Q}$
\item (\define{identity-preservation}) $\mathscr{R}\left(\id_{(\Alg{A},\omega)}\right)=\id_{(\Alg{A},\omega)}$ for all objects $(\Alg{A},\omega)$ in $\mathbf{Q}$ 
\item (\define{compositionality}) $\mathscr{R}(\mathcal{F})\circ \mathscr{R}(\mathcal{E})=\mathscr{R}(\mathcal{E}\circ\mathcal{F})$ for every pair $(\Alg{C},\zeta)\xrightarrow{\mathcal{F}}(\Alg{B},\xi)\xrightarrow{\mathcal{E}}(\Alg{A},\omega)$ of composable morphisms in $\mathbf{Q}$
\item (\define{tensoriality}) $\mathscr{R}(\mathcal{E}_{1}\otimes\mathcal{E}_{2})=\mathscr{R}(\mathcal{E}_{1})\otimes\mathscr{R}(\mathcal{E}_{2})$ for all pairs of morphisms $(\Alg{B}_1,\xi_1)\xrightarrow{\mathcal{E}_1}(\Alg{A}_1,\omega_1)$ and $(\Alg{B}_2,\xi_2)\xrightarrow{\mathcal{E}_2}(\Alg{A}_2,\omega_2)$ in $\mathbf{Q}$
\item (\define{extending inversion}) $\mathscr{R}$ is \define{inverting} in the sense that $\mathscr{R}(\mathcal{E})=\mathcal{E}^{-1}$ whenever $\mathcal{E}$ is an isomorphism in $\mathbf{Q}$ (i.e., a morphism in $\mathbf{I}$)
\item (\define{involutivity}) $\mathscr{R}$ is \define{involutive} in the sense that $\mathscr{R}\circ\mathscr{R}=\id_{\mathbf{Q}}$. 
\end{enumerate}
In other words, $\mathbf{Q}$ equipped with $\mathscr{R}$ is a 
unitary dagger symmetric monoidal category~\cites{Se07,Ka18,PaBu22} (\emph{unitary} here means that every isomorphism is unitary, i.e., the inverse is equal to its $\mathscr{R}$-adjoint, not that every morphism is implemented by a unitary element in the algebra).  
\ed

\br
Compositionality is sometimes called the \emph{chain rule} in the literature~\cite{Ts22b}. This is because the chain rule from differential geometry (multivariable calculus) sending smooth functions between pointed manifolds (open sets in Euclidean space) to the associated derivative linear transformations (Jacobian matrices) between tangent spaces is precisely describing a functor~\cites{Lee2013Manifolds,MilnorTopDiff1997,BaMu94}. 
\er

\bt
\label{thm:Petzretrodictionfunctor}
The assignment $\mathscr{R}:\mathbf{Q}\to\mathbf{Q}^{\op}$ sending a morphism $(\Alg{B},\xi)\xrightarrow{\mathcal{E}}(\Alg{A},\omega)$ to its Petz retrodiction map $(\Alg{B},\xi)\xleftarrow{\mathcal{E}^{\star}}(\Alg{A},\omega)$ defines a retrodiction functor, henceforth called the \define{Petz retrodiction functor}. 
\et

The proof of this will be split up into several propositions that isolate some of the properties of the Petz recovery map. 

\bn
\label{prop:PetzFunctorial}
The assignment $\mathscr{R}:\mathbf{Q}\to\mathbf{Q}^{\op}$ sending a morphism $(\Alg{B},\xi)\xrightarrow{\mathcal{E}}(\Alg{A},\omega)$ to its Petz retrodiction map $(\Alg{B},\xi)\xleftarrow{\mathcal{E}^{\star}}(\Alg{A},\omega)$ satisfies the recovery property and defines a functor, i.e., $\mathscr{R}$ is identity-preserving and satisfies compositionality. 
\en

\bprf
First note the fact that $\mathscr{R}$ lands in the category $\mathbf{Q}$ follows from the fact that $\mathcal{E}^{\star}$ is NCPU and state-preserving for all morphisms $(\Alg{B},\xi)\xrightarrow{\mathcal{E}}(\Alg{A},\omega)$ in $\mathbf{Q}$ by Theorem~\ref{thm:PetzRecovery}. The recovery property also follows from Theorem~\ref{thm:PetzRecovery}. To see that identity-preservation holds, let $(\Alg{A},\omega)$ be an object in $\mathbf{Q}$. By Theorem~\ref{thm:PetzRecovery}, there exists a linear map $\id_{(\Alg{A},\omega)}^{\star}:\Alg{A}\to\Alg{A}$ such that 
\[
\<\!\<a_2,a_1\>\!\>_{\omega}
=\<\!\<\id_{(\Alg{A},\omega)}(a_2),a_1\>\!\>_{\omega}
=\<\!\<a_2,\id_{(\Alg{A},\omega)}^{\star}(a_1)\>\!\>_{\omega}
\]
for all $a_1,a_2\in\Alg{A}$. 
By \emph{uniqueness} of $\id_{(\Alg{A},\omega)}^{\star}$, it follows that $\id_{(\Alg{A},\omega)}^{\star}=\id_{(\Alg{A},\omega)}$. As for the compositionality property, let $(\Alg{C},\zeta)\xrightarrow{\mathcal{F}}(\Alg{B},\xi)\xrightarrow{\mathcal{E}}(\Alg{A},\omega)$ be a composable pair of morphisms in $\mathbf{Q}$. Then, on the one hand, 
\[
\big\<\!\big\<(\mathcal{E}\circ\mathcal{F})(c),a\big\>\!\big\>_{\omega}
=\big\<\!\big\<c,(\mathcal{E}\circ\mathcal{F})^{\star}(a)\big\>\!\big\>_{\zeta}
\]
holds for all $a\in\Alg{A}$ and $c\in\Alg{C}$ by Theorem~\ref{thm:PetzRecovery} applied to $\mathcal{E}\circ\mathcal{F}$. On the other hand, 
\[
\Big\<\!\!\Big\<\mathcal{E}\big(\mathcal{F}(c)\big),a\Big\>\!\!\Big\>_{\omega}=\big\<\!\big\<\mathcal{F}(c),\mathcal{E}^{\star}(a)\big\>\!\big\>_{\xi}
=\Big\<\!\!\Big\<c,\mathcal{F}^{\star}\big(\mathcal{E}^{\star}(a)\big)\Big\>\!\!\Big\>_{\zeta}
\]
holds for all $a\in\Alg{A}$ and $c\in\Alg{C}$ by Theorem~\ref{thm:PetzRecovery} applied to $\mathcal{E}$ and $\mathcal{F}$, respectively. Since the left-hand-sides of these two expressions are equal by definition of composition, uniqueness from Theorem~\ref{thm:PetzRecovery} implies that $\mathcal{F}^{\star}\circ\mathcal{E}^{\star}=(\mathcal{E}\circ\mathcal{F})^{\star}$, which proves compositionality of $\mathscr{R}$. 
\eprf

\bn
\label{prop:PetzTensorial}
The assignment $\mathscr{R}:\mathbf{Q}\to\mathbf{Q}^{\op}$ sending a morphism $(\Alg{B},\xi)\xrightarrow{\mathcal{E}}(\Alg{A},\omega)$ to its Petz retrodiction map $(\Alg{B},\xi)\xleftarrow{\mathcal{E}^{\star}}(\Alg{A},\omega)$ is tensorial. 
\en

\bprf
Let $(\Alg{B}_1,\xi_1)\xrightarrow{\mathcal{E}_1}(\Alg{A}_1,\omega_1)$ and $(\Alg{B}_2,\xi_2)\xrightarrow{\mathcal{E}_2}(\Alg{A}_2,\omega_2)$ be a pair of morphisms in $\mathbf{Q}$. 
Then, setting $a=a_1\otimes a_2$ and $b=b_1\otimes b_2$, we have 
\begin{align*}
\big\<\!\big\<b,(\mathcal{E}_{1}\otimes\mathcal{E}_{2})^{\star}(a)\big\>\!\big\>_{\xi_1\otimes\xi_2}
&=\big\<\!\big\<(\mathcal{E}_{1}\otimes\mathcal{E}_{2})(b_1\otimes b_2),a_1\otimes a_2)\big\>\!\big\>_{\omega_1\otimes\omega_2} &&\mbox{ by Definition~\ref{defn:PetzRecoveryMap}} \\
&=\big\<\!\big\<\mathcal{E}_{1}(b_1)\otimes\mathcal{E}_{2}(b_2),a_1\otimes a_2)\big\>\!\big\>_{\omega_1\otimes\omega_2} &&\mbox{ by Definition of $\mathcal{E}_{1}\otimes\mathcal{E}_{2}$} \\
&=\big\<\!\big\<\mathcal{E}_{1}(b_1),a_1\big\>\!\big\>_{\omega_{1}}\big\<\!\big\<\mathcal{E}_{2}(b_2),a_2\big\>\!\big\>_{\omega_{2}} &&\mbox{ by Lemma~\ref{lem:PetzIPtensor}} \\
&=\big\<\!\big\<b_1,\mathcal{E}_{1}^{\star}(a_1)\big\>\!\big\>_{\xi_{1}}\big\<\!\big\<b_2,\mathcal{E}_{2}^{\star}(a_2)\big\>\!\big\>_{\xi_{2}} &&\mbox{ by Definition~\ref{defn:PetzRecoveryMap}} \\
&=\big\<\!\big\<b_1\otimes b_2,\mathcal{E}_{1}^{\star}(a_1)\otimes\mathcal{E}_{2}^{\star}(a_2)\big\>\!\big\>_{\xi_1\otimes\xi_2} &&\mbox{ by Lemma~\ref{lem:PetzIPtensor}} 
\end{align*}
for all $a_1\in\Alg{A}_{1},a_{2}\in\Alg{A}_{2},b_{1}\in\Alg{B}_{1},b_{2}\in\Alg{B}_{2}$. Hence, since the 
KMS inner product is sesquilinear, this result extends to linear combinations, i.e., 
\be
\label{eqn:tensorCPalgebraicproduct}
\big\<\!\big\<b,(\mathcal{E}_1\otimes\mathcal{E}_2)^{\star}(a)\big\>\!\big\>_{\xi_1\otimes\xi_2}=\big\<\!\big\<b,(\mathcal{E}_1^{\star}\otimes\mathcal{E}_2^{\star})(a)\big\>\!\big\>_{\xi_1\otimes\xi_2}
\ee
for all $a\in\Alg{A}_{1}\otimes\Alg{A}_{2}$ and $b\in\Alg{B}_{1}\otimes\Alg{B}_{2}$, where recall that $\otimes$ here denotes the algebraic tensor product (given by finite linear combinations of tensors). We next extend this identity beyond algebraic tensor products. For this, first fix $a\in\Alg{A}_1\otimes\Alg{A}_{2}$ and let $b:B\to\Alg{B}_{1}\otimes\Alg{B}_{2}$ be a net that converges to $b_{\infty}:=\lim_{\beta \in B}b_{\beta}\in\Alg{B}_{1}\overline{\otimes}\Alg{B}_{2}$ in the weak operator topology. Then
\begin{align*}
\big\<\!\big\<b_{\infty},(\mathcal{E}_{1}\otimes\mathcal{E}_{2})^{\star}(a)\big\>\!\big\>_{\xi_1\otimes\xi_2}&=\lim_{\beta\in B} \big\<\!\big\<b_{\beta},(\mathcal{E}_{1}\otimes\mathcal{E}_{2})^{\star}(a)\big\>\!\big\>_{\xi_1\otimes\xi_2} && \mbox{ by Lemma~\ref{limilemm} }\\
&=\lim_{\beta\in B} \big\<\!\big\<b_{\beta},(\mathcal{E}_{1}^{\star}\otimes\mathcal{E}_{2}^{\star})(a)\big\>\!\big\>_{\xi_1\otimes\xi_2} && \mbox{ by~\eqref{eqn:tensorCPalgebraicproduct} } \\
&=\big\<\!\big\<b_{\infty},(\mathcal{E}_{1}^{\star}\otimes\mathcal{E}_{2}^{\star})(a)\big\>\!\big\>_{\xi_1\otimes\xi_2} && \mbox{ by Lemma~\ref{limilemm}.}
\end{align*} 
Since this is true for all $a\in\Alg{A}_{1}\otimes\Alg{A}_{2}$ and for all $b\in\Alg{B}_{1}\overline{\otimes}\Alg{B}_{2}$, we conclude that $\mathcal{E}_{1}^{\star}\otimes\mathcal{E}_{2}^{\star}=(\mathcal{E}_{1}\otimes\mathcal{E}_{2})^{\star}$ on $\Alg{A}_{1}\otimes\Alg{A}_{2}$. In the remainder of the proof, we show that the equality extends to all of $\Alg{A}_{1}\overline{\otimes}\Alg{A}_{2}$. So now let $a:A\to\Alg{A}_{1}\otimes\Alg{A}_{2}$ be a net that converges to $a_{\infty}:=\lim_{\alpha\in A}a_{\alpha}\in\Alg{A}_{1}\overline{\otimes}\Alg{A}_{2}$ in the weak operator topology. Then, using what we've shown thus far, we obtain (dropping the subscript $\xi_{1}\otimes\xi_{2}$ from the 
KMS inner product to avoid clutter)
\begin{align*}
\big\<\!\big\<b,(\mathcal{E}_{1}\otimes\mathcal{E}_{2})^{\star}(a_{\infty})\big\>\!\big\>&= \Big\<\!\!\Big\<b, \lim_{\alpha\in A}(\mathcal{E}_{1}\otimes\mathcal{E}_{2})^{\star}(a_{\alpha})\Big\>\!\!\Big\>&&\mbox{ because $(\mathcal{E}_{1}\otimes\mathcal{E}_{2})^{\star}$ is normal by Theorem~\ref{thm:PetzRecovery} }\\
& && \mbox{ and hence weakly continuous }\\
&=\lim_{\alpha\in A} \big\<\!\big\<b,(\mathcal{E}_{1}\otimes\mathcal{E}_{2})^{\star}(a_{\alpha})\big\>\!\big\>  && \mbox{ by Lemma~\ref{limilemm}} \\
&=\lim_{\alpha\in A}\big\<\!\big\<b,(\mathcal{E}_{1}^{\star}\otimes\mathcal{E}_{2}^{\star})(a_{\alpha})\big\>\!\big\> && \mbox{ since $\mathcal{E}_{1}^{\star}\otimes\mathcal{E}_{2}^{\star}=(\mathcal{E}_{1}\otimes\mathcal{E}_{2})^{\star}$ on $\Alg{A}_{1}\otimes\Alg{A}_{2}$} \\
&=\Big\<\!\!\Big\<b,\lim_{\alpha\in A}(\mathcal{E}_{1}^{\star}\otimes\mathcal{E}_{2}^{\star})(a_{\alpha})\Big\>\!\!\Big\> && \mbox{ by Lemma~\ref{limilemm}} \\
&=\big\<\!\big\<b,(\mathcal{E}_{1}^{\star}\otimes\mathcal{E}_{2}^{\star})(a_{\infty})\big\>\!\big\> && \mbox{ since $\mathcal{E}_{1}^{\star}\otimes\mathcal{E}_{2}^{\star}$ is normal } \\
& && \mbox{ and hence weakly continuous }
\end{align*}
for all $b\in\Alg{B}_{1}\overline{\otimes}\Alg{B}_{2}$. Thus, $\mathcal{E}_{1}^{\star}\otimes\mathcal{E}_{2}^{\star}=(\mathcal{E}_{1}\otimes\mathcal{E}_{2})^{\star}$, proving that the Petz retrodiction assignment is tensorial on $\mathbf{Q}$. 
\eprf

\bprf[Proof of Theorem~\ref{thm:Petzretrodictionfunctor}]
The fact that $\mathscr{R}$ satisfies the first five conditions in Definition~\ref{defn:retrodiction} was proved in Propositions~\ref{prop:PetzFunctorial},~\ref{prop:PetzTensorial}, and~\ref{prop:PetzInverting}. What remains to show is that $\mathscr{R}$ is involutive, i.e., $(\mathcal{E}^{\star})^{\star}=\mathcal{E}$ for all $(\Alg{B},\xi)\xrightarrow{\mathcal{E}}(\Alg{A},\omega)$ in $\mathbf{Q}$. To see this, note that 
\[
\<\!\<\mathcal{E}(b),a\>\!\>_{\omega}
=\<\!\<b,\mathcal{E}^{\star}(a)\>\!\>_{\xi}
=\overline{\<\!\<\mathcal{E}^{\star}(a),b\>\!\>_{\xi}}
=\overline{\<\!\<a,(\mathcal{E}^{\star})^{\star}(b)\>\!\>_{\omega}}
=\<\!\<(\mathcal{E}^{\star})^{\star}(b),a\>\!\>_{\omega}
\]
for all $a\in\Alg{A}$ and $b\in\Alg{B}$. Since $\<\!\<\,\cdot\,,\,\cdot\,\>\!\>_{\omega}$ is nondegenerate, this implies $(\mathcal{E}^{\star})^{\star}=\mathcal{E}$. 
\eprf

\section{Bayesian inference on commutative von Neumann algebras}
\label{sec:commutativevNAlg}

In this section, we will prove that Petz retrodiction restricts to Bayesian inversion on commutative von~Neumann algebras. 
We assume some familiarity with measure theory~\cites{Fo07,Rudin87,GaPa18,PaRu19,FrV4,Ka21}. References on Gelfand duality might also be useful for a broader perspective~\cites{FuJa13,Pa17,Pavlov2022,FritzLorenzin2026}, though are not necessary to follow this section. 
Let $\mathbf{C}$ denote the symmetric monoidal subcategory of $\mathbf{Q}$ consisting of the objects $(\Alg{A},\omega)$ of $\mathbf{Q}$, where $\Alg{A}$ is commutative. 
Despite introducing the notation $\mathbf{C}$ for the category of ``classical'' (i.e., commutative) systems, we will expand on the Petz recovery map on commutative von~Neumann algebras that come from standard Borel spaces in order to make a cleaner connection to probability theory, Markov kernels, conditional probabilities, and Bayes' rule. In fact, we will first begin with the significantly simpler case where the von~Neumann algebras are finite-dimensional in order to motivate the remainder of the section and to make the comments in the opening paragraph of Section~\ref{sec:intro} more precise. Although we could define the Petz recovery map for any commutative von~Neumann algebra, the particular classes of examples we study here will have a cleaner connection to the typical Bayes' rule expressed in terms of transition matrices and/or probability distributions. 

Let $X$ and $Y$ be finite sets (equipped with the discrete measurable structure and the counting measure). Let $\Alg{A}=\C^{X}$ and $\Alg{B}=\C^{Y}$ denote their von~Neumann algebras of $\C$-valued functions on $X$ and $Y$, respectively (by choosing an ordering of the elements of $X$ and $Y$, one can view $\C^{X}$ and $\C^{Y}$ as diagonal matrices inside $\matr_{m}$ and $\matr_{n}$, respectively, where $m=\# X$ and $n=\# Y$ are the cardinalities of $X$ and $Y$, respectively). For each $x\in X$, let $\delta_{x}\in\Alg{A}=\C^{X}$ denote the function given by the Kronecker delta at $x$, i.e., $\delta_{x}(\tilde{x}):=\delta_{x\tilde{x}}$ for $\tilde{x}\in X$, and $\delta_{x\tilde{x}}=1$ if $x=\tilde{x}$ and $0$ otherwise. Let $p$ be a nowhere-vanishing probability distribution on $X$, i.e., a collection of numbers $\{p_{x}\}$ such that $p_{x}>0$ for all $x\in X$ and $\sum_{x\in X}p_{x}=1$. Let $f:X\to Y$ be a stochastic matrix, i.e., a collection of numbers $\{f_{yx}\}$ such that $f_{yx}\ge0$ for all $x\in X$ and $y\in Y$ and such that $\sum_{y\in Y}f_{yx}=1$ for all $x\in X$. Let $q:=p\circ f$ be the probability distribution on $Y$ given by $(p\circ f)_{y}=\sum_{x\in X}f_{yx}p_{x}$. Suppose further that $q$ is nowhere vanishing. Let $\omega:\Alg{A}\to\C$ be the expectation value functional associated with $p$, i.e., $\omega(\varphi)=\sum_{x\in X}p_{x}\varphi(x)$ for every $\varphi\in\Alg{A}=\C^{X}$, and similarly let $\xi$ be the expectation value functional associated with $q$. Finally, let $\mathcal{F}:\Alg{B}\to\Alg{A}$ be the linear map sending $\psi\in\Alg{B}=\C^{Y}$ to the function $\mathcal{F}(\psi)\in\Alg{A}=\C^{X}$ defined on an element $x\in X$ by $\mathcal{F}(\psi)(x):=\sum_{y\in Y} f_{yx} \psi(y)$. Then $(\Alg{B},\xi)\xrightarrow{\mathcal{F}}(\Alg{A},\omega)$ defines a morphism in $\mathbf{Q}$. Moreover, 
\[
\<\!\<\varphi_{1},\varphi_{2}\>\!\>_{\omega}=\sum_{x\in X}\overline{\varphi_{1}}(x)\varphi_{2}(x) \, p_{x}
\]
for all $\varphi_1,\varphi_2\in\Alg{A}=\C^{X}$, and similarly for $\<\!\<\,\cdot\,,\,\cdot\,\>\!\>_{\xi}$ in terms of $q$. Therefore, the Petz recovery map associated with $\mathcal{F}$ is the morphism $\mathcal{F}^{\star}:(\Alg{A},\omega)\to (\Alg{B},\xi)$ in $\mathbf{Q}$, which is induced by the stochastic matrix $\{f^{\star}_{xy}:=\mathcal{F}^{\star}(\delta_{x})(y)\}$, uniquely determined by $\<\!\<\mathcal{F}(\psi),\varphi\>\!\>_{\omega}=\<\!\<\psi,\mathcal{F}^{\star}(\varphi)\>\!\>_{\xi}$. This yields 
\[
\sum_{x\in X} \left(\sum_{y\in Y} \overline{f_{yx} \psi(y)} \right) \varphi(x) \, p_{x} =\sum_{y\in Y}\overline{\psi(y)} \left( \sum_{x\in X} f^{\star}_{xy} \varphi(x) \right) \, q_{y}
\]
for all $\varphi\in\Alg{A}=\C^{X}$ and $\psi\in\Alg{B}=\C^{Y}$. By setting $\varphi=\delta_{x}$ and $\psi=\delta_{y}$, this condition gives 
\[
f_{yx}p_{x}=f^{\star}_{xy} q_{y} \quad\iff\quad f^{\star}_{xy}=\frac{f_{yx}p_{x}}{q_{y}}. 
\]
Interpreting $f_{yx}$ as the conditional probability of $y$ given $x$, written as $\mathbb{P}(y|x)$ in Section~\ref{sec:intro}, and similarly for the other terms, this is precisely Bayes' rule, so that $f^{\star}_{xy}$ is the conditional probability of $x$ given $y$~\cites{Bayes1763,Pearl88}, written as $\mathbb{P}(x|y)$ in Section~\ref{sec:intro}. 

In the remainder of this section, we will extend the previous discussion from finite sets to standard Borel spaces. This takes a surprising amount of work due to the fact that we will describe how the Petz recovery map on function spaces can be represented by Markov kernels on the underlying measure spaces, which are constructed via regular conditional probabilities/consistent disintegrations. As such, we first recall some definitions and facts from measure theory. The reader familiar with the definitions can jump straight to Proposition~\ref{prop:cvnPetz}, the paragraph preceding it, and the discussion that follows. 

A measurable space $X$ will often be written with its sigma algebra $\Sigma_{X}$ as the pair $(X,\Sigma_{X})$. A measure space will be written as a triple $(X,\Sigma_{X},\mu)$, where $\mu$ is the measure. Given a measurable function $\varphi:X\to[0,\infty)$, its \define{essential supremum} is 
\[
\esssup(\varphi):=\inf\Big\{a\ge 0 \,:\, \mu\Big(\varphi^{-1}\big( (a,\infty)\big)\Big) \Big\},
\]
where $\varphi^{-1}\big( (a,\infty)\big):=\big\{x\in X\,:\, \varphi(x)>a\big\}\in\Sigma_{X}$. Now, given any measurable function $\varphi: X\to\C$, define $\lVert \varphi\rVert_{\infty}=\esssup(|\varphi|)$, which is called the \define{$L^{\infty}$ norm} on measurable functions (we note that the dependence of $\mu$ is hidden from the notation). A measurable function $\varphi$ with finite $L^{\infty}$ norm is called \define{essentially bounded}. 
Let $L^{\infty}(X,\mu)$ denote the Banach space of \define{$\mu$-a.e.\ equivalence classes} of essentially bounded measurable functions from $X$ to $\C$, where the equivalence is defined by $\varphi\sim\psi$ iff $\varphi$ and $\psi$ are equal $\mu$-almost everywhere. 
Despite the fact that $L^{\infty}(X,\mu)$ (and later $L^{2}(X,\mu)$) are spaces of equivalence classes of measurable functions, we often call elements of $L^{\infty}(X,\mu)$ by their function representatives.

In fact, $L^{\infty}(X,\mu)$ is a von~Neumann algebra, where complex conjugation is the $*$ operation and multiplication is given by the multiplication of functions. This can be seen by viewing $L^{\infty}(X,\mu)$ as a subalgebra of $\mathcal{B}(\mathcal{H})$, where $\mathcal{H}=L^{2}(X,\mu)$, where $L^{2}(X,\mu)$ is the Hilbert space of \define{square-normalizable} measurable functions $\varphi:X\to\C$, i.e., $\int_{X}|\varphi|^2\,d\mu<\infty$, and where each element $\varphi\in L^{\infty}(X,\mu)$ is viewed as a multiplication operator $M_{\varphi}$. Namely, given also $\psi\in L^{2}(X,\mu)$, the function $M_{\varphi}(\psi)$, given by $\big(M_{\varphi}(\psi)\big)(x):=\varphi(x)\psi(x)$, is still in $L^{2}(X,\mu)$~\cite[Exercise 6.14]{Fo07}. Moreover, $L^{\infty}(X,\mu)$ is closed under the weak operator topology inside $\mathcal{B}(\mathcal{H})$~\cite[Chapter 7 Proposition 1]{Dix81}. 

Given a measure space $(X,\Sigma_{X},\mu)$ and given another measure $p$ on $(X,\Sigma_{X})$, the measure $p$ is said to be \define{absolutely continuous} with respect to $\mu$, written as $p\ll\mu$, iff for any $A\in\Sigma_{X}$ such that $\mu(A)=0$, then $p(A)=0$. If $p$ is a probability measure, the condition $p\ll\mu$ implies that there exists a nonnegative measurable function $\frac{dp}{d\mu}$ on $(X,\Sigma_{X})$, called the \define{Radon--Nikodym derivative} of $p$ with respect to $\mu$, such that 
\[
\int_{X}\frac{dp}{d\mu}\,d\mu=1
\qquad\text{ and }\qquad 
p(A)=\int_{A}\frac{dp}{d\mu}\,d\mu
\]
for all $A\in\Sigma_{X}$~\cite[Section 3.8]{Fo07} (although we do not go into the details here, this notion of Radon--Nikodym derivative is a special case of the one from Appendix~\ref{sec:reviewModularTheory}). 
Thus, $\frac{dp}{d\mu}$ is called the \define{probability density function} associated with $p$ and $\mu$ (it is uniquely specified up to a $\mu$-null set). 
The condition $p\ll\mu$ also guarantees that if $\varphi\in L^{\infty}(X,\mu)$, then 
\[
\int_{X}|\varphi|\,dp
=\int_{X}|\varphi| \frac{dp}{d\mu} \,d\mu
\le \lVert\varphi\rVert_{\infty} \left(\int_{X} \frac{dp}{d\mu} \,d\mu\right)
=\lVert\varphi\rVert_{\infty} \
< \infty
\]
by~\cite[Theorem 6.8]{Fo07}. 
This fact lets us define the functional $\omega:L^{\infty}(X,\mu)\to\C$ by 
\be
\label{eqn:fromptoomega}
\omega(\varphi):=\int_{X}\varphi\,dp=\int_{X}\varphi \frac{dp}{d\mu}\,d\mu
\ee 
for all $\varphi\in L^{\infty}(X,\mu)$.  
Then $\omega$ is a normal state, where normality follows from the monotone convergence theorem~\cite{Fo07}. The probability measure $p$ is \define{faithful} with respect to $\mu$ iff, in addition to $p\ll \mu$, it is also the case that $\mu\ll p$. This is equivalent to the state $\omega$ from~\eqref{eqn:fromptoomega} being faithful, which recall means that if $\varphi\in L^{\infty}(X,\mu)$ satisfies $\omega(\varphi^* \varphi)=0$, then $\varphi=0$ as an element of $L^{\infty}(X,\mu)$, or if we thought of $\varphi$ as a function representative, then this means there exists a $\mu$-null set $N\in\Sigma_{X}$ such that $\varphi(x)=0$ for all $x\in X\setminus N$.%
\footnote{
Suppose $\mu\ll p$ and let $\varphi\in L^{\infty}(X,\mu)$ be such that $0=\omega(\varphi^*\varphi)=\int_{X}|\varphi|^2\,dp$. This implies there exists a $p$-null set $N\in\Sigma_{X}$ such that $\varphi(x)=0$ for all $x\in X\setminus N$. Since $\mu\ll p$, we have $\mu(N)=0$ so that $\varphi=0$ in $L^{\infty}(X,\mu)$. Conversely, suppose $\omega$ is faithful. Suppose that $N\in\Sigma_{X}$ satisfies $p(N)=0$ and let $\varphi=\chi_{N}$ be the indicator function on $N$. Then $\omega(\varphi^*\varphi)=\int_{N}dp=p(N)=0$, which implies there exists a $\mu$-null set $M\in\Sigma_{X}$ such that $\varphi(x)=0$ for all $x\in X\setminus M$. But $\varphi(x)=1$ if and only if $x\in N$, so that $N\subseteq X\setminus (X\setminus M)=M$. Thus, $\mu(N)\le\mu(M)=0$ implies $\mu(N)=0$, which proves $\mu\ll p$. 
}  
Thus, if we set $\Alg{A}:=L^{\infty}(X,\mu)$, then $(\Alg{A},\omega)$ defines an object in $\mathbf{Q}$. For reference, we note that the 
KMS inner product becomes the ordinary $p$-weighted $L^2$ inner product
\be
\label{eqn:ConnesPetzforCvNAlg}
\<\!\<\varphi,\psi\>\!\>_{\omega}
=\<M_{\varphi}\Omega,J_{\Omega}M_{\psi}^{*}\Omega\>
=\int_{X}\overline{\varphi}\psi\,dp
\ee
for $\varphi,\psi\in L^{\infty}(X,\mu)$. 

Next, let $(Y,\Sigma_{Y})$ be another measurable space, and let $f:X\to Y$ be a \define{transition kernel}, i.e., a function $f:X\times\Sigma_{Y}\to[0,\infty]$ such that $\Sigma_{Y}\ni B\mapsto f_{x}(B):=f(x,B)$ is a measure for each $x\in X$ and $X\ni x\mapsto f_{x}(B)$ is a measurable function for each $B\in\Sigma_{Y}$. A transition kernel is called a \define{Markov kernel} whenever $f_{x}$ is a probability measure for all $x\in X$. Note that a measurable function from $X$ to $Y$ can be viewed as a Markov kernel by sending $x\in X$ to the Dirac delta measure on $f(x)$. Another example is a measure $p$ on $(X,\Sigma_{X})$, which can be viewed as a transition kernel $\{\bullet\}\to X$, where $\{\bullet\}$ is a single element set with the unique measurable structure (it is a Markov kernel precisely when $p$ is a probability measure). 

Transition kernels can be composed as follows. If $(Z,\Sigma_{Z})$ is another measurable space, and if $g:Y\to Z$ is another transition kernel, then the \define{composite} $g\circ f$ is the transition kernel defined by 
\[
X\times \Sigma_{Z}\ni (x,C)\mapsto (g\circ f)_{x}(C):=\int_{Y}g_{y}(C)\,df_{x}(y)\,.
\]
Given a transition kernel $f:X\to Y$ and a measure $\mu:\{\bullet\}\to X$, the \define{pushforward measure} of $\mu$ along $f$ is $f\circ \mu$. When $f$ is a measurable function, this agrees with the pushforward measure along a measurable function given by 
\[
(f\circ \mu)(B)=\int_{X} f_{x}(B)\, d\mu(x)
\]
for $B\in\Sigma_{Y}$, which is often denoted by $\mu\circ f^{-1}$ (this notation is no longer appropriate when $f$ is a transition kernel, which is why we use $f\circ \mu$). 

Given measure spaces $(X,\Sigma_{X},\mu)$ and $(Y,\Sigma_{Y},\nu)$, a transition kernel $f:X\to Y$ is said to \define{preserve $\nu$-null sets} iff $f\circ\mu\ll\nu$, i.e., if $\nu(N)=0$ implies $(f\circ\mu)(N)=0$.  
A Markov kernel $f:X\to Y$ that preserves $\nu$-null sets induces an NCPU map $\mathcal{F}:L^{\infty}(Y,\nu)\to L^{\infty}(X,\mu)$ defined by
\be
\label{eqn:MarkovOperator}
\mathcal{F}(\psi)(x):=\int_{Y} \psi \,d f_{x}.
\ee  
The fact that $\mathcal{F}$ is positive and unital is immediate from the definition, normality follows from the monotone convergence theorem, and complete positivity follows from Stinespring's theorem~\cite{St55}. 
If $p$ is, in addition, a probability measure on $(X,\Sigma_{X})$ satisfying $p\ll\mu$, then $q:=f\circ p\ll f\circ \mu \ll \nu$. The functional $\xi:L^{\infty}(Y,\nu)\to\C$ associated with the pushforward probability measure $q=f\circ p$ is 
\be
\label{eqn:fromqtoxi}
\xi(\psi)=\int_{Y} \psi \,dq
\ee
for $\psi\in L^{\infty}(Y,\nu)$. If $\omega$ is the state associated with $p$, then $\xi=\omega\circ\mathcal{F}$. If we assume $q$ is faithful with respect to $\nu$, and if we set $\Alg{A}=L^{\infty}(X,\mu)$ and $\Alg{B}=L^{\infty}(Y,\nu)$, then $(\Alg{B},\xi)\xrightarrow{\mathcal{F}}(\Alg{A},\omega)$ defines a morphism in $\mathbf{Q}$. 

In what follows, we will slowly build a transition kernel $f^{\star}:Y\to X$ whose associated NCPU map $\mathcal{F}^{\star}:\Alg{A}\to\Alg{B}$ is the Petz recovery map of $(\Alg{B},\xi)\xrightarrow{\mathcal{F}}(\Alg{A},\omega)$. This will occur through an intermediate step of building another measure associated with $f$ and $p$ from the previous paragraph. This is the \define{joint probability measure} $f\star p$ on $(X\times Y,\Sigma_{X\times Y})$ uniquely specified by 
\be
\label{eqn:jointpm}
(f\star p)(A\times B):=\int_{A}f_{x}(B)\,dp(x)
\ee
with $A\in\Sigma_{X}$ and $B\in\Sigma_{Y}$. 
Equivalently, since the diagonal map $\Delta:X\to X\times X$ sending $x\in X$ to $(x,x)$ is measurable, $f\star p$ can be defined as $f\star p=(\id_{X}\times f)\circ\Delta\circ p$, where the product $\times$ of transition kernels is defined via the pointwise product of measures~\cite{ChJa18}. 

Two transition kernels $f,f':X\to Y$ are said to be \define{$\mu$-a.e.\ equivalent}, written as $f\aeequals{\mu}f'$, iff for each $B\in\Sigma_{Y}$, there exists a $N_{B}\in\Sigma_{X}$ such that $\mu(N_{B})=0$ and 
\[
f_{x}(B)=f'_{x}(B)\quad\mbox{ for all $x\in X\setminus N_{B}$}. 
\]
Two such $\mu$-a.e.\ equivalent transition kernels satisfy $f\circ p=f'\circ p$ and $f\star p=f'\star p$ whenever $p$ is a measure on $(X,\Sigma_{X})$ satisfying $p\ll \mu$. The latter claim implies the former claim by postcomposing with $\pi_{Y}$. The latter claim follows from the calculation
\[
\begin{split}
(f'\star p)(A\times B)
&=\int_{A}f'_{x}(B)\,dp(x)
=\int_{A}f'_{x}(B)\frac{dp}{d\mu}\,d\mu(x)\\
&=\int_{A}f_{x}(B)\frac{dp}{d\mu}\,d\mu(x)
=\int_{A}f_{x}(B)\,dp(x)
=(f\star p)(A\times B)
\end{split}
\]
for all $A\in\Sigma_{X}$ and $B\in\Sigma_{Y}$. We will later simplify this definition of a.e.\ equivalence when dealing with standard Borel spaces. 

\bd
\label{defn:RCP}
Let $(W,\Sigma_{W},\kappa)$ and $(Z,\Sigma_{Z},\gamma)$ be measure spaces, and let $h:W\to Z$ be a measurable function such that $h\circ \kappa=\gamma$. A \define{regular conditional probability} for $(\kappa,h)$ is a transition kernel $k:Z\to W$ for which there exists a $\gamma$-null set $N\in\Sigma_{Z}$ such that $k_{z}$ is a probability measure for all $z\in Z\setminus N$ and 
\be
\label{eqn:RCPmeasurablesets}
\kappa\big( A\cap h^{-1}(B)\big) = \int_{B}k_{z}(A)\,d\gamma(z)
\ee
for all $A\in\Sigma_{W}$ and $B\in\Sigma_{Z}$. 
\ed

A regular conditional probability $k$ as in Definition~\ref{defn:RCP} is also sometimes called a \define{disintegration of $\kappa$ over $\gamma$ consistent with $h$} because such a transition kernel $k$ satisfies the conditions for a regular conditional probability if and only if it satisfies the two conditions $k\circ\gamma=\kappa$ and  $h\circ k\aeequals{\gamma}\id_{Z}$~\cite[Appendix A]{PaRu19}. In brief, the relationship between these two definitions can be seen by taking $B=Z$ in~\eqref{eqn:RCPmeasurablesets} to arrive at $k\circ\gamma=\kappa$, whereas taking $A=h^{-1}(B)$ in~\eqref{eqn:RCPmeasurablesets} leads to $h\circ k\aeequals{\gamma}\id_{Z}$. 

We next specialize to a class of measure spaces in which regular conditional probabilities (consistent disintegrations) are guaranteed to exist. A \define{Polish space} is a separable and completely metrizable topological space. A \define{standard Borel space} is the measurable space associated with a Polish space and its sigma algebra of Borel sets. 
The reason we focus on standard Borel spaces is because they admit a notion of disintegration~\cite{Rokhlin1949} (see also Refs.~\cites{GaPa18,BaJaMa2015,CDDG17,Ka21} and~\cite[Section 452E]{FrV4}). 

\begin{theorem}[The disintegration theorem]
\label{thm:disintegration}
Let $(W,\Sigma_{W},\kappa)$ and $(Z,\Sigma_{Z},\gamma)$ be standard Borel spaces with $\sigma$-finite measures (for example, $\kappa$ and $\gamma$ could be probability measures), and let $h:W\to Z$ be a Borel measurable function. If $h \circ \kappa\ll\gamma$, then there exists a $\sigma$-finite disintegration $\{k_z : z\in Z\}$ of $\kappa$ over $\gamma$ consistent with $h$. Moreover, this disintegration is unique in the sense that if $\{k'_z : z\in Z\}$ is another $\sigma$-finite disintegration of $\kappa$ over $\gamma$ consistent with $h$, then $k\aeequals{\gamma}k'$. 
\end{theorem}

We remark that for standard Borel spaces, the notion of a.e.\ equivalence between transition kernels simplifies. Namely, suppose that $k,k':Z\to W$ are two transition kernels such that there exist $N,N'\in\Sigma_{Z}$ with $\gamma(N)=0=\gamma(N')$ such that $k_{z}$ is a probability measure for all $z\in Z\setminus N$ and $k'_{z}$ is a probability measure for all $z\in Z\setminus N'$ (this holds for disintegrations). Then $k$ and $k'$ are $\gamma$-a.e.\ equivalent if and only if there exists a measurable set $M\in\Sigma_{Z}$ with $\gamma(M)=0$ such that $k'_z = k_z$ for all $z\in Z\setminus M$ (notice that this condition is a-priori stronger than the previous definition). This holds because of  the separable assumption, which guarantees a countable base for $\Sigma_{W}$, and can be proved using the Dynkin $\pi$-$\lambda$ theorem~\cite[Theorem 3.2]{Billingsley1995},~\cite{Sengupta2006}.%
\footnote{
\label{footnote:Dynkin}
Here is a sketch of the proof. Let $\{B_{n}\}$ be an at most countable \define{$\pi$-system}, which means that it is closed under finite intersections, that generates $\Sigma_{W}$ (existence is guaranteed by the assumption that $(W,\Sigma_{W})$ is standard Borel). For each $n$, let $N_{n}\in\Sigma_{Z}$ satisfy $\gamma(N_{n})=0$ and $k_{z}(B_{n})=k'_{z}(B_{n})$ for all $z\in Z\setminus N_{n}$. Set $M:=N\cup N'\cup\bigcup_{n} N_{n}$. Then $\gamma(M)=0$ since the union is countable. The claim is that $k_{z}=k'_{z}$ for all $z\in Z\setminus M$. To see this, define $\Lambda:=\{B\in \Sigma_{W} \,:\, k_{z}(B)=k'_{z}(B)\;\forall\,z\in Z\setminus M\}$. Then, $\Lambda$ is a \define{$\lambda$-system}, which means $\varnothing\in\Lambda$, if $B\in\Lambda$ then $B^{c}:=W\setminus B\in\Lambda$, and $\Lambda$ is closed under countable disjoint unions. Since $\{B_{n}\}$ is a $\pi$-system generating $\Sigma_{W}$, the Dynkin $\pi$-$\lambda$ theorem guarantees that $\Sigma_{W}\subseteq\Lambda$, which proves the claim.}
An important consequence of this is that, when $\kappa$ and $\gamma$ are probability measures, the disintegration $k:Z\to W$ from Theorem~\ref{thm:disintegration} is $\gamma$-a.e.\ equivalent to a Markov kernel, and we will henceforth refer to it as such. This is done, for example, by taking that $\gamma$-null set $N$ such that $k_{z}$ is a probability measure for all $z\in Z\setminus N$, and defining $k':Z\to W$ to be $k'_{z}:=\chi_{Z\setminus N}(z) k_{z}+\chi_{N}(z)\varsigma$ for all $z\in Z$, where $\varsigma$ is any fixed probability measure on $(W,\Sigma_{W})$. Then $k'$ is indeed a Markov kernel that satisfies the same defining properties of a consistent disintegration as $k$. 

We will apply the disintegration theorem for a particular setup in the construction of a Bayesian inverse to a Markov kernel~\cite{CuSt14}. Let $(X,\Sigma_{X},\mu)$ and $(Y,\Sigma_{Y},\nu)$ be standard Borel spaces equipped with measures. Let $p$ and $q$ be probability measures on $(X,\Sigma_{X})$ and $(Y,\Sigma_{Y}$) such that $p\ll\mu$ and $q\ll\nu$ and such that $p$ and $q$ are faithful with respect to $\mu$ and $\nu$, respectively. Let $f:X\to Y$ be a Markov kernel such that $f\circ p=q$. Then $(X\times Y,\Sigma_{X\times Y})$ is a standard Borel space~\cite[Lemma 1.2]{Ka21} and $f\star p\ll \mu\times\nu$, where $\mu\times\nu$ is the product measure. Let $\pi_{X}:X\times Y\to X$ and $\pi_{Y}:X\times Y\to Y$ denote the standard projection maps, and note that $\pi_{X}\circ (f\star p)=p$ and $\pi_{Y}\circ (f\star p)=q$. Thus, the disintegration theorem implies that there exists a disintegration $r:Y\to X\times Y$ of $f\star p$ over $q$ consistent with $\pi_{Y}$. This means that there exists a $q$-null set $N\in\Sigma_{Y}$ such that $r_{y}$ is a probability measure on $(X\times Y,\Sigma_{X\times Y})$ for all $y\in Y\setminus N$ and 
\[
\int_{A}f_{x}(B_1\cap B_2)\,dp(x)
=(f\star p)\big(A\times (B_1\cap B_2) \big)
=(f\star p) \Big((A\times B_{1}) \cap \pi_{Y}^{-1} (B_{2})\Big)
=\int_{B_{2}}r_{y}(A\times B_1)\, dq(y)
\]
holds for all $A\in\Sigma_{X}$ and $B_1,B_2\in\Sigma_{Y}$, 
due to~\eqref{eqn:RCPmeasurablesets} and the fact that $\pi_{Y}^{-1}(B_{2})=X\times B_{2}$. In particular, taking $B_1=Y$ and setting $B:=B_2$ yields
\be
\label{eqn:fstarprq}
\int_{A}f_{x}(B)\,dp(x)=\int_{B}r_{y}(A\times Y)\,dq(y)
\ee
for all $A\in\Sigma_{X}$ and $B\in\Sigma_{Y}$, which is the statement $f\star p=r\circ q$. 
Set $f^{\star}:=\pi_{X}\circ r$, which can be taken to be a Markov kernel from $Y$ to $X$ since $r$ is $\nu$-a.e.\ equivalent to a Markov kernel. The construction of $f^{\star}$ can be summarized in the following two commutative diagrams 
\[
\xy0;/r.25pc/:
(-15,-7.5)*+{(X,p)}="X";
(15,-7.5)*+{(Y,q)}="Y";
(0,7.5)*+{(X\times Y,f\star p)}="XY";
{\ar"X";"Y"_{f}};
{\ar"X";"XY"^(0.4){(\id_{X}\times f)\circ\Delta}};
{\ar"XY";"Y"^(0.55){\pi_{Y}}};
\endxy
\quad\implies\quad
\xy0;/r.25pc/:
(-15,-7.5)*+{(X,p)}="X";
(15,-7.5)*+{(Y,q)}="Y";
(0,7.5)*+{(X\times Y,f\star p)}="XY";
{\ar"Y";"X"^{f^{\star}}};
{\ar"XY";"X"_(0.55){\pi_{X}}};
{\ar"Y";"XY"_(0.4){r}};
\endxy
\]
and a formula for it as a transition kernel is given by 
\[
f^{\star}_{y}(A)=r_{y}(A\times Y)
\]
for $y\in Y$ and $A\in\Sigma_{X}$ (we will provide a more explicit formula in terms of conditional probability distributions after Proposition~\ref{prop:cvnPetz}). 
Now, let $\Alg{A}:=L^{\infty}(X,\mu)$ and $\Alg{B}:=L^{\infty}(Y,\nu)$. Let $\omega$ be the faithful normal state associated with $p$ as in~\eqref{eqn:fromptoomega}, let $\xi$ be the faithful normal state associated with $q$ as in~\eqref{eqn:fromqtoxi}, and let $\mathcal{F}$ be the NCPU map associated with $f$ as in~\eqref{eqn:MarkovOperator}. Similarly, let $\mathcal{F}^{\star}$ be the NCPU map associated with $f^{\star}$, i.e., 
\be
\label{eqn:fromfstartoFstar}
\mathcal{F}^{\star}(\varphi)(y)
:=\int_{X}\varphi \, df^{\star}_{y}
\ee
for all $\varphi\in L^{\infty}(X,\mu)$ and $y\in Y$. 

\bn
\label{prop:cvnPetz}
Given the notation in the previous paragraph, i.e., a morphism $(L^{\infty}(Y,\nu),\xi)\xrightarrow{\mathcal{F}} (L^{\infty}(X,\mu),\omega)$ in $\mathbf{C}\subset\mathbf{Q}$, the NCPU map $\mathcal{F}^{\star}:L^{\infty}(X,\mu)\to L^{\infty}(Y,\nu)$ defined by~\eqref{eqn:fromfstartoFstar} is the Petz recovery map of $\mathcal{F}$. 
\en

\bprf
On the one hand
\be
\label{eqn:LHScvn}
\big\<\!\big\<\mathcal{F}(\psi),\varphi\big\>\!\big\>_{\omega}
=\int_{X} \overline{\left(\int_{Y}\psi\,df_{x}\right)}\varphi(x)\,dp(x)
=\int_{X}\int_{Y}\overline{\psi}(y)\varphi(x)\,df_{x}(y)\,dp(x), 
\ee
while on the other hand
\be
\label{eqn:RHScvn}
\big\<\!\big\<\psi,\mathcal{F}^{\star}(\varphi)\big\>\!\big\>_{\xi}
=\int_{Y}\overline{\psi}(y)\left(\int_{X}\varphi \, df^{\star}_{y}\right)\,dq(y)
=\int_{Y}\int_{X}\overline{\psi}(y)\varphi(x)\,df^{\star}_{y}(x)\,dq(y)
\ee
for all $\varphi\in L^{\infty}(X,\mu)$ and $\psi\in L^{\infty}(Y,\nu)$. Taking $\varphi=\chi_{A}$ and $\psi=\chi_{B}$ to be characteristic functions on $A\in\Sigma_{X}$ and $B\in\Sigma_{Y}$, respectively,~\eqref{eqn:LHScvn} becomes
\be
\label{eqn:LHScvn2}
\big\<\!\big\<\mathcal{F}(\psi),\varphi\big\>\!\big\>_{\omega}
=\int_{A}\int_{B}df_{x}(y)\,dp(x)
=\int_{A}f_{x}(B)\,dp(x),
\ee
while~\eqref{eqn:RHScvn} becomes
\be
\label{eqn:RHScvn2}
\big\<\!\big\<\psi,\mathcal{F}^{\star}(\varphi)\big\>\!\big\>_{\xi}
=\int_{B}\int_{A}df^{\star}_{y}(x)\,dq(y)
=\int_{B}f^{\star}_{y}(A)\,dq(y)
=\int_{B}r_{y}(A\times Y)\,dq(y).
\ee
Thus, by~\eqref{eqn:fstarprq}, $\big\<\!\big\<\mathcal{F}(\psi),\varphi\big\>\!\big\>_{\omega}=\big\<\!\big\<\psi,\mathcal{F}^{\star}(\varphi)\big\>\!\big\>_{\xi}$ for characteristic functions $\varphi=\chi_{A}$ and $\psi=\chi_{B}$. Since all measurable functions are limits of linear combinations of characteristic functions, this proves $\mathcal{F}^{\star}$ is the Petz recovery map of $\mathcal{F}$. 
\eprf

Before concluding this section, we relate the identity $\big\<\!\big\<\mathcal{F}(\psi),\varphi\big\>\!\big\>_{\omega}=\big\<\!\big\<\psi,\mathcal{F}^{\star}(\varphi)\big\>\!\big\>_{\xi}$ to a more common version of Bayes' rule in terms of probability density functions. We do this by using the Radon--Nikodym derivative to re-express the measures in the equality
\be
\label{eqn:almostBayescnV}
\int_{A}\int_{B}df_{x}(y)\,dp(x)=\int_{B}\int_{A}df^{\star}_{y}(x)\,dq(y),
\ee
which came from the fact that~\eqref{eqn:LHScvn2} equals~\eqref{eqn:RHScvn2}. To do this, we first note that if a transition kernel $f:X\to Y$ satisfies $f\circ\mu\ll\nu$ and if $p$ is a probability distribution on $X$ that satisfies $p\ll\mu$ and is also faithful, i.e., $\mu\ll p$, then $f_{x}\ll\nu$ for $\mu$-almost all $x\in X$. To see this, suppose that $\nu(N)=0$ with $N\in\Sigma_{Y}$. Since $f\circ p\ll\nu$, we have 
\[
0=(f\circ p)(N)=\int_{X}f_{x}(N)\,dp(x). 
\]
Since $f_{x}(N)\ge0$ for all $x\in X$, there exists a $p$-null set $M_{N}\in\Sigma_{X}$ such that $f_{x}(N)=0$ for all $x\in X\setminus M_{N}$. Hence, $\mu(M_{N})=0$ since $p$ is faithful with respect to $\mu$. Thus, $f_{x}(N)=0$ for $\mu$-almost all $x\in X$. The Dynkin $\pi$-$\lambda$ theorem~\cite{Sengupta2006} can be used here as well (compare Footnote~\ref{footnote:Dynkin}) to conclude that there exists a single $p$-null set $M\in \Sigma_{X}$ such that $f_{x}(N)=0$ for all $\nu$-null sets $N\in\Sigma_{Y}$ and for all $x\in X\setminus M$. Using this, we obtain 
\be
\int_{A}\int_{B}\frac{df_{x}}{d\nu}(y)\frac{dp}{d\mu}(x)\,d\nu(y)\,d\mu(x)=\int_{B}\int_{A}\frac{df^{\star}_{y}}{d\mu}(x) \frac{dq}{d\nu}(y) \, d\mu(x) \,d\nu(y)
\ee
directly from~\eqref{eqn:almostBayescnV}. Since this is true for all $A\in\Sigma_{X}$ and $B\in\Sigma_{Y}$, the integrands are equal almost everywhere with respect to the product measure $\mu\times\nu$ on $(X\times Y,\Sigma_{X\times Y})$ by Fubini's theorem. Thus, 
\be
\label{eqn:Bayesdensities}
\frac{df_{x}}{d\nu}(y)\frac{dp}{d\mu}(x)=\frac{df^{\star}_{y}}{d\mu}(x) \frac{dq}{d\nu}(y)
\ee
for $(\mu\times\nu)$-almost all $(x,y)\in X\times Y$. The measurable function $\frac{df_{x}}{d\nu}(y)$ is the \define{conditional probability distribution} of $y$ given $x$. By making the identifications 
\[
\mathbb{P}(y|x):=\frac{df_{x}}{d\nu}(y),
\qquad
\mathbb{P}(x)=\frac{dp}{d\mu}(x),
\qquad
\mathbb{P}(x|y):=\frac{df^{\star}_{y}}{d\mu}(x),
\qquad
\mathbb{P}(y)=\frac{dq}{d\nu}(y),
\]
the equality~\eqref{eqn:Bayesdensities} becomes
\[
\mathbb{P}(y|x)\mathbb{P}(x)=\mathbb{P}(x|y)\mathbb{P}(y)
\quad
\iff
\quad
\mathbb{P}(x|y)=\frac{\mathbb{P}(y|x)\mathbb{P}(x)}{\mathbb{P}(y)}, 
\]
which is the familiar Bayes' rule for conditional probability distribution functions~\cite[Theorem 1.31]{Schervish1995},~\cites{Bayes1763,Pearl88}.

\section{Markov maps are covariant quantum operations}
\label{sec:MarkovMaps}

In this section, we study a subclass of morphisms $(\Alg{B},\xi)\xrightarrow{\mathcal{E}}(\Alg{A},\omega)$ in $\mathbf{Q}$ that obey a particular modular covariance property formalized by Accardi and Cecchini~\cite{AcCe82}, which has its origins in Takesaki's theorem on conditional expectations~\cite{Takesaki1972conditional}. 

\bd
Let $(\Alg{A},\omega)$ and $(\Alg{B},\xi)$ be objects in $\mathbf{Q}$. Let $\sigma^{\omega}_{t}:\Alg{A}\to\Alg{A}$ and $\sigma^{\xi}_{t}:\Alg{B}\to\Alg{B}$ denote the modular automorphism groups associated with $(\Alg{A},\omega)$ and $(\Alg{B},\xi)$, respectively. A morphism $(\Alg{B},\xi)\xrightarrow{\mathcal{E}}(\Alg{A},\omega)$ in $\mathbf{Q}$ is a \define{Markov map} iff 
\[
\sigma^{\omega}_{t}\circ\mathcal{E}=\mathcal{E}\circ\sigma^{\xi}_{t}
\quad\text{ for all $t\in\R$}, 
\]
which is called the \define{Accardi--Cecchini condition} for the morphism $\mathcal{E}$.
Given a morphism $(\Alg{B},\xi)\xrightarrow{\mathcal{E}}(\Alg{A},\omega)$ in $\mathbf{Q}$, its \define{GNS adjoint} is the linear map $\mathcal{E}_{\mathrm{GNS}}^{\star}:\Alg{A}\to\Alg{B}$ satisfying
\[
\omega\big( \mathcal{E}(b)^* a \big) = \xi \big( b^* \mathcal{E}_{\mathrm{GNS}}^{\star} (a)\big) 
\quad\text{ for all $a\in\Alg{A}$, $b\in\Alg{B}$}. 
\]
\ed

There are many names that could be used instead of \emph{Markov maps}~\cites{HaMu2011,Anantharaman2006}, including \emph{modular flow intertwiner} or \emph{modular time-covariant morphism}. The former is reasonable since the modular group is a representation of $\R$ on the automorphism group of the von~Neumann algebra, and $\mathcal{E}$ serves as an intertwiner for this representation~\cite{Ha13}. When the modular automorphism group can be interpreted as time translation, this is sometimes called \emph{time-translation symmetry} or \emph{time-translation invariance}~\cites{LKJR15,Marvian2020coherence}. The terminology of \emph{covariant quantum operations} is also used, though those are applicable to more general group actions and symmetries, not necessarily restricted to the symmetry induced by modular flow~\cites{Ma12,MarvianSpekkens2014}. Examples of Markov maps abound, and besides the references already mentioned, additional examples include Davies' maps in the context of open systems and thermal equilibrium~\cites{Davies74,RFZ10}, thermal operations~\cite{AWWW18}, perfect error-correcting codes~\cites{KnillLaflamme1997,PaBayes}, state-preserving conditional expectations~\cites{Takesaki1972conditional,GPRR21}, generalized (hypergroup) symmetries of conformal nets~\cites{Bis17,BiDeGi2021}, GNS detailed balance~\cite{CaMa17}, and much more. 

\bn
The collection of Markov maps inside $\mathbf{Q}$ defines a symmetric monoidal subcategory $\mathbf{M}$ of $\mathbf{Q}$. Moreover, $\mathcal{E}^{\star}=\mathcal{E}_{\mathrm{GNS}}^{\star}$ for all morphisms $\mathcal{E}$ in $\mathbf{M}$.
\en

Thus, not only does the Petz retrodiction functor $\mathscr{R}$ restrict to $\mathbf{M}$ in the sense that $\mathscr{R}(\mathcal{E})$ is a morphism in $\mathbf{M}$ whenever $\mathcal{E}$ is a morphism in $\mathbf{M}$, but the Petz retrodiction of a morphism coincides with the GNS adjoint of that morphism. 

\bprf
The claim that $\mathbf{M}$ is a symmetric monoidal subcategory of $\mathbf{Q}$ follows directly from the definitions and/or from techniques that we have already employed in this work. Namely, identity morphisms satisfy the Accardi--Cecchini condition, the composition of $(\Alg{C},\zeta)\xrightarrow{\mathcal{F}}(\Alg{B},\xi)\xrightarrow{\mathcal{E}}(\Alg{A},\omega)$ also satisfies it because
\[
\sigma^{\omega}_{t}\circ\mathcal{E}\circ\mathcal{F}
=\mathcal{E}\circ\sigma^{\xi}_{t}\circ\mathcal{F}
=\mathcal{E}\circ\mathcal{F}\circ\sigma^{\zeta}_{t}
\]
for all $t\in\R$. 
Finally, the tensor product of two Markov maps is itself a Markov map because the modular operators are tensorial, as discussed in Appendix~\ref{sec:reviewModularTheory}, and hence the modular automorphism group is also tensorial~\cite[Proposition 4.3 Chapter VIII]{Tak2}. As for the final claim, if $\mathcal{E}_{\mathrm{GNS}}^{\star}$ denotes the GNS adjoint of a Markov map $\mathcal{E}$, which necessarily exists, then complete positivity of $\mathcal{E}_{\mathrm{GNS}}^{\star}$ is equivalent to the morphism $(\Alg{B},\xi)\xrightarrow{\mathcal{E}}(\Alg{A},\omega)$ being a Markov map~\cite[Proposition 6.1]{AcCe82} (see Ref.~\cite{GPRR21} for a relatively simple proof in finite dimensions). Let $(\mathcal{H},\pi,\Omega)$, $(\mathcal{K},\tau,\Xi)$, $\Alg{M}$, $\Alg{N}$, and $E:\Alg{N}\,\Xi\to\Alg{M}\,\Omega$ be defined as in Lemma~\ref{lem:contractionforE}, and let us identify $\Alg{A}$ with $\Alg{M}$, $\Alg{B}$ with $\Alg{N}$, and $\mathcal{E}$ with $\widetilde{\mathcal{E}}$ for the calculations that follow. In terms of $E$, the Accardi--Cecchini condition is equivalent to~\cite[Proposition 6.1]{AcCe82}
\be
\label{eqn:JEEJ}
J_{\Omega} E=E J_{\Xi}. 
\ee
Now, translating the definition of the GNS adjoint in terms of the cyclic and separating vector representations yields 
\[
\big\< b\,\Xi, E^* a\,\Omega\>
=\< E b \, \Xi, a\,\Omega\>
=\big\<\mathcal{E}(b)\,\Omega, a\,\Omega\big\>
=\omega\big( \mathcal{E}(b)^* a \big)
=\xi \big( b^* \mathcal{E}_{\mathrm{GNS}}^{\star} (a)\big) 
=\big\<b\,\Xi,\mathcal{E}^{\star}_{\mathrm{GNS}}(a)\,\Xi\big\>
\]
for all $a\in\Alg{A}$ and $b\in\Alg{B}$. Thus, by cyclicity of $\Xi$, we conclude that 
\be
\label{eqn:EGNS}
\mathcal{E}^{\star}_{\mathrm{GNS}}(a) \,\Xi = E^* a\,\Omega
\ee
for all $a\in\Alg{A}$. Hence, 
\begin{align*}
\big\<\!\big\< \mathcal{E}(b),a\big\>\!\big\>_{\omega}
&= \big\langle \E(b) \, \Omega, J_{\Omega} a^* \, \Omega \big\rangle && \mbox{ by definition of $\<\!\<\,\cdot\,,\,\cdot\,\>\!\>_{\omega}$}\\
&= \big\langle E b \, \Xi, J_{\Omega} a^* \, \Omega \big\rangle && \mbox{ by definition of $E$ } \\
&= \overline{\big\langle J_{\Omega} E b \, \Xi, a^* \, \Omega \big\rangle} && \mbox{ since $J_{\Omega}$ is a conjugate-linear unitary } \\
&= \overline{\big\langle E J_{\Xi} b \, \Xi, a^* \, \Omega \big\rangle} && \mbox{ by~\eqref{eqn:JEEJ}}\\
& = \big\langle b \, \Xi, J_{\Xi} E^* a^* \, \Omega \big\rangle && \mbox{ since $J_{\Xi}$ is a conjugate-linear unitary } \\ 
&=  \big\langle b \, \Xi, J_{\Xi} \mathcal{E}^{\star}_{\mathrm{GNS}}( a^* ) \, \Xi \big\rangle && \mbox{ by~\eqref{eqn:EGNS}}\\ 
&= \big\<\!\big\< b,\mathcal{E}^{\star}_{\mathrm{GNS}}(a) \big\>\!\big\>_{\xi} && \mbox{ since $\mathcal{E}^{\star}_{\mathrm{GNS}}$ is positive and by definition of $\<\!\<\,\cdot\,,\,\cdot\,\>\!\>_{\xi}$},
\end{align*}
which, by the uniqueness of the Petz recovery map, proves that the GNS adjoint of the morphism $\mathcal{E}$ equals the Petz recovery map of $\mathcal{E}$. 
\eprf

\section{Summary and discussion}\label{sec:conclusion}

In this work, we showed that the Petz recovery map for state-preserving normal completely positive unital (NCPU) maps between arbitrary von~Neumann algebras satisfies the same categorical axioms that classical Bayesian inversion satisfies. 
This provides a rigorous and structural argument for how the Petz recovery map extends Bayes' rule beyond the commutative setting, and it shows how the Petz recovery map defines a retrodiction functor~\cites{PaBu22,Pa24}. 
Although some of these properties were known since the original work of Petz~\cite{Pe84}, our work improves on these results in three ways; namely, we included the monoidal structure, we explicitly showed how this functor restricts to the standard Bayes' rule for standard Borel spaces~\cites{Bayes1763,Pearl88,Schervish1995}, and we showed how the essential properties of the Petz recovery map and Bayesian inversion satisfy purely categorical axioms. 

In the process, we also reviewed several aspects of the Petz recovery map, including its historical derivation together with a clear connection between the version often used in the quantum information theory literature~\cites{Wilde15,WildeQIT16} (whose formula is valid for finite-dimensional systems but is no longer valid for the algebras appearing in quantum field theory~\cite{Yngvason2005}) and the original definition in terms of modular theory for von~Neumann algebras~\cite{Pe84}. By framing the construction of the Petz recovery map from a categorical perspective and proving that it defines a retrodiction functor~\cites{PaBu22,Pa24}, we showed how its restriction to several subcategories agrees with other notions of adjoints or inversion in the literature. For example, it specializes to the case of the actual inverse (hence the name) when the NCPU maps are invertible and have an NCPU inverse, i.e., an isomorphism in the category. We also showed how it specializes to the GNS adjoint and the notions of Bayesian inverses from categorical probability~\cites{CDDG17,Jacobs2020,ChJa18,Fr20,PaBayes,PaRuBayes} (the notion of Bayesian inversion in categorical probability coincides with the notion of GNS adjoint~\cites{PaRuBayes,PaBayes,GPRR21,FuPa22a}). We summarize many of these results in Figure~\ref{fig:diagramfunctors}, which is an extension of the results from Ref.~\cite{PaBu22}.

Motivated by the results of Ref.~\cite{PaBu22}, which showed that a variety of other endofunctors (such as the rotated Petz recovery map~\cite{Wilde15} and its averaged variants~\cites{JRSWW16,JRSWW18}) are \emph{not} retrodiction functors, we state the following conjecture about the uniqueness of retrodiction. 

\begin{conjecture}
If $\mathscr{R}$ is a retrodiction functor on $\mathbf{Q}$, then $\mathscr{R}$ coincides with the Petz retrodiction functor. In other words, Petz retrodiction is uniquely characterized by the axioms in Definition~\ref{defn:retrodiction}. In particular, Bayesian inference is characterized by these same axioms. 
Informally: classical and quantum inference are uniquely characterized by structural process-theoretic axioms. 
\end{conjecture}

Existing derivations of Bayes' rule, Bayesian inversion, and/or the Petz recovery map are often based on optimizing a wide variety of statistical distances, fidelity, divergences, entropy, or information gain/loss~\cites{Zellner1988,Caticha2021,BaiBuscemiScarani2025,Cs91,ShoreJohnson1980}. In contrast, this conjecture suggests a fundamentally different \emph{definition} of classical and quantum inference that is based on structural and categorical axioms, in which the more familiar optimization-based principles would emerge as consequences, rather than as assumptions. This approach is somewhat similar in spirit to Refs.~\cites{Cs91,ShoreJohnson1980}, but instead utilizes natural process-theoretic axioms rather than specific linear algebra and probabilistic concepts like affine subspaces, convex constraints, and projections that seem specific to classical probability and are still rooted in optimization-based approaches. Our conjecture also goes beyond ordinary Bayes' rule, which specifies a posterior based on hard evidence, in that it aims to characterize the mechanism of classical and quantum statistical inference that can be used more broadly for soft evidence, Jeffrey's update rule, Pearl's virtual evidence method, and probability kinematics~\cites{Je90,Ja19,ChanDarwiche2005,Pearl88}. In other words, the conjecture claims that Bayes' rule is not merely an algorithm, but a structural necessity. 
In addition to mathematically formalizing this major open problem, we provide several other questions for future research. 

(1) 
The Petz recovery map is seen as one instance of recent quantum generalizations of Bayes' rule and Bayesian inverses~\cites{FuPa22a,Ts22,SASDS23}. Although conditions for the existence of GNS Bayesian inverses are well-known~\cites{AcCe82,PaRuBayes,FaUm07,GPRR21}, it is significantly less clear what conditions (say, in terms of the modular group) are needed for other types of Bayesian inverses to exist, such as the one of Ref.~\cite{PaFu24TSC}. The latter is defined as the adjoint with respect to the inner product
\[
\<\!\<a_1,a_2\>\!\>_{\omega}^{\mathrm{J}}:=\frac{1}{2}\omega(a_1^* a_2+a_2 a_1^*)
\]
for all $a_1,a_2\in\Alg{A}$ with $\omega$ a faithful state on $\Alg{A}$. 
Only recently has there been a characterization for unital qubit quantum channels~\cite{TiFuWu2026} and for the amplitude-damping channel~\cite{PaFu24TSC}. The importance of understanding this problem is closely connected to distinguishing between correlation and causation in measurements of Pauli observables for multiqubit systems~\cites{FJV15,SNREG23,JSK23,SongParzygnat2025,FuPa24a,liu2023quantum}. 

(2) 
In our category $\mathbf{Q}$, if the faithful states on the von~Neumann algebras are replaced by semifinite weights, how does the expression for the Petz recovery map change? 

(3) Although we did not discuss it here, because the crossed product construction can take a faithful state on a type $\text{III}_{1}$ von~Neumann algebra and transfer it to a faithful state on a type $\text{II}_{\infty}$ von~Neumann algebra~\cite{Takesaki1973CrossProduct}, could we use the semifinite trace in order to write a more direct formula for Petz recovery maps involving densities defined with respect to the semifinite trace? Such a Petz recovery map seems like it would be useful in the context of quantum field theory and gravity~\cites{Witten2022CrossedProduct,AhmadKlinger2025}. 

(4) Regarding the crossed product $\Alg{A}\rtimes_{\sigma}\R$ associated with a type $\text{III}_{1}$ von~Neumann algebra and a faithful state $(\Alg{A},\omega)$, if $\mathcal{E}:\Alg{A}\rtimes_{\sigma}\R\to\Alg{A}$ denotes the restriction of the map on $\Alg{A}\overline{\otimes}\mathcal{B}(L^2(\R))\to\Alg{A}$ uniquely determined by sending $a\otimes T$ to $\xi(T) a$ for some fixed faithful state $\xi$ on $\mathcal{B}(L^2(\R))$, is there a more explicit description of the Petz recovery map of $\mathcal{E}$ in this case? Since this map appears in Ref.~\cite{AhmadKlinger2025}, we expect it to provide some implications towards quantum gravity~\cite{Witten2022CrossedProduct}. 

(5) Beyond classical and quantum probability as described by von~Neumann algebras, in what other context do retrodiction functors appear? Is the definition sufficiently robust enough so as to merit being further developed as an abstract categorical concept to be later specialized in a variety of settings? As suggested by Ref.~\cites{PaBu22,Pa24}, such a notion would generalize the concept of time reversal, or in mathematical terms the notion of an inverse.

\appendix

\section{Concepts from analysis and von~Neumann algebras}
\label{sec:reviewModularTheory}

Here, we review some elements of functional analysis~\cites{Fo07,Rudin1991,ReedSimonI1980,Schmudgen2012}, von~Neumann algebras~\cites{Topping1971,Dix81,Fillmore96,BrownOzawa2008,Paulsen02,Tak1}, and modular theory~\cites{Takesaki1970,Tak2,StratilaZsido2019,StratilaModular2020}. Introductions with a focus on quantum applications include Refs.~\cites{Sorce2023,Witten18,Schechter2002,OhPe93,HallQuantum13,Teschl2014}. We find Ref.~\cite{Sorce2023} particularly clear and concise. We assume some basic familiarity with topology and analysis~\cites{Munkres2000,Fo07,Rudin76}. In what follows, all Hilbert spaces and $*$-algebras will be defined over the field of complex numbers. All Hilbert spaces will be assumed separable and equipped with their standard norm topologies, with the norm $\lVert\;\cdot\;\rVert$ induced from its inner product $\<\;\cdot\;,\;\cdot\;\>$. All algebras will be assumed unital, and the unit of an algebra $\Alg{M}$ will be denoted by $1_{\Alg{M}}$. 

Given Hilbert spaces $\mathcal{H}$ and $\mathcal{K}$, an \define{operator} from $\mathcal{H}$ to $\mathcal{K}$ is a pair $(\mathcal{D}_{T},T)$, where $\mathcal{D}_{T}\subseteq\mathcal{H}$ is a vector subspace of $\mathcal{H}$ and $T:\mathcal{D}_{T}\to\mathcal{K}$ is a linear transformation. The subspace $\mathcal{D}_{T}$ is called the \define{domain} of $T$ (if no domain is mentioned, it is either specified in the surrounding text or it can be assumed that $\mathcal{D}_{T}=\mathcal{H}$). A \define{conjugate-linear operator} from $\mathcal{H}$ to $\mathcal{K}$ is essentially the same except that it is \emph{conjugate}-linear, meaning that $T(cx)=\overline{c}T(x)$ for all $c\in\C$ and $x\in\mathcal{D}_{T}$, where $\overline{c}$ denotes the complex conjugate of $c$. Most definitions that follow hold equally well for conjugate-linear operators as for ordinary operators, except that tensor products are not defined between linear and conjugate-linear operators (additionally, the definition of the adjoint of a densely defined conjugate-linear operator is different, as will be spelled out later). 
If $(\mathcal{D}_{T},T)$ is an operator from $\mathcal{H}$ to $\mathcal{K}$ and $(\mathcal{D}_{S},S)$ is an operator from $\mathcal{K}$ to $\mathcal{L}$, the \define{composite} $(\mathcal{D}_{ST},ST)$ is an operator from $\mathcal{H}$ to $\mathcal{L}$ with domain $\mathcal{D}_{ST}:=\{x\in\mathcal{D}_{T}\,:\,Tx\in\mathcal{D}_{S}\}$ (note that $\mathcal{D}_{ST}\subseteq\mathcal{D}_{T}$ by definition). 
An operator $(\mathcal{D}_{T},T)$ from $\mathcal{H}$ to $\mathcal{K}$ is \define{bounded on its domain} $\mathcal{D}_{T}$ iff its \define{operator norm} 
\[
\lVert T\rVert:=\sup_{x\in \mathcal{D}_{T}}\big\{\lVert T x\rVert\;:\;\lVert x\rVert =1\big\} 
\]
is finite. Note that there are three different usages of the notation $\lVert\;\cdot\;\rVert$ for norms here---two are for vectors, with one in $\mathcal{H}$ and another in $\mathcal{K}$, while the other usage is for operators. We also define the operator norm similarly for linear transformations between Banach algebras. If an operator $T$ has domain $\mathcal{H}$ and is bounded on its domain, we simply say that $T$ is \define{bounded} or is a \define{bounded operator} (in other words, if a domain is not specified when an operator is bounded, one can assume it is the entire Hilbert space). The set of bounded operators on $\mathcal{H}$ is denoted by $\mathcal{B}(\mathcal{H})$, and it is an associative unital algebra (over $\C$) with multiplication given by function composition. 
An operator $(\mathcal{D}_{T},T)$ from $\mathcal{H}$ to $\mathcal{K}$ is \define{densely defined} iff its domain $\mathcal{D}_{T}$ is dense in $\mathcal{H}$ (in its norm topology).%
\footnote{Given densely defined operators $T$ from $\mathcal{H}$ to $\mathcal{K}$ and $S$ from $\mathcal{K}$ to $\mathcal{L}$, it need not be the case that $ST$ is densely defined. 
As an example (using some terminology that will be explained in subsequent paragraphs in the text), let $\mathcal{H}=\mathcal{K}=\mathcal{L}=L^{2}\big([0,1],\mu\big)$, elements of which are a.e.~equivalence classes of measurable and square-integrable functions with the standard Lebesgue measure. 
Let $f:[0,1]\to\C$ be the function $f(x)=\sqrt{2x}$, which satisfies $\lVert f\rVert=1$. Notice that $f$ is not differentiable at $x=0$, and, although $f$ is differentiable on $(0,1]$, its derivative is not in $L^{2}\big([0,1],\mu\big)$. 
Let $T$ be the operator that gives the orthogonal projection onto the linear span $\mathrm{span}\{f\}$, which is a closed one-dimensional subspace of $\mathcal{H}$. Namely, $Tg=\left(\int_{0}^{1}f(x)g(x)\,d\mu(x)\right)f$ for all $g\in\mathcal{H}$. Note that the domain of $T$ is all of $\mathcal{H}$ by the Cauchy--Schwarz inequality. Next, let $\mathcal{D}_{S}$ be the dense subspace of $L^{2}\big([0,1],\mu\big)$ admitting representative functions that are differentiable and whose derivatives are square-integrable, and let $S$ be the operator that acts by differentiation, i.e., $(Sg)(x)=\frac{dg}{dx}$ for all $g\in\mathcal{D}_{S}$. Then $\mathcal{D}_{ST}=\ker(T)$ since $f$ is not continuously differentiable. Since $T$ is an orthogonal projection, $\ker(T)$ is a closed subspace of $L^{2}\big([0,1],\mu\big)$. Since $\ker(T)$ does not contain $\mathrm{span}\{f\}$, $\mathcal{D}_{ST}=\ker(T)$ is not dense in $\mathcal{H}$, so $ST$ is not densely defined. 
}
An \define{extension} of an operator $(\mathcal{D}_{T},T)$ is an operator $(\mathcal{D}_{S},S)$ such that  $\mathcal{D}_{T}\subseteq\mathcal{D}_{S}$ and the restriction of $S$ to $\mathcal{D}_{T}$ agrees with $T$. The notation $(\mathcal{D}_{T},T)\subseteq(\mathcal{D}_{S},S)$ is used to indicate that $S$ is an extension of $T$, or more concisely $T\subseteq S$ when the domains are clear from context. Meanwhile, if $(\mathcal{D}_{T},T)$ is an operator from $\mathcal{H}$ to $\mathcal{K}$ and $\mathcal{V}\subseteq\mathcal{H}$ is a vector subspace of $\mathcal{H}$, then the notation $T|_{\mathcal{V}}$ is used to denote the restriction of $T$ to $\mathcal{V}\cap\mathcal{D}_{T}$. 

If $\mathcal{S}\subseteq\mathcal{H}$ is a subset (not necessarily a subspace) of a Hilbert space, the set $\mathcal{S}^{\perp}:=\{x\in\mathcal{H}\,:\,\<x,y\>=0\text{ for all }y\in \mathcal{H}\}$ is its \define{orthogonal complement}. Note that $\mathcal{S}^{\perp}$ is always a closed subspace of $\mathcal{H}$ in the norm topology of $\mathcal{H}$ because it is the intersection  $\bigcap_{y\in\mathcal{S}}f_{y}^{-1}(\{0\})$ of closed sets $f_{y}^{-1}(\{0\})$, where $f_{y}:\mathcal{H}\to\C$ defined as $f_{y}(x):=\<x,y\>$ is continuous for all $y$. If $(\mathcal{D}_{T},T)$ is an operator from $\mathcal{H}$ to $\mathcal{K}$, then $\ker(T):=\{x\in\mathcal{D}_{T}\,:\,Tx=0\}$ denotes its \define{kernel}.

An operator $(\mathcal{D}_{T},T)$ from $\mathcal{H}$ to $\mathcal{K}$ is \define{closed} iff the following condition holds: given any sequence $x:\N\to\mathcal{D}_{T}$ that converges to $x_{\infty}\in\mathcal{H}$ and such that $T\circ x:\N\to\mathcal{K}$ converges, then $x_{\infty}\in\mathcal{D}_{T}$ and $T\circ x$ converges to $T(x_{\infty})$.
\footnote{Warning: An operator $(\mathcal{D}_{T},T)$ being closed does not mean that the domain $\mathcal{D}_{T}$ of $T$ is closed in the norm topology of $\mathcal{H}$. This is because $\mathcal{D}_{T}$ need not contain the limits of all convergent sequences. 
One explanation for the terminology comes from the fact that the graph $\Gamma(T):=\{x \oplus T(x)\,:\,x\in\mathcal{D}_{T}\}$ of $T$ is a closed subset of $\mathcal{H}\oplus\mathcal{K}$ (the direct sum Hilbert space) if and only if $T$ is a closed operator. 
}
An operator $(\mathcal{D}_{T},T)$ from $\mathcal{H}$ to $\mathcal{K}$ is \define{closable} iff there exists an extension of $T$ to a closed operator, i.e., there exists an operator $(\mathcal{D}_{S},S)$ from $\mathcal{H}$ to $\mathcal{K}$ such that $T\subseteq S$ and $S$ is closed.%
\footnote{Not every densely defined operator is closable. A simple example can be obtained as follows. First let $\mathcal{H}=L^{2}\big([0,1],\mu\big)$. Then let $\mathcal{D}$ be the subspace of equivalence classes of measurable functions that contain continuous representatives, i.e., elements of $C\big([0,1]\big)$. Let $T$ be the operator from $\mathcal{H}$ to itself, but with domain $\mathcal{D}$, given by the constant function $T(f):=f(0)$, which is defined by choosing a continuous representative of $f$ and evaluating that representative at $0$. The operator $T$ is well-defined on $\mathcal{D}$ because if $f$ and $f'$ are two continuous representatives, then $f-f'$ is continuous and a.e.~equal to $0$, which implies $f-f'=0$ everywhere. Moreover, $\mathcal{D}$ is dense inside $\mathcal{H}$. To see that $T$ is not closable, for each $n\in\N$, let $f_{n}(x):=(1-nx)\chi_{[0,1/n]}(x)$, where $\chi_{E}$ is the indicator function on a measurable subset $E\subseteq[0,1]$, which is $1$ if $x\in E$ and $0$ otherwise. Then, the sequence $\{f_{n}\}$ converges not to $\lim_{n\to\infty}f_{n}=0$ in $\mathcal{H}$. Meanwhile, $T(f_{n})=f_{n}(0)=1$ (the constant function whose value is $1$) for all $n\in\N$ so that $\lim_{n\to\infty}T(f_{n})=1$. Since this is not equal to $T(\lim_{n\to\infty}f_{n})=T(0)=0$, the operator $T$ is not closable.
}
If $T$ is closable, its smallest closed extension is often denoted by $\overline{T}$. 
Given a densely defined operator $T$ from $\mathcal{H}$ to $\mathcal{K}$, its \define{adjoint} is the operator $T^*$ from $\mathcal{K}$ to $\mathcal{H}$ defined as follows~\cite[VIII.1]{ReedSimonI1980}.
\footnote{We do not define the adjoint of an operator $T$ that does not have a dense domain.}
First, the domain of $T^*$ is given by
\[
\begin{split}
\mathcal{D}_{T^*}&:=\big\{z\in\mathcal{K}\,:\,\text{there exists an } x\in\mathcal{H}\text{ such that }\<z,Ty\>=\<x,y\>\text{ for all } y\in\mathcal{D}_{T}\big\}\\
&\;=\big\{z\in\mathcal{K}\,:\, \mathcal{D}_{T}\ni y\mapsto \<z,Ty\> \text{ is bounded on its domain $\mathcal{D}_{T}$}\big\}
\end{split}
\]
(the equivalence between these two definitions follows from the Riesz lemma). Because the operator $T$ is densely defined, the element $x\in\mathcal{H}$ that satisfies the condition $\<x,y\>=\<z,Ty\>$ for all $y\in\mathcal{D}_{T}$ is unique, and is therefore taken as the definition of $T^*z$. In other words, $\<z,Ty\>=\<T^*z,y\>$ for all $y\in\mathcal{D}_{T}$ and $z\in\mathcal{D}_{T^*}$. 
We note that if $T$ happens to be a conjugate-linear operator defined on a dense domain, then its \define{adjoint} $T^*$ is defined similarly except that it satisfies $\<T y,z\>=\<T^* z,y\>$. 
Although the domain $\mathcal{D}_{T^*}$ of $T^{*}$ need not be dense in general (even though the domain of $T$ is dense), this is the case with closable operators. Namely, a densely defined operator $T$ is closable if and only if its adjoint $T^*$ is densely defined.%
\footnote{If we go back to our earlier example of the nonclosable operator $T$ defined on the dense domain $C\big([0,1]\big)$ inside $L^2\big([0,1],\mu\big)$ as evaluation at $0$, then this says the adjoint $T^*$ cannot be densely defined. Indeed, $g\in\mathcal{D}_{T^*}$ whenever the function 
$
C\big([0,1]\big) \ni f \mapsto  \langle g, T(f) \rangle = f(0)\int \overline{g}\,d\mu 
$
is bounded, i.e., whenever there exists a real number $M_{g}\ge0$ such that 
$
\big|f(0)\big| \; \left|\int \overline{g}\,d\mu\right|\le M_{g} \lVert f\rVert_{2} 
$
for all $f\in C\big([0,1]\big)$. In particular, let $\{f_{n}\}$ be the sequence of continuous functions given by $f_{n}(x)=\sqrt{3n}(1-nx)\chi_{[0,1/n]}(x)$. Then $\lim_{n\to\infty}f_{n}(0)=\infty$ and yet $\lVert f_{n}\rVert_{2}=1$. The previous inequality therefore requires $\int g\,d\mu=0$. Hence, 
$
\mathcal{D}_{T^*}
=\left \{g\in L^2\big([0,1],\mu\big): \int g\,d\mu=0\right\}
=\mathrm{span}\{1\}^{\perp}, 
$
where $1$ is the constant function whose value is $1$ everywhere. This shows that $\mathcal{D}_{T^*}$ is a closed subspace of $L^2\big([0,1],\mu\big)$, and is therefore not dense. As for how  $T^*$ acts on $g\in \mathcal{D}_{T^*}$, we note that 
$
0=\overline{f(0)}\int g\,d\mu
=\big\<T(f),g\big\>
=\big\<f,T^*(g)\big\>
=\int \overline{f} \, T^*(g)\,d\mu
$
for all $f\in C\big([0,1]\big)$. This implies $T^*(g)=0$. Thus, $T^*$ acts as the zero operator on $\mathcal{D}_{T^*}$ (compare this with~\cite[VIII.1 Example 4]{ReedSimonI1980}). 
}
In such a case, $T^*=\overline{T}^*$ and $(T^*)^*=\overline{T}$. 
In fact, $T^*$ is always closed. Additionally, when $T\subseteq S$, then $S^*\subseteq T^*$. 
When $T$ is a closed operator defined on a dense domain, then $T^*$ is also a closed operator defined on a dense domain. 
When $T$ is a bounded operator (on all of $\mathcal{H}$), then $T^*$ is also a bounded operator (on all of $\mathcal{K}$). 
If $T$ is a densely defined operator from $\mathcal{H}$ to $\mathcal{K}$ and if $S$ is bounded operator from $\mathcal{K}$ to $\mathcal{L}$, then $\mathcal{D}_{ST}=\mathcal{D}_{T}$ is still dense and $T^* S^*=(ST)^*$.
More generally, if $S$ is a densely defined operator and if $ST$ is also densely defined, then $T^*S^*\subseteq (ST)^*$.
%
\footnote{Warning: If $S$ is only densely defined, it need \emph{not} be true that $T^*S^*$ and $(ST)^*$ are equal since the domains might be different. We can see this using the position and momentum operators from quantum mechanics~\cite{ReedSimonI1980}. Namely, let $\mathcal{H}=\mathcal{K}=\mathcal{L}=L^2(\R,\mu)$. We set the position operator $S$, defined by $(Sf)(x)=x f(x)$, to have domain $\mathcal{D}_{S}=\left\{f \in L^2(\mathbb{R},\mu): \int  x^2 |f(x)|^2  < \infty\right\}$. We set the momentum operator $T$, defined by $(Tf)(x)=-i \frac{df}{dx}$, to have domain $\mathcal{D}_{T}=\left\{ f \in L^2(\mathbb{R},\mu): f \text{ absolutely continuous on }\mathbb{R}, \frac{df}{dx} \text{ exists, and } \lVert\frac{df}{dx}\rVert_{2}<\infty\right\}$. Then both $S$ and $T$ are self-adjoint on these domains. Now, let $f$ be the function defined by $f(x)=\frac{1}{\sqrt{1+x^2}}$. Then $f\in L^2(\R,\mu)$. Moreover, 
$f\notin\mathcal{D}_{TS}=\mathcal{D}_{T^*S^*}$ because $f\notin\mathcal{D}_{S}=\mathcal{D}_{S^*}$ as can be checked by an explicit integral computation. We claim that $f\in\mathcal{D}_{(ST)^*}$, which would prove that $\mathcal{D}_{T^*S^*}\subsetneq \mathcal{D}_{(ST)^*}$. To see this, we must prove that the functional $\mathcal{D}_{ST}\ni h\mapsto \<f,STh\>$ is bounded, i.e., there exists a $C\ge0$ such that $|\<f,STh\>|\le C\lVert h\rVert_{2}$ for all $h\in\mathcal{D}_{ST}$. This follows from integration by parts and the Cauchy--Schwarz inequality because $|\<f,STh\>|=\left|\int_{-\infty}^{\infty}\frac{x}{\sqrt{1+x^2}} \frac{dh}{dx}\,dx\right|=\left|\int_{-\infty}^{\infty}\frac{h(x)}{(1+x^2)^{3/2}}\,dx\right|\le \left(\int_{-\infty}^{\infty}\frac{dx}{(1+x^2)^3}\right)^{1/2} \lVert h\rVert_{2}=\sqrt{2}\,\lVert h\rVert_{2}$. 
}
A densely-defined operator $T$ is \define{self-adjoint} iff $T^*=T$ (including equality of domains). 

A \define{partial isometry} from $\mathcal{H}$ to $\mathcal{K}$ is a bounded operator $V:\mathcal{H}\to\mathcal{K}$ such that $V^* V$ and $VV^*$ are both projections. It is an \define{isometry} iff $V^* V=\id_{\mathcal{H}}$ and \define{unitary} iff, in addition, $VV^*=\id_{\mathcal{K}}$. 
A self-adjoint operator $S$ with (necessarily dense) domain $\mathcal{D}_{S}\subseteq\mathcal{H}$ is \define{positive} iff $\<x,Sx\>\ge0$ for all $x\in\mathcal{D}_{S}$. If $T$ is a densely defined closed operator from $\mathcal{H}$ to $\mathcal{K}$, then $S:=T^*T$ is positive (and in particular self-adjoint and densely defined) and $\overline{T|_{\mathcal{D}_{S}}}=T$. 
Moreover, $S$ admits a positive square root $|T|:=\sqrt{S}$ with domain $\mathcal{D}_{|T|}=\mathcal{D}_{T}$, which means that $|T|$ is a positive operator satisfying $|T|^2=S$. 
Secondly, there exists a partial isometry $V:\mathcal{H}\to\mathcal{K}$ such that $T=V\,|T|$ and $V^*V$ projects onto $\ker(T)^{\perp}$. The decomposition $T=V\,|T|$ is called the \define{polar decomposition} of $T$. The polar decomposition is unique in the sense that if $T=WP$, with $P$ positive and $W$ a partial isometry such that $W^*W$ projects onto $\ker(T)^{\perp}$, then $W=V$ and $P=|T|$. This fact about uniqueness can be found in~\cite[Theorem 7.20]{Weidmann1980} and~\cite[Theorem 7.2]{Schmudgen2012}, for example. 

The $*$-algebra of bounded operators on a Hilbert space $\mathcal{H}$ is denoted by $\mathcal{B}(\mathcal{H})$, where the $*$ operation is given by the adjoint. We first discuss several types of topologies on $\mathcal{B}(\mathcal{H})$, which are used to make sense of continuity and limits~\cite[Section I.3.1]{Dix81}. The \define{norm topology} on $\mathcal{B}(\mathcal{H})$ is the metric space topology generated by the operator norm. Under this norm, $\mathcal{B}(\mathcal{H})$ is a Banach algebra. 
The \define{weak operator topology} on $\mathcal{B}(\mathcal{H})$ is the smallest topology (fewest open sets) such that the assignment 
\[
\mathcal{B}(\mathcal{H})\ni T\mapsto \<x,Ty\>
\]
is continuous for all $x,y\in\mathcal{H}$. 
We can better understand the weak operator topology in terms of convergence of nets, a slight generalization of sequences. A \define{net} in a topological space $(X,\tau)$, where $X$ is the set and $\tau$ its topology of open sets, is a function of the form $T:A\to X$, where $(A,\le)$ is a directed set~\cite{Willard2004}. Such a net \define{converges} to an element $T_{\infty}\in X$ iff for every open set $U\in\tau$ containing $T_{\infty}$, there exists a $\beta\in A$ such that $T_{\alpha}\in U$ for all $\alpha\ge\beta$. 
Thus, if $T:A\to\mathcal{B}(\mathcal{H})$ is a net of operators, then the net converges to $T_{\infty}\in\mathcal{B}(\mathcal{H})$ in the weak operator topology if and only if 
\[
\lim_{a\in A}\<x, T_{\alpha} y\> = \<x,T_{\infty}y\>
\]
for all $x,y\in \mathcal{H}$, and we say that $T$ \define{weakly converges to} $T_{\infty}$. Note that a net $T$ may weakly converge to $T_{\infty}$ even though $\lim_{\alpha\in A}\lVert T_{\infty} - T_{\alpha}\rVert$ need not be zero.%
\footnote{If we let $\ell^2$ be the Hilbert space of square-summable sequences $a:\N\to\C$, then a simple example is the sequence (which is an example of a net) $T:\N\to\mathcal{B}(\ell^2)$ of operators defined by $T_{n}(a):=a_{n}$, evaluation of $a\in\ell^2$ at $n\in\N$. In this case, $\lVert T_{n}\rVert=1$ for all $n\in\N$, $T$ does not converge to any operator in the norm topology, but $T$ does converge weakly to the zero operator.}
Although there are other topologies on $\mathcal{B}(\mathcal{H})$, we will only need to use these two for the vast majority of the paper. 

Given Hilbert spaces $\mathcal{H}$ and $\mathcal{K}$, their algebraic tensor product is denoted by $\mathcal{H}\otimes\mathcal{K}$, while the Hilbert space completion is denoted by $\mathcal{H}\widehat{\otimes}\mathcal{K}$ (this is completion in the norm topology induced by the inner product). Note that the algebraic tensor product is a vector subspace of the completion, i.e., $\mathcal{H}\otimes\mathcal{K}\subseteq\mathcal{H}\widehat{\otimes}\mathcal{K}$. 
As for the associated algebras of bounded operators, let $\mathcal{B}(\mathcal{H})\otimes\mathcal{B}(\mathcal{K})$ denote the algebraic tensor product of $\mathcal{B}(\mathcal{H})$ with $\mathcal{B}(\mathcal{K})$. There is a unique linear injection $\mathcal{B}(\mathcal{H})\otimes\mathcal{B}(\mathcal{K})\to\mathcal{B}(\mathcal{H}\widehat{\otimes}\mathcal{K})$ specified by sending $S\otimes T$ to the operator on $\mathcal{H}\otimes\mathcal{K}$ defined uniquely by its action on elements $x\otimes y$ by $S(x)\otimes T(y)$. This map defines an injective $*$-homomorphism so that $\mathcal{B}(\mathcal{H})\otimes\mathcal{B}(\mathcal{K})$ can be viewed as a subalgebra of $\mathcal{B}(\mathcal{H}\widehat{\otimes}\mathcal{K})$. As such, we will often identify $\mathcal{B}(\mathcal{H})\otimes\mathcal{B}(\mathcal{K})$ with its image inside $\mathcal{B}(\mathcal{H}\widehat{\otimes}\mathcal{K})$. In particular, if $S\in\mathcal{B}(\mathcal{H})$ and $T\in\mathcal{B}(\mathcal{K})$, then $S\otimes T$ will equivalently be viewed as an element of $\mathcal{B}(\mathcal{H}\widehat{\otimes}\mathcal{K})$. 
More generally, if $(T_{j},\mathcal{D}_{T_{j}})$ are operators from $\mathcal{H}_{j}$ to $\mathcal{K}_{j}$, with $j\in\{1,2\}$, note that $\mathcal{D}_{T_{1}}\otimes\mathcal{D}_{T_{2}}$ is a vector subspace inside $\mathcal{H}_{1}\otimes\mathcal{H}_{2}\subseteq \mathcal{H}_{1}\widehat{\otimes}\mathcal{H}_{2}$. We now further assume that $\mathcal{D}_{T_{1}}$ and $\mathcal{D}_{T_{2}}$ are dense inside $\mathcal{H}_{1}$ and $\mathcal{H}_{2}$, respectively. If the \emph{algebraic} tensor product of $T_{1}$ with $T_{2}$ is closable, we define the \define{tensor product} $(T_{1}\otimes T_{2},\mathcal{D}_{T_{1}\otimes T_{2}})$ to be the closure of this algebraic tensor product~\cite[Section 9.33]{StratilaZsido2019}. Note that the algebraic tensor product is guaranteed to be closable if both $(T_{1},\mathcal{D}_{T_1})$ and $(T_{2},\mathcal{D}_{T_2})$ are closable operators~\cite[VIII.10]{ReedSimonI1980} 
(we do not define the tensor product if the algebraic tensor product $T_{1} \otimes T_{2}$ is not closable). 

A \define{von~Neumann algebra} $\Alg{M}$ is a unital%
\footnote{All von~Neumann algebras in this paper will be unital, so we will henceforth omit saying so.}
 $*$-subalgebra of bounded operators $\mathcal{B}(\mathcal{H})$ on some Hilbert space $\mathcal{H}$ that is closed in the weak operator topology (this implies it is closed in the norm topology as well, thereby making it a Banach algebra, and in fact a $C^*$-algebra). 
If $\Alg{M}\subseteq\mathcal{B}(\mathcal{H})$ is a von Neumann algebra, and if $x\in\Alg{M}$, then there exists a decomposition of $x$ into a complex combination of positive elements as
\be
\label{eqn:positivedecomp}
x=\underbrace{\frac{x+x^*}{2}}_{x_{\Re}}+i \underbrace{\frac{x-x^*}{2i}}_{x_{\Im}}=x_{\Re}^{+}-x_{\Re}^{-}+ix_{\Im}^{+}-ix_{\Im}^{-},
\ee
where $x_{\Re}^{+},x_{\Re}^{-},x_{\Im}^{+},x_{\Im}^{-}\in\Alg{M}$ and $x_{\Re}^{+},x_{\Re}^{-},x_{\Im}^{+},x_{\Im}^{-}\ge0$, where $x_{\Re},x_{\Im}\in\Alg{M}$ are self-adjoint. 
If $\Alg{N}\subseteq\mathcal{B}(\mathcal{K})$ is another von~Neumann algebra, the \define{von Neumann algebraic tensor product} of $\Alg{M}$ and $\Alg{N}$, denoted as $\Alg{M}\overline{\otimes}\Alg{N}$, is the weak closure of (meaning, the smallest closed set in the weak operator topology containing) the algebraic tensor product $\Alg{M}\otimes\Alg{N}$ when viewed as a (unital) $*$-subalgebra of $\mathcal{B}(\mathcal{H}\widehat{\otimes}\mathcal{K})$%
\footnote{This closure can also be viewed as the bicommutant of $\Alg{M}\otimes\Alg{N}$, as discussed later in Footnote~\ref{footnote:bicommutant} and the surrounding text.}. 
 We note that the norm closure of $\Alg{M}\otimes\Alg{N}$ is contained in the weak closure $\Alg{M}\overline{\otimes}\Alg{N}$ (and the containment could be strict). This implies that if we have two bounded linear maps $\mathcal{E}_{1}:\Alg{N}_{1}\to\Alg{M}_{1}$ and $\mathcal{E}_{2}:\Alg{N}_{2}\to\Alg{M}_{2}$ between von~Neumann algebras, then their algebraic tensor product is only guaranteed to be uniquely extendable to the norm closure of $\Alg{N}_{1}\otimes\Alg{N}_{2}$, but \emph{not} necessarily to the weak closure $\Alg{N}_{1}\overline{\otimes}\Alg{N}_{2}$. This subtle point will be fixed by the notion of normality, which is a type of continuity stronger than boundedness, and it will be discussed in the next paragraph. 

Let $\Alg{M}$ and $\Alg{N}$ be von~Neumann algebras. A linear map $\mathcal{E}:\Alg{N}\to\Alg{M}$ is \define{positive} iff $\mathcal{E}(y)\ge0$ whenever $y\ge0$, where $y\ge0$ means that $y$ is positive, i.e., there exists a $b\in\Alg{N}$ such that $y=b^*b$. 
A positive map is automatically bounded (continuous in the operator norm topology).  
A linear map $\mathcal{E}:\Alg{N}\to\Alg{M}$ is \define{unital} iff $\mathcal{E}(1_{\Alg{N}})=1_{\Alg{M}}$. 
A linear map $\mathcal{E}:\Alg{N}\to\Alg{M}$ is \define{completely positive} (CP) iff $\id_{\matr_{n}}\otimes\mathcal{E}:\matr_{n}\otimes\Alg{N}\to\matr_{n}\otimes\Alg{M}$ is positive for all $n\in\N$, where $\matr_{n}$ denotes the algebra of $n\times n$ complex-valued matrices, and $\id_{\matr_{n}}$ denotes the identity linear map on $\matr_{n}$. 
In more detail, if $\Alg{N}\subseteq\mathcal{B}(\mathcal{K})$, then we can view $\matr_{n}\otimes \Alg{N}$ as $\matr_{n}(\Alg{N})\subseteq\mathcal{B}(\mathcal{K}^{\oplus n})$, the algebra of $n\times n$ block matrices with entries in $\Alg{N}$, via the direct sum of $\mathcal{K}$ with itself~\cite{Paulsen02}. Note that $\matr_{n}\otimes \Alg{N}\cong \matr_{n}(\Alg{N})$ is already a von Neumann algebra inside of $\mathcal{B}(\C^{n}\otimes\mathcal{K})\cong\mathcal{B}(\mathcal{K}^{\oplus n})$, i.e., no completion of the tensor product is needed because every element is, by definition, (finite) linear combinations of the form 
\be
\label{eqn:yinMnN}
\sum_{j,k=1}^{n} E_{jk}\otimes y_{jk}
=\begin{bmatrix} y_{11} & \cdots & y_{1n} \\
 \vdots & & \vdots \\
 y_{n1} & \cdots & y_{nn} 
\end{bmatrix}
, 
\ee
where $E_{jk}$ is the standard matrix element with $1$ in the $jk$ entry and $0$ everywhere else (written as $|j\>\<k|$ in Dirac notation). 
If we set $y\in \matr_{n}(\Alg{N})$ to be the element given by~\eqref{eqn:yinMnN}, then 
\[
\begin{split}
y^*y&=
\begin{bmatrix}
y_{11}^* & 0 & \cdots & 0 \\
\vdots & \vdots & & \vdots \\
y_{1n}^* & 0 & \cdots & 0 
\end{bmatrix}
\begin{bmatrix}
y_{11} & \cdots & y_{1n} \\
0 & \cdots & 0 \\
\vdots & & \vdots \\
0 & \cdots & 0 
\end{bmatrix}
+ \cdots + 
\begin{bmatrix}
0 & \cdots & 0 & y_{n1}^* \\
\vdots & & \vdots & \vdots \\
0 & \cdots & 0 & y_{nn}^* 
\end{bmatrix}
\begin{bmatrix}
0 & \cdots & 0 \\
\vdots & & \vdots \\
0 & \cdots & 0 \\
y_{n1} & \cdots & y_{nn}
\end{bmatrix} \\
&=
\begin{bmatrix}
y_{11}^* & 0 & \cdots & 0 \\
\vdots & \vdots & & \vdots \\
y_{1n}^* & 0 & \cdots & 0 
\end{bmatrix}
\begin{bmatrix}
y_{11} & \cdots & y_{1n} \\
0 & \cdots & 0 \\
\vdots & & \vdots \\
0 & \cdots & 0 
\end{bmatrix}
+ \cdots +
\begin{bmatrix}
y_{n1}^* & 0 & \cdots & 0 \\
\vdots & \vdots & & \vdots \\
y_{nn}^* & 0 & \cdots & 0 
\end{bmatrix}
\begin{bmatrix}
y_{n1} & \cdots & y_{nn} \\
0 & \cdots & 0 \\
\vdots & & \vdots \\
0 & \cdots & 0 
\end{bmatrix}
\end{split}
\]
Therefore, since every positive element can be written in this way, complete positivity of $\mathcal{E}$ is equivalent to 
\be
\label{eqn:whatCPmeans}
\sum_{j,k=1}^{n} \<v_{j}, \mathcal{E}(y_{j}^* y_{k}) v_{k} \>\ge0
\ee
for all $y_{1},\dots,y_{n}\in\Alg{N}$, all $v_{1},\dots,v_{n}\in\mathcal{K}$, and all $n\in\N$.
A completely positive unital map is abbreviated as a \define{CPU map}. The composite of two CPU maps is also CPU. 
A positive bounded linear functional $\chi:\Alg{M}\to\C$ is called a \define{finite weight}. A \define{state} on $\Alg{M}$ is a positive unital linear functional $\omega:\Alg{M}\to\C$. 
Note that a state $\omega:\Alg{M}\to\C$ is automatically a finite weight, i.e., it is bounded, because $\omega(x^*x)\le\omega\big(\lVert x^*x\rVert 1_{\Alg{M}}\big)=\lVert x\rVert^2 \omega(1_{\Alg{M}})=\lVert x\rVert^2$ for all $x\in\Alg{M}$.
A finite weight $\chi:\Alg{M}\to\C$ is \define{faithful} whenever given some $x\in\Alg{M}$, then $\chi(x^*x)=0$ implies $x=0$.
Besides being bounded, there is an additional assumption on continuity that is often used for linear maps between von~Neumann algebras $\Alg{M}$ and $\Alg{N}$. An \define{increasing net of positive elements} in a von~Neumann algebra $\Alg{N}$ is a net $y:A\to\Alg{N}$ satisfying $y_{\alpha}\ge0$ for all $\alpha\in A$ and $y_{\alpha}\le y_{\beta}$ whenever $\alpha\le\beta$ in $A$. A positive linear map $\mathcal{E}:\Alg{N}\to\Alg{M}$ is \define{normal} whenever $\sup_{\alpha}\mathcal{E}(y_\alpha) =\mathcal{E}(\sup_{\alpha}y_{\alpha})$ for all increasing nets of positive elements $y:A\to\Alg{N}$ such that $\sup \alpha$ exists in $A$. It is immediate from this definition that the composite of two normal maps is normal. 
Although it is true that a normal positive map is weakly continuous, not every weakly continuous positive map is normal. 
A normal completely positive unital map is abbreviated as a \define{NCPU map}. 
Given two NCPU maps $\mathcal{E}_{1}:\Alg{N}_{1}\to\Alg{M}_{1}$ and $\mathcal{E}_{2}:\Alg{N}_{2}\to\Alg{M}_{2}$, their algebraic tensor product extends uniquely to an NCPU map $\mathcal{E}_{1}\otimes\mathcal{E}_{2}:\Alg{M}_{1}\overline{\otimes}\Alg{M}_{2}\to\Alg{N}_{1}\overline{\otimes}\Alg{N}_{2}$~\cite[Proposition 5.13]{Tak1}. 
We note that although the tensor product of two CPU maps does uniquely extend to the \emph{norm completion} of the tensor products of the von Neumann algebras~\cite[Theorem 3.5.3]{BrownOzawa2008}, if either of the constituent CPU maps is not normal, the resulting tensor product will not uniquely extend to the von Neumann algebra tensor product completion. 

Given a finite weight $\chi:\Alg{M}\to\C$ on a von~Neumann algebra $\Alg{M}$, let $\<\;\cdot\;,\;\cdot\;\>_{\chi}$ denote the sesquilinear form on $\Alg{M}$ by $\<x,y\>_{\chi}:=\chi(x^*y)$. Let $N_{\chi}:=\{x\in\Alg{M}\,:\,\chi(x^*x)=0\}$ denote the \define{left nullspace} of $\chi$. It is a left ideal in $\Alg{M}$ in the sense that it is a subalgebra of $\Alg{M}$ that moreover satisfies the condition that $yx\in N_{\chi}$ whenever $x\in N_{\chi}$ and $y\in\Alg{M}$. 
The sesquilinear form $\<\,\cdot\,,\,\cdot\,\>_{\chi}$ descends to the quotient space $\Alg{M}/N_{\chi}$, where it is nondegenerate and hence defines an inner product. Elements of $\Alg{M}/N_{\chi}$ will often be denoted as $[x]$ with $x$ a representative in $\Alg{M}$. In case there are two such finite weights on a given algebra, a subscript such as in $[x]_{\chi}$ will be used to distinguish between the two. 
Let $\mathcal{H}_{\chi}$ be the completion of $\Alg{M}/N_{\chi}$ under this inner product so that $\Alg{M}/N_{\chi}$ is a dense subspace of $\mathcal{H}_{\chi}$. Let $\pi_{\chi}:\Alg{M}\to\mathcal{B}(\mathcal{H}_{\chi})$ be the unital $*$-homomorphism sending $x$ to the left multiplication operator $\pi_{\chi}(x)$, i.e., $\pi_{\chi}(x)[y]:=[xy]$ for all $[y]\in\Alg{M}/N_{\chi}\subseteq\mathcal{H}_{\chi}$, which is indeed a bounded operator on $\mathcal{H}_{\chi}$. It follows that $\big\<[1_{\Alg{M}}],\pi_{\chi}(x)[1_{\Alg{M}}]\big\>_{\chi}=\chi(x)$ for all $x\in\Alg{M}$. The tuple $(\mathcal{H}_{\chi},\pi_{\chi},[1_{\Alg{M}}])$ is called the \define{GNS representation} associated with $(\Alg{M},\chi)$, where GNS stands for Gelfand--Naimark--Segal~\cites{PaGNS,GN43,Se47}. If $\chi$ is, in addition, normal, then the $*$-homomorphism $\pi_{\chi}$ is normal and $\pi_{\chi}(\Alg{M})$ is a von~Neumann subalgebra of $\mathcal{B}(\mathcal{H}_{\chi})$. 

If $S\subseteq\mathcal{B}(\mathcal{H})$ is a subset, the \define{commutant} of $S$ in $\mathcal{B}(\mathcal{H})$ is the set $S'\subseteq\mathcal{B}(\mathcal{H})$ of elements $x\in\mathcal{B}(\mathcal{H})$ that commute with all elements in $S$, i.e., $S':=\{x\in\mathcal{B}(\mathcal{H})\,:\,xy=yx \text{ for all }y\in S\}$.
The bicommutant theorem states that if $\Alg{S}\subseteq\mathcal{B}(\mathcal{H})$ is a (unital) $*$-algebra, then $\Alg{S}''$ is the smallest von~Neumann algebra containing $\Alg{S}$~\cite[Section I.1.1]{Dix81}, and, moreover, $\Alg{S}''$ is equal to the topological closure of $\Alg{S}$ in any one of the weak operator, strong operator, ultra-weak (also called sigma-weak or $\sigma$-weak in the literature~\cite{Tak1}), or ultra-strong topologies~\cite[Section I.3.4]{Dix81}%
\footnote{\label{footnote:bicommutant} Although we only discussed the weak topology here, we state this result because when we reference some results in the literature, those results may be stated using any of these topologies---we stress that it is important that $\Alg{S}$ is a unital $*$-subalgebra of $\mathcal{B}(\mathcal{H})$, and not just a subset.}.
Given two finite weights $\chi,\eta:\Alg{M}\to\C$, the weight $\eta$ is \define{bounded by} the weight $\chi$, written as $\eta\le\chi$, iff $\chi-\eta$ is a positive linear map. More generally, $\eta$ is \define{uniformly bounded by} $\chi$, written as $\eta\le_{u}\chi$, iff there exists a $\lambda>0$ such that $\eta\le\lambda\chi$. In such a case, let $(\mathcal{H}_{\eta},\pi_{\eta},[1_{\Alg{M}}])$ be its GNS representation. Then there exists a unique element $\frac{d\eta}{d\chi}\in\pi_{\chi}(\Alg{M})'\subseteq\mathcal{B}(\mathcal{H}_{\chi})$ such that 
\[
0\le\frac{d\eta}{d\chi}\le\lambda\id_{\mathcal{H}_{\chi}}\;\;\text{ for some $\lambda>0$}\;\;
\quad\text{ and }\quad
\eta(x)=\left\<[1_{\Alg{M}}],\frac{d\eta}{d\chi}\pi_{\chi}(x)[1_{\Alg{M}}]\right\>_{\chi} \;\;\text{for all }x\in\Alg{M}.
\]
This element $\frac{d\eta}{d\chi}$ is called the \define{Radon--Nikodym derivative} of $\eta$ with respect to $\chi$, and it is constructed as follows (see~\cite[Lemma 5.19]{StratilaZsido2019}, \cite[Proposition 2.3]{StratilaModular2020}, or Refs.~\cites{Gh10,BeSt86} for a version more directly related to our construction). 
First note that the identity map $\id_{\Alg{M}}:\Alg{M}\to\Alg{M}$ descends to a (well-defined) bounded linear map $G_{\eta,\chi}:\Alg{M}/N_{\chi}\to\Alg{M}/N_{\eta}$ because 
\[
\lVert x\rVert_{\eta}^2 = \eta(x^* x) \le \lambda \chi(x^* x) = \lambda \lVert x\rVert_{\chi}^2
\]
for all $x\in\Alg{M}$ implies $\lVert G_{\eta,\chi}\rVert\le \lambda.$
Let $G_{\eta,\chi}:\mathcal{H}_{\chi}\to\mathcal{H}_{\eta}$ also denote its unique extension to the completions.  
The bounded operator $\frac{d\eta}{d\chi}:=G_{\eta,\chi}^* G_{\eta,\chi}\in\mathcal{B}(\mathcal{H}_{\chi})$ satisfies all of the required properties. If $\eta_1$ and $\eta_2$ are two finite weights on $\Alg{M}$ with $\eta_1,\eta_2\le_{u}\chi$ and if $c\ge0$, then it immediately follows from uniqueness of the Radon--Nikodym derivative and linearity of $\eta_{1}$ and $\eta_{2}$ that 
\[
\frac{d(\eta_1+c\eta_2)}{d\chi}=\frac{d\eta_1}{d\chi}+c\frac{d\eta_2}{d\chi}.
\] 

Let $\Alg{M}\subseteq\mathcal{B}(\mathcal{H})$ be a von~Neumann subalgebra of $\mathcal{B}(\mathcal{H})$ and let $\Omega\in\mathcal{H}$. The vector $\Omega$ is \define{cyclic} for $\Alg{M}$ iff the set $\Alg{M}\Omega:=\{x\Omega\,:\,x\in\Alg{M}\}$ is dense in $\mathcal{H}$ (i.e., its norm closure is all of $\mathcal{H}$). 
The notation $L^{2}(\Alg{M},\Omega)$ is sometimes used in place of $\mathcal{H}$.
The vector $\Omega$ is \define{separating} for $\Alg{M}$ iff whenever $x\in\Alg{M}$ satisfies $x\Omega=0$, then $x=0$ (equivalently, $\Omega$ is cyclic for the commutant $\Alg{M}'$). In other words, $\Omega$ is cyclic and separating if and only if $x\Omega=0$ for $x\in\Alg{M}$ implies $x=0$ and $x'\Omega=0$ for $x'\in\Alg{M}'$ implies $x'=0$. As an example, if $\omega:\Alg{A}\to\C$ is a faithful normal state on a von~Neumann algebra $\Alg{A}$, then the vector $\Omega=[1_{\Alg{A}}]$ from the GNS representation $(\mathcal{H}_{\omega},\pi_{\omega},[1_{\Alg{A}}])$ is cyclic and separating for $\Alg{M}=\pi_{\omega}(\Alg{A})$ inside $\mathcal{B}(\mathcal{H}_{\omega})$. 

Given a normal faithful state $\omega$ on a von~Neumann algebra $\Alg{A}$, by the previous paragraphs, there always exists a normal faithful representation $\pi$ of $\Alg{A}$ on some Hilbert space $\mathcal{H}$ together with a cyclic and separating normalized vector $\Omega\in\mathcal{H}$, meaning $\lVert\Omega\rVert=1$, such that $\<\Omega,\pi(a)\,\Omega\>=\omega(a)$ for all $a\in\Alg{A}$. 
Given a vector $\Omega\in\mathcal{H}$ that is cyclic and separating for $\Alg{M}:=\pi(\Alg{A})$, define the unique conjugate-linear operator $S_{0}$ on the domain $\mathcal{D}_{S_{0}}=\Alg{M}\Omega$ specified by $S_{0}(x\,\Omega):=x^{*}\Omega$ for all $x\in\Alg{M}$. Note that the domain $\mathcal{D}_{S_{0}}$ is dense in $\mathcal{H}$ by the assumption that $\Omega$ is cyclic. Although the operator $S_{0}$ is not in general bounded,%
\footnote{When we look at the finite-dimensional case, we will see why $S_{0}$ is in general not bounded for infinite-dimensional von~Neumann algebras. 
}
it is defined on a dense domain, it is invertible on its domain since $S_{0}^{-1}=S_{0}$, and it is closable. Its closure $S:=\overline{S_{0}}$ is called the \define{Tomita operator}, and the positive operator $\Delta:=S^*S$ is called the \define{modular operator} associated with $(\Alg{M},\Omega)$. Since $S$ is a closed and invertible conjugate-linear operator defined on a dense domain, it admits a polar decomposition $S=J\Delta^{1/2}$, where $J:\mathcal{H}\to\mathcal{H}$ is a conjugate-linear partial isometry and $\Delta^{1/2}$ is the unique positive square root of $\Delta$, which is an invertible positive operator on $\mathcal{H}$ with dense domain $\mathcal{D}_{\Delta^{1/2}}=\mathcal{D}_{S}$ that contains $\mathcal{D}_{S_{0}}=\Alg{M}\,\Omega$. In fact, $J$ is a conjugate-linear unitary operator satisfying $J\Omega=\Omega$, $J^2=\id_{\mathcal{H}}$, $J^*=J$, and $J\mathcal{M}J=\mathcal{M}'$ (and hence also $J\mathcal{M}'J=\mathcal{M}$). The operator $J$ is called the \define{modular conjugation} associated with $(\Alg{M},\Omega)$. 
We also note that $\Alg{M}\,\Omega\subseteq\mathcal{D}_{S}=\mathcal{D}_{\Delta^{1/2}}\subseteq\mathcal{D}_{\Delta^{1/4}}$, where $\Delta^{1/4}$ is the unique positive square root of $\Delta^{1/2}$. 
Since $\Delta$ is an (invertible) positive self-adjoint operator defined on a dense domain $\mathcal{D}_{\Delta}\subseteq\mathcal{H}$, the functional calculus guarantees that $\Delta^{it}$ defines a unitary operator on $\mathcal{H}$ for all $t\in\R$~\cite[Section 9.15]{StratilaZsido2019}. 
To make sense of $\Delta^{it}$ in a bit more detail, recall that the complex logarithm is defined on the set $\C\setminus(-\infty,0]$ 
by $\ln(z)=\ln|z|+\arg(z)$~\cite{Rudin87}, where if $z=re^{i\theta}$, with $r\in(0,\infty)$ and $-\pi<\theta<\pi$, then $|z|=r$ and $\arg(z)=\theta$. Then define
\[
\big(\C\setminus(-\infty,0) \big)\times \C \ni (\lambda,\alpha)\mapsto z^{\alpha}
:=\begin{cases}
e^{\alpha \ln(z) } &\mbox{ if $z\ne0$ }\\
0 &\mbox{ if $z=0$ }
\end{cases}
\]
In particular, for each $t\in\R$ and $\lambda\in[0,\infty)$, the expression $\lambda^{it}$ is defined, and if $\lambda\in(0,\infty)$, then $|\lambda^{it}|=|e^{it\ln(\lambda)}|=1$ for all $t\in\R$. Thus, since $\Delta$ is a positive operator, it does not contain $0$ in its spectrum, \cite[Corollary 9.13]{StratilaZsido2019} implies that $\Delta^{it}=e^{it \log(\Delta)}$ is unitary for all $t\in\R$. 
For each $t\in\R$, let $\Ad_{\Delta^{it}}$ be the operator on $\mathcal{B}(\mathcal{H})$ defined by $\Delta^{it} T \Delta^{-it}$ for all $T\in\mathcal{B}(\mathcal{H})$. The Tomita--Takesaki theorem states that $\Ad_{\Delta^{it}}(\Alg{M})\subseteq\Alg{M}$ for all $t\in\R$~\cites{Takesaki1970,Sorce2024}. The restriction of $\Ad_{\Delta^{it}}$ to $\Alg{M}$ is denoted by $\sigma_{t}:\Alg{M}\to\Alg{M}$ and is called the \define{modular automorphism group} associated with $(\Alg{M},\Omega)$. The reason for the group terminology is because $\sigma_{s+t}=\sigma_{s}\circ\sigma_{t}$ for all $s,t\in\R$. Since $\pi$ is a $*$-isomorphism from $\Alg{A}$ to $\Alg{M}$, the modular automorphism group can equivalently be viewed as an automorphism on $\Alg{A}$. We will write $\sigma^{\omega}_{t}:\Alg{A}\to\Alg{A}$ and $\sigma^{\Omega}_{t}:\Alg{M}\to\Alg{M}$ if we want to distinguish between the two actions. Finally, an element $x\in\Alg{M}$ is called \define{entire analytic} for the modular automorphism group $\sigma^{\Omega}$ iff there exists a function $f^{\Omega}_{x}:\C\to\Alg{M}$ such that $f^{\Omega}_{x}(t)=\sigma_{t}^{\Omega}(x)$ for all $t\in\R$ and $\big\<v,f^{\Omega}_{x}(z) w\big\>$ is an analytic function in $z\in\C$ (an entire function) for all $v,w\in\mathcal{H}$~\cite[Definition 2.5.20 and Proposition 2.5.22]{BrRo1}. The set of entire analytic elements in $\Alg{M}$ is a (unital) $*$-algebra of $\mathcal{B}(\mathcal{H})$%
\footnote{It is immediate that $1_{\Alg{M}}=\id_{\mathcal{H}}$ is entire analytic, since $f^{\Omega}_{1_{\Alg{M}}}(z)=1_{\Alg{M}}$ for all $z\in\C$. If $x$ is entire analytic, then $f^{\Omega}_{x^*}(t):=f^{\Omega}_{x}(t)^*$ },
 and its weak operator closure is all of $\Alg{M}$~\cite[Proposition 2.5.22]{BrRo1}. 

\bx
\label{ex:fdcaseintro}
This example will be referred to throughout the paper.
We illustrate some of these definitions in the case where $\Alg{A}=\matr_{m}$, $\Alg{M}=\matr_{m}\otimes\mathds{1}_{m}$, $\mathcal{H}=\C^{m}\otimes\C^{m}$, and $\Alg{M}'=\mathds{1}_{m}\otimes\matr_{m}$. The algebra $\Alg{A}$ can be viewed as the algebra of observables associated with Alice, while $\Alg{M}$ has the same interpretation but when viewed inside of a specific larger system associated with the purification of a state on $\Alg{A}$ (which will be introduced shortly). The algebra $\Alg{M}'$ is then associated with the environment in this larger system. Since $\Alg{M}=\matr_{m}\otimes\mathds{1}_{m}$ and $\Alg{M}'=\mathds{1}_{m}\otimes\matr_{m}$, this means that every $x\in\Alg{M}$ is of the form $x=a\otimes\mathds{1}_{m}$ with $a\in\matr_{m}$, and similarly every $x'\in\Alg{M}'$ is of the form $x'=\mathds{1}_{m}\otimes a'$ with $a'\in\matr_{m}$. In what follows, the conjugate transpose (adjoint) of a matrix $a$ will be written as $a^{\dag}$. Vectors $\mathbf{u}$ will be treated as column vectors and their adjoints $\mathbf{u}^{\dag}$ as row vectors. We mention that Dirac notation is often used in the physics literature, so for example, $|\mathbf{u}\>:=\mathbf{u}$ and $\<\mathbf{u}|:=\mathbf{u}^{\dag}$, and $\<\mathbf{u}|\mathbf{v}\>:=\mathbf{u}^{\dag}\mathbf{v}=\<\mathbf{u},\mathbf{v}\>$. Every faithful state on $\matr_{m}$ can be expressed as $\omega(a)=\Tr[\rho a]$ for all $a\in\matr_{m}$ for a unique full rank density operator $\rho$, i.e., $\rho$ is a positive definite matrix such that $\Tr[\rho]=1$. Pick an orthonormal basis $\{\mathbf{u}_{i}\}$ of $\C^{m}$ such that 
\[
\rho=\sum_{i=1}^{m}p_{i}\,\mathbf{u}_{i}\mathbf{u}_{i}^{\dag}\equiv\sum_{i=1}^{m}p_{i}\,|\mathbf{u}_{i}\>\<\mathbf{u}_{i}|
\]
and $\{p_{i}\}$ defines a nowhere vanishing probability distribution on $\{1,\dots,m\}$ (in other words, find an orthonormal diagonalization of $\rho$). Set 
\be
\label{eqn:OmegaFDcase}
\Omega:=\sum_{i=1}^{m}\sqrt{p_{i}}\,\mathbf{u}_{i}\otimes\mathbf{u}_{i}
\equiv\sum_{i=1}^{m}\sqrt{p_{i}}\,|\mathbf{u}_{i}\>\otimes|\mathbf{u}_{i}\>,
\ee
which is a vector in $\mathcal{H}=\C^{m}\otimes\C^{m}$, sometimes called the \emph{canonical purification} (cf.~\cite[Exercise 5.2]{WildeQIT16} or~\cite[Section 6.3]{Holevo2019}), the \emph{thermofield double} (cf.~\cite{ValdiviaMera2025}), or the \emph{thermal vacuum} (cf.~\cite{Nair2015thermofield}). Then one can quickly check that $\omega(a)=\<\Omega,(a\otimes\mathds{1}_{m})\Omega\>$ for all $a\in\matr_{m}$. Moreover, the vector $\Omega$ is cyclic for $\Alg{M}$ because if $\Psi=\sum_{i,j}\psi_{ij}\,\mathbf{u}_{i}\otimes\mathbf{u}_{j}$ is a general vector in $\mathcal{H}$ for some $\psi_{ij}\in\C$ (this is so because $\{\mathbf{u}_{i}\otimes\mathbf{u}_{j}\}$ is an orthonormal basis for $\mathcal{H}$), upon setting $a=\sum_{i,j}\frac{\psi_{ij}}{\sqrt{p_{j}}}\mathbf{u}_{i}\mathbf{u}_{j}^{\dag}$, one sees that 
\[
(a\otimes\mathds{1}_{m})\Omega
=\sum_{i,j,k}\frac{\psi_{ij}}{\sqrt{p_{j}}}\sqrt{p_{k}}(\mathbf{u}_{i}\mathbf{u}_{j}^{\dag}\mathbf{u}_{k}\otimes\mathbf{u}_{k})
=\sum_{i,j,k}\frac{\psi_{ij}\sqrt{p_{k}}}{\sqrt{p_{j}}}\delta_{jk} \mathbf{u}_{i}\otimes\mathbf{u}_{k}
=\sum_{i,j}\psi_{ij}\mathbf{u}_{i}\otimes\mathbf{u}_{j}
=\Psi.
\]
Moreover, $\Omega$ is separating for $\Alg{M}$ because if $x=a\otimes\mathds{1}_{m}$, where $a=\sum_{i,j}a_{ij}\mathbf{u}_{i}\mathbf{u}_{j}^{\dag}$, satisfies $x\,\Omega=0$, then this yields
\[
0=x\,\Omega=\sum_{i,j} a_{ij}\sqrt{p_{j}}\,\mathbf{u}_{i}\otimes\mathbf{u}_{j}
\]
and the only way this is zero is if each $a_{ij}\sqrt{p_j}$ is zero, but since $p_{j}\ne0$, this forces $a_{ij}=0$ for all $i,j,$ so that $x=0$. Hence, $\Omega$ is indeed cyclic and separating for $\Alg{M}$. 

Before moving on to looking at the modular theory for this general setup, we take a brief moment to look at a particular example for qubits. If $m=2$ and $\rho=\frac{\mathds{1}_{2}}{2}=\frac{1}{2}\big(|0\>\<0|+|1\>\<1|\big)$, where we use the notation $|0\>=(1,0)$ and $|1\>=(0,1)$, then $|\Omega\>=\frac{1}{\sqrt{2}}\big(|00\>+|11\>\big)$, where $|jk\>:=|j\>\otimes|k\>$ for $j,k\in\{0,1\}$. Thus, $|\Omega\>$ is one of the four Bell states~\cite{NiCh11}. Another way to see that $|\Omega\>$ is separating is to notice that 
\[
\big(|j\>\<k|\otimes\mathds{1}_{2}\big)|\Omega\>=\frac{1}{\sqrt{2}}|jk\>
\]
for all $j,k\in\{0,1\}$, which forms a basis for $\C^{2}\otimes\C^{2}$. Thus, all of $\C^{2}\otimes\C^{2}$ is obtained by acting on $|\Omega\>$ by $\matr_{2}\otimes\mathds{1}_{2}$ since every $a\in\matr_{2}$ can be written as $a=\sum_{j,k}a_{jk}|j\>\<k|$ for some $a_{jk}\in\C$. 

We now return to the general finite-dimensional case and next investigate the operator $S:\mathcal{H}\to\mathcal{H}$, which is given on vectors of the form $(a\otimes\mathds{1}_{m})\,\Omega$ by 
\[
S\big((a\otimes\mathds{1}_{m})\,\Omega\big):=(a^{\dag}\otimes\mathds{1}_{m})\Omega
=\sum_{i,j} \overline{a_{ji}}\sqrt{p_{j}}\,\mathbf{u}_{i}\otimes\mathbf{u}_{j}
.
\]
We will provide a more explicit formula for $S$ by actually obtaining its polar decomposition first. Let $J:\mathcal{H}\to\mathcal{H}$ be the conjugate-linear operator uniquely defined on an arbitrary vector $\Psi=\sum_{i,j}\psi_{ij}\,\mathbf{u}_{i}\otimes\mathbf{u}_{j}$ by 
\be
\label{eqn:JoperatorFDcase}
J(\Psi):=\sum_{i,j}\overline{\psi_{ji}}\,\mathbf{u}_{i}\otimes\mathbf{u}_{j}.
\ee
In other words, if the $\{\psi_{ij}\}$ coefficients are viewed as the entries of a matrix, then $J$ takes the conjugate-transpose of that matrix. It is immediate from this definition that $J\Omega=\Omega$, $J^2=\id_{\mathcal{H}}$, and $J^*=J$ (in particular, $J$ is a conjugate-linear unitary operator). Moreover, we claim that 
\[
J(a\otimes\mathds{1}_{m})J=\mathds{1}_{m}\otimes\overline{a}, 
\]
where $\overline{a}$ denotes the matrix whose entries in the $\mathbf{u}_{i}$ basis are given by complex conjugation, i.e., 
\[
\overline{a}\equiv\overline{\sum_{k,l}a_{kl}\mathbf{u}_{k}\mathbf{u}_{l}^{\dag}}:=\sum_{k,l}\overline{a_{kl}}\mathbf{u}_{k}\mathbf{u}_{l}^{\dag}.
\]
To see that $J(a\otimes\mathds{1}_{m})J=\mathds{1}_{m}\otimes\overline{a}$, we show that this equality holds when acting on an arbitrary vector $\Psi$:   
\begingroup
\allowdisplaybreaks
\begin{align}
J(a\otimes\mathds{1}_{m})J\Psi&=J \sum_{k,l}\big(a_{kl}\mathbf{u}_{k}\mathbf{u}_{l}^{\dag}\otimes\mathds{1}_{m}\big)\sum_{i,j}\overline{\psi_{ji}}\,\mathbf{u}_{i}\otimes\mathbf{u}_{j} \nonumber \\
&=J\left(\sum_{j,k}\left(\sum_{i}a_{ki}\overline{\psi_{ji}}\right)\mathbf{u}_{k}\otimes\mathbf{u}_{j}\right) \nonumber \\
&=\sum_{i,j,k}\overline{a_{ji}}\psi_{ki}\mathbf{u}_{k}\otimes\mathbf{u}_{j} \nonumber\\
&=\sum_{i,j,k}\overline{a_{kj}}\psi_{ij}\mathbf{u}_{i}\otimes\mathbf{u}_{k}\nonumber\\
&=\sum_{k,l}\big(\mathds{1}_{m}\otimes \overline{a_{kl}}\mathbf{u}_{k}\mathbf{u}_{l}^{\dag}\big)\sum_{i,j}\psi_{ij}\mathbf{u}_{i}\otimes\mathbf{u}_{j} \nonumber\\
&=(\mathds{1}_{m}\otimes\overline{a})\Psi. \nonumber
\end{align}
\endgroup
Now, let $\Delta^{1/2}:\mathcal{H}\to\mathcal{H}$ be defined by its action on a vector $\Psi$ by 
\be
\label{eqn:DeltahalfFDcase}
\Delta^{1/2}(\Psi):=\left(\rho^{\frac{1}{2}}\otimes\rho^{-\frac{1}{2}}\right)\Psi
=\sum_{i,j}\frac{\psi_{ij}\sqrt{p_{i}}}{\sqrt{p_{j}}}\mathbf{u}_{i}\otimes\mathbf{u}_{j}.
\ee
We next check that $S=J\Delta^{1/2}$ indeed provides a polar decomposition of $S$. First, the identity $S=J\Delta^{1/2}$ is verified by acting on a general element in $\mathcal{H}$ of the form $x\,\Omega$ with $x=a\otimes\mathds{1}_{m}$ via 
\[
J\Delta^{1/2}(x\,\Omega)
=J\Delta^{1/2}\sum_{i,j}a_{ij}\sqrt{p_j}\,\mathbf{u}_{i}\otimes\mathbf{u}_{j}
=J\sum_{i,j}a_{ij}\sqrt{p_i}\,\mathbf{u}_{i}\otimes\mathbf{u}_{j}
=\sum_{i,j}\overline{a_{ji}}\sqrt{p_j}\,\mathbf{u}_{i}\otimes\mathbf{u}_{j}
=S(x\,\Omega).
\]
A quick calculation shows that $\Delta^{1/2}$ is a self-adjoint, which means $\<\Delta^{1/2}\Phi,\Psi\>=\<\Phi,\Delta^{1/2}\Psi\>$ for all $\Psi,\Phi\in\mathcal{H}$.
From this, it is immediate that $\Delta^{1/2}$ is a positive operator by the eigenvalue test for positivity since its eigenvalues are precisely of the form $\sqrt{\frac{p_i}{p_{j}}}$ with the corresponding eigenvector being $\mathbf{u}_{i}\otimes\mathbf{u}_{j}$. 

From the finite-dimensional case, we can see why $S$ is in general unbounded for infinite-dimensional von~Neumann algebras. Namely, the singular values of $S$, which are the square-roots of the eigenvalues of $S^*S=\Delta$, are precisely given by $\sqrt{\frac{p_i}{p_{j}}}$. Since density matrices may have arbitrary eigenvalues describing probability distributions, consider an infinite-dimensional density matrix $\rho$ with $k^{\text{th}}$ eigenvalue given by $p_{k}=\frac{6}{\pi^2 k^2}$. Then $\sum_{k=1}^{\infty}p_{k}=1$ so that $\{p_{k}\}$ does indeed define a probability distribution on $\N$. Hence, $\sqrt{\frac{p_i}{p_j}}$ can be arbitrarily large, and therefore $\Delta^{1/2}$ is a positive, but not bounded, operator. 
\ex

Let $\Alg{A}_1$ and $\Alg{A}_2$ be two von~Neumann algebras, let $\omega_1$ and $\omega_2$ be faithful states on $\Alg{A}_1$ and $\Alg{A}_2$, respectively. Let $(\mathcal{H}_{i},\pi_{i},\Omega_{i})$ denote their associated GNS representations and set $\Alg{M}_{i}=\pi_{i}(\Alg{A}_i)$ for $i=1,2$. Then $(\mathcal{H}_{1}\widehat{\otimes}\mathcal{H}_{2},\pi_{1}\otimes\pi_{2},\Omega_{1}\otimes\Omega_{2})$ is a GNS representations of $(\Alg{A}_{1}\overline{\otimes}\Alg{A}_{2},\omega_{1}\otimes\omega_{2})$. Moreover, $(\Alg{M}_{1}\overline{\otimes}\Alg{M}_{2})'=\Alg{M}_{1}'\overline{\otimes}\Alg{M}_{2}'$ inside $\mathcal{B}(\mathcal{H}_{1}\widehat{\otimes}\mathcal{H}_{2})$ and $\Omega_1\otimes\Omega_2$ is cyclic and separating for $\Alg{M}_{1}\overline{\otimes}\Alg{M}_{2}$ inside $\mathcal{B}(\mathcal{H}_{1}\widehat{\otimes}\mathcal{H}_{2})$. If $S_{1}, \Delta_{1}, J_{1},$ and $S_{2},\Delta_{2},J_{2},$ are the Tomita operators, modular operators, and modular conjugations associated with $(\Alg{M}_{1},\Omega_1)$ and $(\Alg{M}_{2},\Omega_2)$, respectively, and if $S_{12}, \Delta_{12}, J_{12}$ are the Tomita operators, modular operators, and modular conjugations associated with $(\Alg{M}_{1}\overline{\otimes}\Alg{M}_{2},\Omega_1\otimes\Omega_2)$, then $S_{12}=S_{1}\otimes S_{2}$ (this is well-defined and closable precisely because each $S_{1}$ and $S_{2}$ are densely-defined and closable)~\cite[Section 10.7]{StratilaZsido2019}. Furthermore, the fact that $\Delta_{12}=\Delta_{1}\otimes\Delta_{2}$ and $J_{12}=J_{1}\otimes J_{2}$ follows from the uniqueness of polar decomposition. Indeed, on the one hand, we have $S_{12}=J_{12} \Delta_{12}^{1/2}$, and on the other hand, we have $S_{12}=J_1 \Delta_1^{\frac{1}{2}} \otimes J_2 \Delta_2^{\frac{1}{2}}=(J_1 \otimes J_2)(\Delta_{1}^{1/2} \otimes \Delta_2^{1/2})$.

\vspace{3mm}
\textbf{AI usage declaration.}
Microsoft 365 Copilot and Google Gemini were used to search for references, to generate BibTex data, to learn background information, for exploratory mathematical assistance, and for proofreading.  
All results generated by this method were checked. 
All of the writing was done solely by the authors.

\addcontentsline{toc}{section}{\numberline{}Bibliography}
\bibliographystyle{eptcs}
\bibliography{references}

\begin{bibdiv}
\begin{biblist}

\bib{AcCe82}{article}{
      author={Accardi, Luigi},
      author={Cecchini, Carlo},
       title={Conditional expectations in von {N}eumann algebras and a theorem
  of {T}akesaki},
        date={1982},
        ISSN={0022-1236},
     journal={J. Funct. Anal.},
      volume={45},
      number={2},
       pages={245\ndash 273},
         url={https://doi.org/10.1016/0022-1236(82)90022-2},
}

\bib{AhmadKlinger2025}{article}{
      author={Ahmad, Shadi~Ali},
      author={Klinger, Marc~S.},
       title={Emergent geometry from quantum probability},
        date={2025May},
     journal={Phys. Rev. D},
      volume={111},
       pages={105015},
      eprint={2411.07288},
         url={https://link.aps.org/doi/10.1103/PhysRevD.111.105015},
}

\bib{AWWW18}{article}{
      author={Alhambra, \'Alvaro~M.},
      author={Wehner, Stephanie},
      author={Wilde, Mark~M.},
      author={Woods, Mischa~P.},
       title={Work and reversibility in quantum thermodynamics},
        date={2018},
     journal={Phys. Rev. A},
      volume={97},
       pages={062114},
      eprint={1506.08145},
         url={https://link.aps.org/doi/10.1103/PhysRevA.97.062114},
}

\bib{Anantharaman2006}{article}{
      author={Anantharaman-Delaroche, Claire},
       title={On ergodic theorems for free group actions on noncommutative
  spaces},
        date={2006},
     journal={Probab. Theory Relat. Fields},
      volume={135},
      number={4},
       pages={520\ndash 546},
      eprint={math/0412253},
         url={https://doi.org/10.1007/s00440-005-0456-1},
}

\bib{AwBuSc21}{article}{
      author={Aw, Clive~Cenxin},
      author={Buscemi, Francesco},
      author={Scarani, Valerio},
       title={Fluctuation theorems with retrodiction rather than reverse
  processes},
        date={2021},
     journal={{AVS} Quantum Science},
      volume={3},
      number={4},
       pages={045601},
      eprint={2106.08589},
         url={https://doi.org/10.1116/5.0060893},
}

\bib{BaJaMa2015}{article}{
      author={Backs, Karl},
      author={Jackson, Steve},
      author={Mauldin, R.~Daniel},
       title={{CH}, $\mathbf{V}=\mathbf{L}$, disintegration of measures, and
  $\boldsymbol{II}^{1}_{1}$ sets},
        date={2015},
        ISSN={0001-8708},
     journal={Adv. Math.},
      volume={274},
       pages={76\ndash 96},
         url={https://doi.org/10.1016/j.aim.2014.12.011},
}

\bib{BaMu94}{book}{
      author={Baez, John},
      author={Muniain, Javier~P.},
       title={Gauge fields, knots and gravity},
      series={Series on Knots and Everything},
   publisher={World Scientific Publishing Co. Pte. Ltd.},
     address={Singapore},
        date={1994},
      volume={4},
         url={https://doi.org/10.1142/2324},
}

\bib{BaiBuscemiScarani2025}{article}{
      author={Bai, Ge},
      author={Buscemi, Francesco},
      author={Scarani, Valerio},
       title={Quantum {B}ayes' rule and {P}etz transpose map from the minimum
  change principle},
        date={2025Aug},
     journal={Phys. Rev. Lett.},
      volume={135},
       pages={090203},
      eprint={2410.00319},
         url={https://link.aps.org/doi/10.1103/5n4p-bxhm},
}

\bib{BaRuSp2007}{article}{
      author={Bartlett, Stephen~D.},
      author={Rudolph, Terry},
      author={Spekkens, Robert~W.},
       title={Reference frames, superselection rules, and quantum information},
        date={2007Apr},
     journal={Rev. Mod. Phys.},
      volume={79},
       pages={555\ndash 609},
      eprint={quant-ph/0610030},
         url={https://link.aps.org/doi/10.1103/RevModPhys.79.555},
}

\bib{Bayes1763}{article}{
      author={Bayes, Thomas},
       title={{LII. An essay towards solving a problem in the doctrine of
  chances. By the late Rev. Mr. Bayes, FRS communicated by Mr. Price, in a
  letter to John Canton, A.M.F.R.S}},
        date={1763},
     journal={Philos. Trans. R. Soc.},
      number={53},
       pages={370\ndash 418},
         url={https://doi.org/10.1098/rstl.1763.0053},
}

\bib{BeSt86}{article}{
      author={Belavkin, Viacheslav~P.},
      author={Staszewski, Przemys{\l}aw},
       title={A {Radon}--{Nikodym} theorem for completely positive maps},
        date={1986},
        ISSN={0034-4877},
     journal={Rep. Math. Phys.},
      volume={24},
      number={1},
       pages={49 \ndash  55},
         url={https://doi.org/10.1016/0034-4877(86)90039-X},
}

\bib{Billingsley1995}{book}{
      author={Billingsley, Patrick},
       title={Probability and measure.},
     edition={3rd ed.},
   publisher={John Wiley \& Sons Ltd.},
     address={Chichester, UK},
        date={1995},
        ISBN={0-471-00710-2},
}

\bib{Bis17}{article}{
      author={Bischoff, Marcel},
       title={Generalized orbifold construction for conformal nets},
        date={2017},
        ISSN={0129-055X,1793-6659},
     journal={Rev. Math. Phys.},
      volume={29},
      number={1},
       pages={1750002, 53},
         url={https://doi.org/10.1142/S0129055X17500027},
}

\bib{BiDeGi2021}{article}{
      author={Bischoff, Marcel},
      author={Del~{V}ecchio, Simone},
      author={Giorgetti, Luca},
       title={Compact hypergroups from discrete subfactors},
        date={2021Jul},
     journal={J. Funct. Anal.},
      volume={281},
       pages={109004},
      eprint={2007.12384},
         url={https://doi.org/10.1016/j.jfa.2021.109004},
}

\bib{BrRo1}{book}{
      author={Bratteli, Ola},
      author={Robinson, Derek~W.},
       title={Operator algebras and quantum statistical mechanics 1},
    subtitle={{$C^\ast$}- and {$W^\ast$}-algebras, symmetry groups,
  decomposition of states},
     edition={2},
      series={Texts and Monographs in Physics},
   publisher={Springer-Verlag, New York},
        date={1987},
        ISBN={0-387-17093-6},
         url={http://dx.doi.org/10.1007/978-3-662-02520-8},
}

\bib{BrRo2}{book}{
      author={Bratteli, Ola},
      author={Robinson, Derek~W.},
       title={Operator algebras and quantum statistical mechanics 2},
    subtitle={Equilibrium states, models in quantum statistical mechanics},
     edition={Second},
      series={Texts and Monographs in Physics},
   publisher={Springer-Verlag, Berlin},
        date={1997},
        ISBN={3-540-61443-5},
         url={http://dx.doi.org/10.1007/978-3-662-03444-6},
}

\bib{BrownOzawa2008}{book}{
      author={Brown, Nathanial~P.},
      author={Ozawa, Narutaka},
       title={{\(C^*\)}-algebras and finite-dimensional approximations},
      series={Grad. Stud. Math.},
   publisher={American Mathematical Society},
     address={Providence, RI},
        date={2008},
      volume={88},
        ISBN={978-0-8218-4381-9},
         url={https://doi.org/10.1090/gsm/088},
}

\bib{CaMa17}{article}{
      author={Carlen, Eric~A.},
      author={Maas, Jan},
       title={Gradient flow and entropy inequalities for quantum {M}arkov
  semigroups with detailed balance},
        date={2017},
        ISSN={0022-1236},
     journal={J. Funct. Anal.},
      volume={273},
      number={5},
       pages={1810\ndash 1869},
      eprint={1609.01254},
         url={https://doi.org/10.1016/j.jfa.2017.05.003},
}

\bib{CaVe20}{article}{
      author={Carlen, Eric~A.},
      author={Vershynina, Anna},
       title={Recovery map stability for the data processing inequality},
        date={2020},
     journal={J. Phys. A},
      volume={53},
      number={3},
       pages={035204},
      eprint={1710.02409},
         url={https://doi.org/10.1088/1751-8121/ab5ab7},
}

\bib{Caticha2021}{article}{
      author={Caticha, Ariel},
       title={Entropy, information, and the updating of probabilities},
        date={2021},
     journal={Entropy},
      volume={23},
      number={7},
       pages={895},
      eprint={2107.04529},
         url={https://doi.org/10.3390/e23070895},
}

\bib{ChanDarwiche2005}{article}{
      author={Chan, Hei},
      author={Darwiche, Adnan},
       title={On the revision of probabilistic beliefs using uncertain
  evidence},
        date={2005},
     journal={Artif. Intell.},
      volume={163},
       pages={67\ndash 90},
         url={https://doi.org/10.1016/j.artint.2004.09.005},
}

\bib{ChJa18}{article}{
      author={Cho, Kenta},
      author={Jacobs, Bart},
       title={Disintegration and {B}ayesian inversion via string diagrams},
        date={2019},
     journal={Math. Struct. Comp. Sci.},
       pages={1\ndash 34},
      eprint={1709.00322},
         url={https://doi.org/10.1017/S0960129518000488},
}

\bib{Choi1974}{article}{
      author={Choi, Man-Duen},
       title={A {S}chwarz inequality for positive linear maps on
  {$C^*$}-algebras},
        date={1974},
     journal={Illinois J. Math.},
      volume={18},
      number={4},
       pages={565\ndash 574},
         url={https://doi.org/10.1215/ijm/1256051007},
}

\bib{CDDG17}{inproceedings}{
      author={Clerc, Florence},
      author={Danos, Vincent},
      author={Dahlqvist, Fredrik},
      author={Garnier, Ilias},
       title={Pointless learning},
        date={2017},
   booktitle={International conference on foundations of software science and
  computation structures},
   publisher={Springer},
       pages={355\ndash 369},
         url={https://doi.org/10.1007/978-3-662-54458-7_21},
}

\bib{Connes74}{article}{
      author={Connes, Alain},
       title={Caract\'erisation des espaces vectoriels ordonn\'es sous-jacents
  aux alg\`ebres de von~{Neumann}},
        date={1974},
     journal={Ann. Inst. Fourier},
      volume={24},
      number={4},
       pages={121\ndash 155},
         url={http://www.numdam.org/articles/10.5802/aif.534/},
}

\bib{Cs91}{article}{
      author={Csisz{\'a}r, Imre},
       title={Why least squares and maximum entropy? an axiomatic approach to
  inference for linear inverse problems},
        date={1991},
     journal={Ann. Statist.},
      volume={19},
      number={4},
       pages={2032\ndash 2066},
}

\bib{CuSt14}{article}{
      author={Culbertson, Jared},
      author={Sturtz, Kirk},
       title={A categorical foundation for {B}ayesian probability},
        date={2014},
     journal={Appl. Categ. Struct.},
      volume={22},
      number={4},
       pages={647\ndash 662},
      eprint={1205.1488},
         url={https://doi.org/10.1007/s10485-013-9324-9},
}

\bib{Davies74}{article}{
      author={Davies, E.~Brian},
       title={{Markovian Master Equations}},
        date={1974},
     journal={Commun. Math. Phys.},
      volume={39},
      number={2},
       pages={91\ndash 110},
         url={https://doi.org/10.1007/BF01608389},
}

\bib{Dix81}{book}{
      author={Dixmier, Jacques},
       title={{von~{N}eumann algebras}},
      series={{North-Holland Mathematical Library}},
   publisher={North-Holland Publishing Co.},
     address={Amsterdam},
        date={1981},
      volume={27},
        ISBN={0-444-86308-7},
}

\bib{FaUm07}{article}{
      author={Fagnola, Franco},
      author={Umanita, Veronica},
       title={Generators of detailed balance quantum {M}arkov semigroups},
        date={2007-09},
        ISSN={1793-6306},
     journal={Infin. Dimens. Anal. Quantum Probab. Relat. Top.},
      volume={10},
      number={03},
       pages={335\ndash 363},
      eprint={0707.2147},
         url={http://dx.doi.org/10.1142/S0219025707002762},
}

\bib{FewsterVerch2020quantum}{article}{
      author={Fewster, Christopher~J.},
      author={Verch, Rainer},
       title={Quantum fields and local measurements},
        date={2020},
     journal={Commun. Math. Phys.},
      volume={378},
      number={2},
       pages={851\ndash 889},
      eprint={1810.06512},
         url={https://doi.org/10.1007/s00220-020-03800-6},
}

\bib{Fillmore96}{book}{
      author={Fillmore, Peter~A.},
       title={A user's guide to operator algebras},
      series={Canadian Mathematical Society Series of Monographs and Advanced
  Texts},
   publisher={John Wiley \& Sons, Inc., New York},
        date={1996},
        ISBN={0-471-31135-9},
        note={A Wiley-Interscience Publication},
}

\bib{FJV15}{article}{
      author={Fitzsimons, Joseph~F.},
      author={Jones, Jonathan~A.},
      author={Vedral, Vlatko},
       title={Quantum correlations which imply causation},
        date={2015},
     journal={Sci. Rep.},
      volume={5},
      number={1},
       pages={18281},
      eprint={1302.2731},
         url={https://doi.org/10.1038/srep18281},
}

\bib{Fo07}{book}{
      author={Folland, Gerald~B.},
       title={Real analysis: Modern techniques and their applications},
     edition={2},
   publisher={Wiley},
        date={2007},
}

\bib{FrV4}{book}{
      author={Fremlin, D.~H.},
       title={Measure theory. {V}ol. 4},
    subtitle={Topological measure spaces. part i, ii, corrected second printing
  of the 2003 original},
   publisher={Torres Fremlin, Colchester},
        date={2006},
        ISBN={0-9538129-4-4},
        note={Updated version (as of 23.3.10) available at
  \url{https://www1.essex.ac.uk/maths/people/fremlin/cont45.htm}},
}

\bib{Fr20}{article}{
      author={Fritz, Tobias},
       title={A synthetic approach to {M}arkov kernels, conditional
  independence and theorems on sufficient statistics},
        date={2020},
     journal={Adv. Math.},
      volume={370},
       pages={107239},
      eprint={1908.07021},
         url={https://doi.org/10.1016/j.aim.2020.107239},
}

\bib{FritzLorenzin2026}{article}{
      author={Fritz, Tobias},
      author={Lorenzin, Antonio},
       title={Categories of abstract and noncommutative measurable spaces},
        date={2026-03},
        ISSN={0001-8708},
     journal={Adv. Math.},
      volume={488},
       pages={110793},
      eprint={2504.13708},
         url={http://dx.doi.org/10.1016/j.aim.2026.110793},
}

\bib{FuPa24a}{misc}{
      author={{Fullwood}, James},
      author={{Parzygnat}, Arthur~J.},
       title={Operator representation of spatiotemporal quantum correlations},
        date={2024},
        note={Available at
  \href{https://arxiv.org/abs/2405.17555}{arXiv:2405.17555}},
}

\bib{FuJa13}{article}{
      author={Furber, Robert},
      author={Jacobs, Bart},
       title={From {K}leisli categories to commutative {$C^*$}-algebras:
  probabilistic {G}elfand duality},
        date={2015},
     journal={Log. Methods Comput. Sci.},
      volume={11},
       pages={1\ndash 28},
      eprint={1303.1115},
         url={https://doi.org/10.1007/978-3-642-40206-7_12},
}

\bib{GaPa18}{incollection}{
      author={Gagn\'e, Nicolas},
      author={Panangaden, Prakash},
       title={A categorical characterization of relative entropy on standard
  {B}orel spaces},
        date={2018},
   booktitle={The {T}hirty-third {C}onference on the {M}athematical
  {F}oundations of {P}rogramming {S}emantics ({MFPS} {XXXIII})},
      series={Electron. Notes Theor. Comput. Sci.},
      volume={336},
   publisher={Elsevier},
     address={Amsterdam},
       pages={135\ndash 153},
         url={https://doi.org/10.1016/j.entcs.2018.03.020},
}

\bib{GaoWilde2021}{article}{
      author={Gao, Li},
      author={Wilde, Mark~M.},
       title={Recoverability for optimized quantum $f$-divergences},
        date={2021},
     journal={J. Phys. A: Math. Theor.},
      volume={54},
      number={38},
       pages={385302},
      eprint={2008.01668},
         url={http://dx.doi.org/10.1088/1751-8121/ac1dc2},
}

\bib{GN43}{article}{
      author={Gelfand, Israel},
      author={Neumark, Mark},
       title={On the imbedding of normed rings into the ring of operators in
  {H}ilbert space},
        date={1943},
     journal={Rec. Math. [Mat. Sbornik] N.S.},
      volume={12(54)},
       pages={197\ndash 213},
}

\bib{Gh10}{article}{
      author={Gheondea, Aurelian},
       title={The three equivalent forms of completely positive maps on
  matrices},
        date={2010},
     journal={Ann. Univ. Bucharest (Math. Ser.)},
      volume={1},
      number={LIX},
       pages={79\ndash 98},
        note={Accessed from
  \href{https://repository.bilkent.edu.tr/bitstreams/46f7ed8f-1278-46aa-a652-f2911e495a2a/download}{https://repository.bilkent.edu.tr/bitstreams/46f7ed8f-1278-46aa-a652-f2911e495a2a/download}
  on 2026-08-19},
}

\bib{GPRR21}{article}{
      author={{Giorgetti}, Luca},
      author={{Parzygnat}, Arthur~J.},
      author={{Ranallo}, Alessio},
      author={{Russo}, Benjamin~P.},
       title={Bayesian inversion and the {T}omita--{T}akesaki modular group},
        date={202303},
     journal={Q. J. Math.},
      volume={74},
      number={3},
       pages={975\ndash 1014},
      eprint={2112.03129},
         url={https://doi.org/10.1093/qmath/haad014},
}

\bib{Ha1975}{article}{
      author={Haagerup, Uffe},
       title={The standard form of von {N}eumann algebras},
        date={1975},
        ISSN={0025-5521,1903-1807},
     journal={Math. Scand.},
      volume={37},
      number={2},
       pages={271\ndash 283},
         url={https://doi-org.shsu.idm.oclc.org/10.7146/math.scand.a-11606},
}

\bib{HaMu2011}{article}{
      author={Haagerup, Uffe},
      author={Musat, Magdalena},
       title={Factorization and dilation problems for completely positive maps
  on von~{N}eumann algebras},
        date={2011},
        ISSN={0010-3616,1432-0916},
     journal={Comm. Math. Phys.},
      volume={303},
      number={2},
       pages={555\ndash 594},
      eprint={1009.0778},
         url={https://doi.org//10.1007/s00220-011-1216-y},
}

\bib{Ha13}{book}{
      author={Hall, Brian~C.},
       title={Lie groups, {L}ie algebras, and representations},
    subtitle={An introduction},
   publisher={Springer Cham},
        date={2013},
         url={https://doi.org/10.1007/978-3-319-13467-3},
}

\bib{HallQuantum13}{book}{
      author={Hall, Brian~C.},
       title={Quantum theory for mathematicians},
      series={Graduate Texts in Mathematics},
   publisher={Springer, New York},
        date={2013},
      volume={267},
        ISBN={978-1-4614-7115-8; 978-1-4614-7116-5},
         url={https://doi.org/10.1007/978-1-4614-7116-5},
}

\bib{Ha76}{article}{
      author={Hawking, S.~W.},
       title={Breakdown of predictability in gravitational collapse},
        date={1976},
     journal={Phys. Rev. D},
      volume={14},
       pages={2460\ndash 2473},
         url={https://link.aps.org/doi/10.1103/PhysRevD.14.2460},
}

\bib{Holevo2019}{book}{
      author={Holevo, Alexander~S.},
       title={Quantum systems, channels, information},
    subtitle={A mathematical introduction},
     edition={2nd revised and extended edition},
      series={Texts Monogr. Theor. Phys.},
   publisher={De Gruyter},
     address={Berlin, Boston},
        date={2019},
        ISBN={978-3-11-064224-7; 978-3-11-064249-0},
         url={https://doi.org/10.1515/9783110642490},
}

\bib{Ja19}{article}{
      author={Jacobs, Bart},
       title={The mathematics of changing one's mind, via {J}effrey's or via
  {P}earl's update rule},
        date={2019},
     journal={J. Artificial Intelligence Res.},
      volume={65},
       pages={783\ndash 806},
      eprint={1807.05609},
}

\bib{Jacobs2020}{article}{
      author={Jacobs, Bart},
       title={A channel-based perspective on conjugate priors},
        date={2020},
     journal={Math. Struct. Comput. Sci.},
      volume={30},
      number={1},
       pages={44\ndash 61},
      eprint={1707.00269},
}

\bib{Je90}{book}{
      author={Jeffrey, Richard~C.},
       title={The logic of decision},
     edition={2},
   publisher={University of Chicago Press},
        date={1990},
}

\bib{JSK23}{article}{
      author={Jia, Zhian},
      author={Song, Minjeong},
      author={Kaszlikowski, Dagomir},
       title={Quantum space-time marginal problem: global causal structure from
  local causal information},
        date={2023-12},
        ISSN={1367-2630},
     journal={New J. Phys.},
      volume={25},
      number={12},
       pages={123038},
      eprint={2303.12819},
         url={http://dx.doi.org/10.1088/1367-2630/ad1416},
}

\bib{JRSWW16}{inproceedings}{
      author={Junge, Marius},
      author={Renner, Renato},
      author={Sutter, David},
      author={Wilde, Mark~M.},
      author={Winter, Andreas},
       title={Universal recoverability in quantum information},
        date={2016},
   booktitle={{2016 IEEE International Symposium on Information Theory
  (ISIT)}},
       pages={2494\ndash 2498},
}

\bib{JRSWW18}{article}{
      author={Junge, Marius},
      author={Renner, Renato},
      author={Sutter, David},
      author={Wilde, Mark~M.},
      author={Winter, Andreas},
       title={Universal recovery maps and approximate sufficiency of quantum
  relative entropy},
        date={2018},
        ISSN={1424-0661},
     journal={Ann. Henri Poincar\'e},
      volume={19},
      number={10},
       pages={2955\ndash 2978},
      eprint={1509.07127},
         url={http://dx.doi.org/10.1007/s00023-018-0716-0},
}

\bib{Ka21}{book}{
      author={Kallenberg, Olav},
       title={Foundations of modern probability},
     edition={3},
   publisher={Springer},
     address={Cham, CH},
        date={2021},
         url={https://doi.org/10.1007/978-3-030-61871-1},
}

\bib{Ka18}{thesis}{
      author={{Karvonen}, Martti},
       title={{The Way of the Dagger}},
        type={Ph.D. Thesis},
        date={2018},
        note={The University of Edinburgh. Available at
  \href{https://arxiv.org/abs/1904.10805}{arXiv:1904.10805 [math.CT]}},
}

\bib{KnillLaflamme1997}{article}{
      author={Knill, Emanuel},
      author={Laflamme, Raymond},
       title={Theory of quantum error-correcting codes},
        date={1997Feb},
     journal={Phys. Rev. A},
      volume={55},
       pages={900\ndash 911},
      eprint={quant-ph/9604034},
         url={https://link.aps.org/doi/10.1103/PhysRevA.55.900},
}

\bib{Kubo1957}{article}{
      author={Kubo, Ryogo},
       title={Statistical-mechanical theory of irreversible processes. {I}.
  {G}eneral theory and simple applications to magnetic and conduction
  problems},
        date={1957},
     journal={J. Phys. Soc. Jpn.},
      volume={12},
      number={6},
       pages={570\ndash 586},
         url={https://doi.org/10.1143/JPSJ.12.570},
}

\bib{Lee2013Manifolds}{book}{
      author={Lee, John~M.},
       title={Introduction to smooth manifolds},
     edition={2nd revised ed},
      series={Grad. Texts Math.},
   publisher={Springer},
        date={2013},
      volume={218},
        ISBN={978-1-4419-9981-8; 978-1-4419-9982-5},
         url={https://doi.org/10.1007/978-1-4419-9982-5},
}

\bib{LeRu99}{article}{
      author={Lesniewski, Andrew},
      author={Ruskai, Mary~Beth},
       title={Monotone {R}iemannian metrics and relative entropy on
  noncommutative probability spaces},
        date={1999-11},
        ISSN={1089-7658},
     journal={J. Math. Phys.},
      volume={40},
      number={11},
       pages={5702\ndash 5724},
      eprint={math-ph/9808016},
         url={http://dx.doi.org/10.1063/1.533053},
}

\bib{LiWi18}{article}{
      author={Li, Ke},
      author={Winter, Andreas},
       title={{Squashed Entanglement, k-Extendibility, Quantum Markov Chains,
  and Recovery Maps}},
        date={2018},
     journal={Found. Phys.},
      volume={48},
      number={8},
       pages={910\ndash 924},
      eprint={1410.4184},
}

\bib{liu2023quantum}{article}{
      author={Liu, Xiangjing},
      author={Qiu, Yixian},
      author={Dahlsten, Oscar},
      author={Vedral, Vlatko},
       title={Quantum causal inference with extremely light touch},
        date={2025-03},
        ISSN={2056-6387},
     journal={npj Quantum Inf.},
      volume={11},
       pages={54},
      eprint={2303.10544},
         url={http://dx.doi.org/10.1038/s41534-024-00956-0},
}

\bib{LKJR15}{article}{
      author={Lostaglio, Matteo},
      author={Korzekwa, Kamil},
      author={Jennings, David},
      author={Rudolph, Terry},
       title={Quantum coherence, time-translation symmetry, and
  thermodynamics},
        date={2015Apr},
     journal={Phys. Rev. X},
      volume={5},
       pages={021001},
      eprint={1410.4572},
         url={https://link.aps.org/doi/10.1103/PhysRevX.5.021001},
}

\bib{Ma98}{book}{
      author={Mac~Lane, Saunders},
       title={Categories for the working mathematician},
     edition={Second ed.},
      series={Graduate Texts in Mathematics},
   publisher={Springer-Verlag, New York},
        date={1998},
      volume={5},
        ISBN={0-387-98403-8},
}

\bib{MartinSchwingerI1959}{article}{
      author={Martin, Paul~C.},
      author={Schwinger, Julian},
       title={Theory of many-particle systems. {I}},
        date={1959Sep},
     journal={Phys. Rev.},
      volume={115},
       pages={1342\ndash 1373},
         url={https://link.aps.org/doi/10.1103/PhysRev.115.1342},
}

\bib{Ma12}{thesis}{
      author={Marvian, Iman},
       title={Symmetry, asymmetry and quantum information},
        type={Ph.D. Thesis},
   publisher={University of Waterloo},
     address={UWSpace},
        date={2012},
         url={http://hdl.handle.net/10012/7088},
        note={Accessed from
  \href{http://hdl.handle.net/10012/7088}{http://hdl.handle.net/10012/7088} on
  2026-08-19},
}

\bib{Marvian2020coherence}{article}{
      author={Marvian, Iman},
       title={Coherence distillation machines are impossible in quantum
  thermodynamics},
        date={2020Jan},
     journal={Nat. Commun.},
      volume={11},
       pages={25},
      eprint={1805.01989},
         url={https://doi.org/10.1038/s41467-019-13846-3},
}

\bib{MarvianSpekkens2014}{article}{
      author={Marvian, Iman},
      author={Spekkens, Robert~W.},
       title={Modes of asymmetry: The application of harmonic analysis to
  symmetric quantum dynamics and quantum reference frames},
        date={2014Dec},
     journal={Phys. Rev. A},
      volume={90},
       pages={062110},
      eprint={1312.0680},
         url={https://link.aps.org/doi/10.1103/PhysRevA.90.062110},
}

\bib{MilnorTopDiff1997}{book}{
      author={Milnor, John~W.},
       title={Topology from the differentiable viewpoint},
      series={Princeton Landmarks in Mathematics and Physics},
   publisher={Princeton University Press},
     address={Princeton, NJ},
        date={1997},
        ISBN={0-691-04833-9},
        note={Based on notes by David W.\ Weaver; Revised reprint of the 1965
  original},
}

\bib{Mo54}{article}{
      author={Moy, Shu-Teh~Chen},
       title={Characterizations of conditional expectation as a transformation
  on function spaces},
        date={1954},
     journal={Pac. J. Math.},
      volume={4},
      number={1},
       pages={47\ndash 63},
}

\bib{Munkres2000}{book}{
      author={Munkres, James~R.},
       title={Topology: a first course},
     edition={2},
   publisher={Prentice-Hall, Inc.},
     address={Saddle River, NJ},
        date={2000},
         url={https://api.semanticscholar.org/CorpusID:118005688},
}

\bib{Nair2015thermofield}{article}{
      author={Nair, V.~Parameswaran},
       title={Thermofield dynamics and gravity},
        date={2015Nov},
     journal={Phys. Rev. D},
      volume={92},
       pages={104009},
      eprint={1508.00171},
         url={https://link.aps.org/doi/10.1103/PhysRevD.92.104009},
}

\bib{Nair2020}{article}{
      author={Nair, V.~Parameswaran},
       title={Entanglement for quantum {H}all states and a generalized
  {C}hern-{S}imons form},
        date={2020Jun},
     journal={Phys. Rev. D},
      volume={101},
       pages={125021},
      eprint={2001.04957},
         url={https://link.aps.org/doi/10.1103/PhysRevD.101.125021},
}

\bib{NiCh11}{book}{
      author={Nielsen, Michael~A.},
      author={Chuang, Isaac~L.},
       title={Quantum computation and quantum information},
     edition={10th Anniversary},
   publisher={Cambridge University Press, Cambridge},
        date={2011},
        ISBN={0-521-63235-8; 0-521-63503-9},
         url={https://doi.org/10.1017/CBO9780511976667},
}

\bib{OhPe93}{book}{
      author={Ohya, Masanori},
      author={Petz, D\'enes},
       title={Quantum entropy and its use},
      series={Texts and Monographs in Physics},
   publisher={Springer-Verlag},
     address={Berlin, DE},
        date={1993},
        ISBN={3-540-54881-5},
         url={https://doi.org/10.1007/978-3-642-57997-4},
}

\bib{Page1993IBHR}{article}{
      author={Page, Don~N.},
       title={Information in black hole radiation},
        date={1993Dec},
     journal={Phys. Rev. Lett.},
      volume={71},
       pages={3743\ndash 3746},
      eprint={hep-th/9306083},
         url={https://link.aps.org/doi/10.1103/PhysRevLett.71.3743},
}

\bib{Pa17}{misc}{
      author={Parzygnat, Arthur~J.},
       title={Discrete probabilistic and algebraic dynamics: a stochastic
  commutative {G}elfand-{N}aimark theorem},
        date={2017},
        note={Available at
  \href{https://arxiv.org/abs/1708.00091}{arxiv:1708.00091}},
}

\bib{PaGNS}{article}{
      author={Parzygnat, Arthur~J.},
       title={From observables and states to {H}ilbert space and back: A
  2-categorical adjunction},
        date={2018},
     journal={Appl. Categ. Structures},
      volume={26},
       pages={1123\ndash 1157},
      eprint={1609.08975},
}

\bib{PaBayes}{misc}{
      author={Parzygnat, Arthur~J.},
       title={Inverses, disintegrations, and {B}ayesian inversion in quantum
  {M}arkov categories},
        date={2020},
        note={Available at
  \href{https://arxiv.org/abs/2001.08375}{arXiv:2001.08375}},
}

\bib{Pa24}{misc}{
      author={Parzygnat, Arthur~J.},
       title={Reversing information flow: retrodiction in semicartesian
  categories},
        date={2024},
        note={Available at
  \href{https://arxiv.org/abs/2401.17447}{arXiv:2401.17447}},
}

\bib{PaBu22}{article}{
      author={{Parzygnat}, Arthur~J.},
      author={{Buscemi}, Francesco},
       title={Axioms for retrodiction: achieving time-reversal symmetry with a
  prior},
        date={2023},
        ISSN={2521-327X},
     journal={{Quantum}},
      volume={7},
       pages={1013},
      eprint={2210.13531},
         url={https://doi.org/10.22331/q-2023-05-23-1013},
}

\bib{FuPa22a}{article}{
      author={Parzygnat, Arthur~J.},
      author={Fullwood, James},
       title={{From time-reversal symmetry to quantum Bayes' rules}},
        date={2023Jun},
     journal={PRX Quantum},
      volume={4},
       pages={020334},
      eprint={2212.08088},
         url={https://link.aps.org/doi/10.1103/PRXQuantum.4.020334},
}

\bib{PaFu24TSC}{article}{
      author={Parzygnat, Arthur~J.},
      author={Fullwood, James},
       title={Time-symmetric correlations for open quantum systems},
        date={2025},
     journal={Ann. Phys.},
      volume={537},
       pages={e00221},
      eprint={2407.11123},
         url={https://doi.org/10.1002/andp.202500221},
}

\bib{PaRuBayes}{article}{
      author={Parzygnat, Arthur~J.},
      author={Russo, Benjamin~P.},
       title={A non-commutative {B}ayes' theorem},
        date={2022},
        ISSN={0024-3795},
     journal={Linear Algebra Its Appl.},
      volume={644},
       pages={28\ndash 94},
      eprint={2005.03886},
  url={https://www.sciencedirect.com/science/article/pii/S0024379522000805},
}

\bib{PaRu19}{article}{
      author={Parzygnat, Arthur~J.},
      author={Russo, Benjamin~P.},
       title={Non-commutative disintegrations: existence and uniqueness in
  finite dimensions},
        date={2023},
     journal={J. Noncommut. Geom.},
      volume={17},
       pages={899\ndash 955},
      eprint={1907.09689},
         url={https://doi.org/10.4171/JNCG/493},
}

\bib{Paulsen02}{book}{
      author={Paulsen, Vern},
       title={Completely bounded maps and operator algebras},
      series={Cambridge Studies in Advanced Mathematics},
   publisher={Cambridge University Press, Cambridge},
        date={2002},
      volume={78},
        ISBN={0-521-81669-6},
         url={https://doi.org/10.1017/CBO9780511546631},
}

\bib{Pavlov2022}{article}{
      author={Pavlov, Dmitri},
       title={Gelfand-type duality for commutative von~{N}eumann algebras},
        date={2022-04},
        ISSN={0022-4049},
     journal={J. Pure Appl. Algebra},
      volume={226},
      number={4},
       pages={106884},
      eprint={2005.05284},
         url={http://dx.doi.org/10.1016/j.jpaa.2021.106884},
}

\bib{Pearl88}{book}{
      author={Pearl, Judea},
       title={Probabilistic reasoning in intelligent systems: Networks of
  plausible inference},
   publisher={Elsevier},
        date={1988},
         url={https://doi.org/10.1016/C2009-0-27609-4},
}

\bib{PeresTerno2004}{article}{
      author={Peres, Asher},
      author={Terno, Daniel~R.},
       title={Quantum information and relativity theory},
        date={2004Jan},
     journal={Rev. Mod. Phys.},
      volume={76},
       pages={93\ndash 123},
      eprint={quant-ph/0212023},
         url={https://link.aps.org/doi/10.1103/RevModPhys.76.93},
}

\bib{Pe84}{article}{
      author={Petz, D{\'e}nes},
       title={A dual in von~{N}eumann algebras with weights},
        date={1984},
     journal={Q. J. Math.},
      volume={35},
      number={4},
       pages={475\ndash 483},
         url={https://doi.org/10.1093/qmath/35.4.475},
}

\bib{Petz96}{article}{
      author={Petz, D{\'e}nes},
       title={Monotone metrics on matrix spaces},
        date={1996},
        ISSN={0024-3795},
     journal={Linear Algebra Its Appl.},
      volume={244},
       pages={81\ndash 96},
         url={https://doi.org/10.1016/0024-3795(94)00211-8},
}

\bib{PetzGhinea2011}{inproceedings}{
      author={Petz, D{\'e}nes},
      author={Ghinea, Catalin},
       title={Introduction to quantum {F}isher information},
        date={2011},
   booktitle={Quantum probability and related topics},
   publisher={World Scientific},
       pages={261\ndash 281},
         url={https://doi.org/10.1142/9789814338745_0015},
}

\bib{ReedSimonI1980}{book}{
      author={Reed, Michael},
      author={Simon, Barry},
       title={Methods of modern mathematical physics {I}: Functional analysis},
     edition={Revised and Enlarged Edition},
   publisher={Academic Press},
        date={1980},
         url={https://doi.org/10.1016/B978-0-12-585001-8.X5001-6},
}

\bib{RFZ10}{article}{
      author={{Roga}, Wojciech},
      author={{Fannes}, Mark},
      author={{{\.Z}yczkowski}, Karol},
       title={{Davies maps for qubits and qutrits}},
        date={2010},
     journal={Rep. Math. Phys.},
      volume={66},
      number={3},
       pages={311\ndash 329},
      eprint={0911.5607},
}

\bib{Rokhlin1949}{article}{
      author={Rokhlin, Vladimir~Abramovich},
       title={On the fundamental ideas of measure theory},
        date={1949},
     journal={Mat. Sb. (N.S.)},
      volume={25(67)},
      number={1},
       pages={107\ndash 150},
}

\bib{Rudin76}{book}{
      author={Rudin, Walter},
       title={Principles of mathematical analysis},
     edition={3},
   publisher={McGraw-Hill},
     address={New York, NY},
        date={1976},
}

\bib{Rudin87}{book}{
      author={Rudin, Walter},
       title={Real and complex analysis},
     edition={Third ed.},
   publisher={McGraw-Hill},
     address={New York, NY},
        date={1987},
        ISBN={0-07-054234-1},
}

\bib{Rudin1991}{book}{
      author={Rudin, Walter},
       title={Functional analysis},
     edition={2nd ed.},
   publisher={McGraw-Hill},
     address={New York, NY},
        date={1991},
        ISBN={0-07-054236-8},
}

\bib{RyuTakayanagi2006}{article}{
      author={Ryu, Shinsei},
      author={Takayanagi, Tadashi},
       title={Holographic derivation of entanglement entropy from the anti--de
  {S}itter space/conformal field theory correspondence},
        date={2006May},
     journal={Phys. Rev. Lett.},
      volume={96},
       pages={181602},
      eprint={hep-th/0603001},
         url={https://link.aps.org/doi/10.1103/PhysRevLett.96.181602},
}

\bib{Sakurai20}{book}{
      author={Sakurai, J.~J.},
      author={Napolitano, Jim},
       title={Modern quantum mechanics},
     edition={3},
   publisher={Cambridge University Press},
        date={2020},
         url={https://doi.org/10.1017/9781108587280},
}

\bib{SASDS23}{article}{
      author={Scandi, Matteo},
      author={Abiuso, Paolo},
      author={Surace, Jacopo},
      author={De~Santis, Dario},
       title={Quantum {F}isher information and its dynamical nature},
        date={2025jul},
     journal={Rep. Prog. Phys.},
      volume={88},
      number={7},
       pages={076001},
      eprint={2304.14984},
         url={https://dx.doi.org/10.1088/1361-6633/ade453},
}

\bib{Schechter2002}{book}{
      author={Schechter, Martin},
       title={Operator methods in quantum mechanics},
     edition={Reprint of the 1981 original},
   publisher={Dover Publications, Mineola, NY},
        date={2002},
        ISBN={0-486-42547-9},
}

\bib{Schervish1995}{book}{
      author={Schervish, Mark~J.},
       title={Theory of statistics},
      series={Springer Series in Statistics},
   publisher={Springer},
     address={New York, NY},
        date={1995},
        ISBN={978-0-387-94546-0},
         url={https://doi.org/10.1007/978-1-4612-4250-5},
}

\bib{Schmudgen2012}{book}{
      author={Schm{\"u}dgen, Konrad},
       title={Unbounded self-adjoint operators on {H}ilbert space},
      series={Graduate Texts in Mathematics},
   publisher={Springer Dordrecht},
        date={2012},
         url={https://doi.org/10.1007/978-94-007-4753-1},
}

\bib{Schumacher1995}{article}{
      author={Schumacher, Benjamin},
       title={Quantum coding},
        date={1995Apr},
     journal={Phys. Rev. A},
      volume={51},
       pages={2738\ndash 2747},
         url={https://link.aps.org/doi/10.1103/PhysRevA.51.2738},
}

\bib{Se47}{article}{
      author={Segal, Irving~E.},
       title={Irreducible representations of operator algebras},
        date={1947},
        ISSN={0002-9904},
     journal={Bull. Amer. Math. Soc.},
      volume={53},
       pages={73\ndash 88},
         url={https://doi.org/10.1090/S0002-9904-1947-08742-5},
}

\bib{Se07}{inproceedings}{
      author={Selinger, Peter},
       title={Dagger compact closed categories and completely positive maps:
  (extended abstract)},
        date={2007},
   booktitle={Proceedings of the 3rd international workshop on quantum
  programming languages ({QPL} 2005)},
      volume={170},
       pages={139\ndash 163},
  url={http://www.sciencedirect.com/science/article/pii/S1571066107000606},
}

\bib{Sengupta2006}{misc}{
      author={Sengupta, Ambar},
       title={Lecture 6: The {D}ynkin $\pi$-$\lambda$ theorem},
        date={2006},
         url={https://www.math.lsu.edu/~sengupta/7360f09/DynkinPiLambda.pdf},
        note={Available at
  \url{https://www.math.lsu.edu/~sengupta/7360f09/DynkinPiLambda.pdf}. Last
  accessed on 2026-07-06},
}

\bib{ShoreJohnson1980}{article}{
      author={Shore, John~E.},
      author={Johnson, Rodney~W.},
       title={Axiomatic derivation of the principle of maximum entropy and the
  principle of minimum cross-entropy},
        date={1980},
     journal={IEEE Trans. Inf. Theory},
      volume={26},
      number={1},
       pages={26\ndash 37},
         url={https://doi.org/10.1109/TIT.1980.1056144},
}

\bib{SNREG23}{article}{
      author={Song, Minjeong},
      author={Narasimhachar, Varun},
      author={Regula, Bartosz},
      author={Elliott, Thomas~J.},
      author={Gu, Mile},
       title={Causal classification of spatiotemporal quantum correlations},
        date={2024Sep},
     journal={Phys. Rev. Lett.},
      volume={133},
       pages={110202},
      eprint={2306.09336},
         url={https://link.aps.org/doi/10.1103/PhysRevLett.133.110202},
}

\bib{SongParzygnat2025}{article}{
      author={Song, Minjeong},
      author={Parzygnat, Arthur~J.},
       title={Bipartite quantum states admitting a causal explanation},
        date={2025},
     journal={AVS Quantum Sci.},
      volume={7},
       pages={045002},
      eprint={2507.14278},
         url={https://doi.org/10.1116/5.0303300},
}

\bib{Sorce2023}{article}{
      author={Sorce, Jonathan},
       title={An intuitive construction of modular flow},
        date={2023-12},
        ISSN={1029-8479},
     journal={J. High Energ. Phys.},
      volume={2023},
      number={12},
       pages={79},
      eprint={2309.16766},
         url={http://dx.doi.org/10.1007/JHEP12(2023)079},
}

\bib{Sorce2024}{article}{
      author={Sorce, Jonathan},
       title={A short proof of {T}omita's theorem},
        date={2024},
        ISSN={0022-1236},
     journal={J. Func. Anal.},
      volume={286},
      number={12},
       pages={110420},
      eprint={2309.16762},
         url={http://dx.doi.org/10.1016/j.jfa.2024.110420},
}

\bib{St55}{article}{
      author={Stinespring, W.~Forrest},
       title={Positive functions on {$C^*$}-algebras},
        date={1955},
     journal={Proc. Am. Math. Soc.},
      volume={6},
      number={2},
       pages={211\ndash 216},
         url={https://doi.org/10.2307/2032342},
}

\bib{StratilaModular2020}{book}{
      author={Str{\u{a}}til{\u{a}}, {\c{S}}erban~Valentin},
       title={Modular theory in operator algebras},
     edition={2nd revised edition},
      series={Camb.-IISc Ser.},
   publisher={Cambridge University Press},
     address={Cambridge, UK},
        date={2020},
        ISBN={978-1-108-48960-7; 978-1-108-77903-6},
         url={https://doi.org/10.1017/9781108489607},
}

\bib{StratilaZsido2019}{book}{
      author={Str{\u{a}}til{\u{a}}, {\c{S}}erban~Valentin},
      author={Zsid{\'o}, L{\'a}szl{\'o}},
       title={Lectures on von {Neumann} algebras},
     edition={2nd edition},
      series={Camb.-IISc Ser.},
   publisher={Cambridge University Press},
        date={2019},
        ISBN={978-1-108-49684-1; 978-1-108-65497-5},
         url={https://doi.org/10.1017/9781108654975},
}

\bib{Takesaki1970}{book}{
      author={Takesaki, Masamichi},
       title={Tomita's theory of modular {Hilbert} algebras and its
  applications},
      series={Lect. Notes Math.},
   publisher={Springer, Berlin, Heidelberg},
        date={1970},
      volume={128},
         url={https://doi.org/10.1007/BFb0065832},
}

\bib{Takesaki1972conditional}{article}{
      author={Takesaki, Masamichi},
       title={Conditional expectations in von~{N}eumann algebras},
        date={1972},
     journal={J. Funct. Anal.},
      volume={9},
      number={3},
       pages={306\ndash 321},
         url={https://doi.org/10.1016/0022-1236(72)90004-3},
}

\bib{Takesaki1973CrossProduct}{article}{
      author={Takesaki, Masamichi},
       title={Duality for crossed products and the structure of von {N}eumann
  algebras of type {III}},
        date={1973},
     journal={Acta Math.},
      volume={131},
      number={1},
       pages={249\ndash 310},
         url={https://doi.org/10.1007/BF02392040},
}

\bib{Tak1}{book}{
      author={Takesaki, Masamichi},
       title={Theory of operator algebras {I}},
   publisher={Springer},
     address={New York, NY},
        date={1979},
      volume={124},
        ISBN={3-540-42248-X},
         url={https://doi.org/10.1007/978-1-4612-6188-9},
}

\bib{Tak2}{book}{
      author={Takesaki, Masamichi},
       title={Theory of operator algebras {II}},
      series={Encycl. Math. Sci.},
   publisher={Springer},
     address={Berlin, Heidelberg},
        date={2003},
      volume={125},
        ISBN={3-540-42914-X},
         url={https://doi.org/10.1007/978-3-662-10451-4},
}

\bib{TKRWV10}{article}{
      author={Temme, Kristan},
      author={Kastoryano, Michael~J.},
      author={Ruskai, Mary~Beth},
      author={Wolf, Michael~M.},
      author={Verstraete, Frank},
       title={The $\chi^2$-divergence and mixing times of quantum {M}arkov
  processes},
        date={201012},
        ISSN={0022-2488},
     journal={J. Math. Phys.},
      volume={51},
      number={12},
       pages={122201},
      eprint={1005.2358},
         url={https://doi.org/10.1063/1.3511335},
}

\bib{Teschl2014}{book}{
      author={Teschl, Gerald},
       title={Mathematical methods in quantum mechanics {With} applications to
  {Schr{\"o}dinger} operators.},
     edition={2nd ed.},
      series={Grad. Stud. Math.},
   publisher={American Mathematical Society},
     address={Providence, RI},
        date={2014},
      volume={157},
        ISBN={978-1-4704-1704-8},
         url={https://doi.org/10.1090/gsm/157},
}

\bib{TiFuWu2026}{misc}{
      author={Ting, Ouyang},
      author={Fullwood, James},
      author={Wu, Zhen},
       title={Operational time-reversal symmetry for unital qubit channels},
        date={2026},
         url={https://arxiv.org/abs/2605.10375},
        note={Available at
  \href{https://arxiv.org/abs/2605.10375}{arXiv:2605.10375}},
}

\bib{Topping1971}{book}{
      author={Topping, David~M.},
       title={Lectures on von {N}eumann algebras},
   publisher={Van Nostrand Reinhold Company},
        date={1971},
      volume={36},
}

\bib{Ts22}{article}{
      author={Tsang, Mankei},
       title={Generalized conditional expectations for quantum retrodiction and
  smoothing},
        date={2022},
     journal={Phys. Rev. A},
      volume={105},
       pages={042213},
      eprint={1912.02711},
         url={https://link.aps.org/doi/10.1103/PhysRevA.105.042213},
}

\bib{Ts22b}{article}{
      author={Tsang, Mankei},
       title={Operational meaning of a generalized conditional expectation in
  quantum metrology},
        date={2023-11},
        ISSN={2521-327X},
     journal={{Quantum}},
      volume={7},
       pages={1162},
      eprint={2212.13162},
         url={https://doi.org/10.22331/q-2023-11-03-1162},
}

\bib{Um54}{article}{
      author={Umegaki, Hisaharu},
       title={Conditional expectation in an operator algebra},
        date={1954},
        ISSN={0040-8735},
     journal={T\^{o}hoku Math. J. (2)},
      volume={6},
       pages={177\ndash 181},
         url={https://doi.org/10.2748/tmj/1178245177},
}

\bib{ValdiviaMera2025}{article}{
      author={Valdivia-Mera, Gustavo},
       title={On the {U}nruh effect and the thermofield double state},
        date={2025},
     journal={Int. J. Mod. Phys. D},
      volume={34},
      number={07n08},
       pages={2530002},
      eprint={2001.09869},
         url={https://doi.org/10.1142/S0218271825300022},
}

\bib{Weidmann1980}{book}{
      author={Weidmann, Joachim},
       title={Linear operators in {Hilbert} spaces},
      series={Grad. Texts Math.},
   publisher={Springer},
     address={New York, NY},
        date={1980},
      volume={68},
         url={https://doi.org/10.1007/978-1-4612-6027-1},
}

\bib{Weinberg1995QFT1}{book}{
      author={Weinberg, Steven},
       title={The quantum theory of fields},
    subtitle={Volume {I}: Foundations},
   publisher={Cambridge University Press},
     address={Cambridge, UK},
        date={1995},
        ISBN={9780521550017},
         url={https://doi.org/10.1017/CBO9781139644167},
}

\bib{Wilde15}{article}{
      author={Wilde, Mark~M.},
       title={Recoverability in quantum information theory},
        date={2015},
        ISSN={1471-2946},
     journal={Proc. R. Soc. A},
      volume={471},
      number={2182},
       pages={20150338},
      eprint={1505.04661},
         url={http://dx.doi.org/10.1098/rspa.2015.0338},
}

\bib{WildeQIT16}{book}{
      author={Wilde, Mark~M.},
       title={Quantum information theory},
     edition={2},
   publisher={Cambridge University Press},
        date={2016},
         url={http://dx.doi.org/10.1017/9781316809976.001},
}

\bib{Willard2004}{book}{
      author={Willard, Stephen},
       title={General topology},
   publisher={Dover Publications},
     address={Mineola, NY},
        date={2004},
        ISBN={9780486434797},
        note={Reprint of the 1970 Addison--Wesley edition},
}

\bib{Wirth2026}{misc}{
      author={Wirth, Melchior},
       title={The {KMS} and {GNS} spectral gap of quantum {M}arkov semigroups},
        date={2026},
         url={https://arxiv.org/abs/2604.21630},
        note={Available at
  \href{https://arxiv.org/abs/2604.21630}{arXiv:2604.21630}},
}

\bib{Witten18}{article}{
      author={Witten, Edward},
       title={{APS} medal for exceptional achievement in research: Invited
  article on entanglement properties of quantum field theory},
        date={2018},
     journal={Rev. Mod. Phys.},
      volume={90},
       pages={045003},
      eprint={1803.04993},
         url={https://link.aps.org/doi/10.1103/RevModPhys.90.045003},
}

\bib{Witten2022CrossedProduct}{article}{
      author={Witten, Edward},
       title={Gravity and the crossed product},
        date={2022},
     journal={J. High Energy Phys.},
      volume={2022},
       pages={8},
      eprint={2112.12828},
         url={https://doi.org/10.1007/JHEP10(2022)008},
}

\bib{Woronowicz1974}{article}{
      author={Woronowicz, Stanis{\l}aw~L.},
       title={Selfpolar forms and their applications to the {{\(C^*\)}}-algebra
  theory},
        date={1974},
        ISSN={0034-4877},
     journal={Rep. Math. Phys.},
      volume={6},
       pages={487\ndash 495},
         url={https://doi.org/10.1016/S0034-4877(74)80012-1},
}

\bib{Yanofsky2024Monoidal}{book}{
      author={Yanofsky, Noson~S.},
       title={Monoidal category theory},
    subtitle={Unifying concepts in mathematics, physics, and computing},
   publisher={The MIT Press},
     address={Cambridge, MA},
        date={2024},
        ISBN={978-0-262-04939-9},
  url={https://mitpress.mit.edu/9780262049399/monoidal-category-theory/},
}

\bib{Yngvason2005}{article}{
      author={Yngvason, Jakob},
       title={The role of type {III} factors in quantum field theory},
        date={2005},
        ISSN={0034-4877},
     journal={Rep. Math. Phys.},
      volume={55},
      number={1},
       pages={135\ndash 147},
      eprint={math-ph/0411058},
         url={https://doi.org/10.1016/S0034-4877(05)80009-6},
}

\bib{Zellner1988}{article}{
      author={Zellner, Arnold},
       title={Optimal information processing and {Bayes}'s theorem},
        date={1988},
     journal={Am. Stat.},
      volume={42},
      number={4},
       pages={278\ndash 280},
         url={https://doi.org/10.1080/00031305.1988.10475585},
}

\end{biblist}
\end{bibdiv}

\Addresses

\end{document}